\documentclass[12pt]{amsart}

\usepackage[left=2cm, right=2cm, top=3cm, bottom=3cm]{geometry}
\usepackage{hyperref}
\usepackage{amsmath, amsthm, amsfonts, amssymb}
\usepackage{mathtools}
\usepackage{mathrsfs}
\usepackage{longtable}
\usepackage{booktabs}
\usepackage{upgreek}

\hypersetup{
	colorlinks=true,
	pdftitle={Correspondence among congruence families for generalized Frobenius partitions via modular permutations},
	pdfauthor={Rong Chen, and Xiao-Jie Zhu},
	pdfsubject={Mathematics},
	pdfkeywords={Partition congruences, Frobenius partitions, Modular functions, Weil representations, Jacobi forms}
}

\numberwithin{equation}{section}

\theoremstyle{plain}
\newtheorem{thm}{Theorem}[section]
\newtheorem*{thm*}{Theorem}
\newtheorem*{thm1.1*}{Theorem \ref{thm:fkbetaSLZ}}
\newtheorem*{thm1.3*}{Theorem \ref{r:thecpsimodfun}}
\newtheorem{coro}[thm]{Corollary}
\newtheorem{conj}[thm]{Conjecture}
\newtheorem{oq}[thm]{Question}
\newtheorem{prop}[thm]{Proposition}

\newtheorem{lemm}[thm]{Lemma}
\newtheorem{lemma}[thm]{Lemma}

\theoremstyle{definition}
\newtheorem{deff}[thm]{Definition}
\newtheorem{examp}[thm]{Example}

\theoremstyle{remark}
\newtheorem{rema}[thm]{Remark}
\newtheorem{algo}[thm]{Algorithm}

\newcommand\twoscript[2]{\substack{{#1} \\ {#2}}}

\newcommand\beq{\begin{equation}}
\newcommand\eeq{\end{equation}}
\newcommand\rmi{\mathrm{i}}
\newcommand\rme{\mathrm{e}}
\newcommand\legendre[2]{\genfrac{(}{)}{}{}{#1}{#2}}
\newcommand\tbtmat[4]{\left(\begin{smallmatrix}{#1} & {#2} \\ {#3} & {#4}\end{smallmatrix}\right)}
\newcommand\tbtMat[4]{\left(\begin{matrix}{#1} & {#2} \\ {#3} & {#4}\end{matrix}\right)}
\newcommand*\abs[1]{\lvert#1\rvert}
\newcommand\etp[1]{\mathfrak{e}\left(#1\right)}

\newcommand\sgn[1]{\mathop{\mathrm{sgn}}\left(#1\right)}

\newcommand\lcm{\mathop{\mathrm{lcm}}}

\newcommand\numZ{\mathbb{Z}}
\newcommand\numQ{\mathbb{Q}}
\newcommand\numR{\mathbb{R}}
\newcommand\numC{\mathbb{C}}

\newcommand\numgeq[2]{\mathbb{#1}_{\geq #2}}

\newcommand\slZ{\mathrm{SL}_2(\mathbb{Z})}

\newcommand\slR{\mathrm{SL}_2(\mathbb{R})}

\newcommand\glpR{\mathrm{GL}_2^{+}(\mathbb{R})}
\newcommand\glptR{\widetilde{\mathrm{GL}_2^{+}(\mathbb{R})}}
\newcommand\sltR{\widetilde{\mathrm{SL}_2(\mathbb{R})}}
\newcommand\sltZ{\widetilde{\mathrm{SL}_2(\mathbb{Z})}}

\newcommand\glnC[1]{\mathrm{GL}_{#1}(\mathbb{C})}

\newcommand\uhp{\mathfrak{H}}

\allowdisplaybreaks

\makeatletter
\@namedef{subjclassname@2020}{\textup{2020} Mathematics Subject Classification}
\makeatother

\begin{document}

\title[Congruences for generalized Frobenius partitions]{Correspondence among congruence families for generalized Frobenius partitions via modular permutations}

\author{Rong Chen}
\address{Department of Mathematics, Shanghai Normal University, Shanghai, People's Republic of China}
\email{rchen@shnu.edu.cn}

\author{Xiao-Jie Zhu}
\address{School of Mathematical Sciences,
Key Laboratory of MEA(Ministry of Education) \& Shanghai Key Laboratory of PMMP,
East China Normal University}
\email{zhuxiaojiemath@outlook.com}



\begin{abstract}
In 2024, Garvan, Sellers and Smoot discovered a remarkable symmetry in the families of congruences for generalized Frobenius partitions $c\psi_{2,0}$ and $c\psi_{2,1}$. They also emphasized that the considerations for the general case of $c\psi_{k,\beta}$ are important for future work. In this paper, for each $k$ we construct a vector-valued modular form for the generating functions of $c\psi_{k,\beta}$, and determine an equivalence relation among all $\beta$. Within each equivalence class, we can identify modular transformations relating the congruences of one $c\psi_{k,\beta}$ to that of another $c\psi_{k,\beta'}$. Furthermore, correspondences between different equivalence classes can also be obtained through linear combinations of modular transformations. As an example, with the aid of these correspondences, we prove a family of congruences of $c\phi_{3}$, the Andrews' $3$-colored Frobenius partition.
\end{abstract}

\thanks{The first author is supported by the National Key R\&D Program of China (Grant No. 2024YFA1014500) and the National Natural Science Foundation of China (Grant No. 12401438). The second author is supported in part by Science and Technology Commission of Shanghai Municipality (No. 22DZ2229014).}

\maketitle

\section{Introduction}

In his 1984 AMS Memoir, Andrews \cite{An84} introduced two families of partition functions. Among them one is $c\phi_k(n)$, the generalized $k$-colored Frobenius partition function. He proved a variety of Ramanujan-type congruences satisfied by these $c\phi_k(n)$ such as
\begin{align}
\label{r:cphi2mod5}
c\phi_2(5n+3)\equiv 0 \pmod{5}.
\end{align}
In 1994, Sellers \cite{Se94} conjectured that \eqref{r:cphi2mod5} has an analog modulo all powers of $5$:
\begin{align}
\label{r:cphi2mod5a}
c\phi_2(5^\alpha n+\lambda_\alpha)\equiv 0 \pmod{5^\alpha},
\end{align}
where $\lambda_\alpha$ satisfies that $12\lambda_\alpha\equiv 1 \pmod{5^\alpha}$. Eighteen years later, a complete proof of this conjecture was found by Paule and Radu \cite{PR12}. Their proof follows the method of Watson \cite{Wa38}, and provides a new perspective on proving Ramanujan-type congruences by using the technique of modular functions.

Recently, Jiang, Rolen and Woodbury \cite{JRW22} considered $(k,a)$-colored Frobenius partition functions $c\psi_{k,a}(n)$, which are natural generalizations of Andrews' $c\phi_k$ since $c\psi_{k,k/2}(n)=c\phi_k(n)$. They proved that for $a \in \frac{k}{2}+\numZ$ the generating function of $c\psi_{k,a}(n)$,
$$C\Psi_{k,a}(q):=\sum_{n=0}^{\infty}c\psi_{k,a}(n)q^n,$$
is the $\zeta^a$ coefficient of
\beq
\label{r:fkdef}
F_k(\tau,z)=\left(\frac{-\vartheta(\tau,z+\frac{1}{2})}{q^{\frac{1}{12}}\eta(\tau)}\right)^k,
\eeq
where $q:=\etp{\tau}:=\rme^{2\uppi\rmi\tau}$, $\zeta:=\etp{z}$, $\tau\in\uhp$ (the upper half plane), $z\in\numC$,
\begin{equation}
\label{eq:eta}
\eta(\tau):=q^{\frac{1}{24}}\prod_{j=1}^\infty (1-q^j),
\end{equation}
and (see Example \ref{examp:classicalTheta})
$$
\vartheta(\tau,z):=\sum_{n\in \frac{1}{2}+\mathbb{Z}}\etp{\frac{1}{2}n^2\tau+n\left(z+\frac{1}{2}\right)}.
$$

Garvan, Sellers and Smoot \cite{GSS24} found that the property of $c\psi_{2,0}$ modulo powers of $5$ is similar to \eqref{r:cphi2mod5a}. They obtained
\begin{align}
\label{r:cpsi2mod5a}
c\psi_{2,\beta}(n)\equiv 0 \pmod{5^\alpha},
\end{align}
where $\beta\in \{0,1\}$ and $3\cdot 2^{\beta+1}n\equiv (-1)^{\beta+1} \pmod{5^\alpha}$. (For $\beta=1$, this is just \eqref{r:cphi2mod5a}.) Furthermore, they realized that the two families of congruences \eqref{r:cpsi2mod5a} bear a strange symmetry between each other via the application of certain Atkin-Lehner involution to their generating functions. This profound observation led them to recognize that one family of congruences may imply other families via certain modular transformations. At the end of their paper, they emphasized five considerations for future work. The first three read as follows:

\begin{enumerate}
\item Is our system of congruences for $c\psi_{2,\beta}(n)$ complete for $\beta = 0, 1$? That is to say, is this pair part of a broader triple or quadruple of families all closely related via their proofs or their behavior through some associated Atkin--Lehner involutions? Are there perhaps other mappings similar to $\sigma$ which would relate these families to others? Indeed, we believe that this specific pair of congruence families is not so closely related to any others in the same way that they are to each other---but we cannot be certain.
\item What are the implications of identifying different congruence families with one another via homomorphisms which fix certain sub-module function spaces?
\item Do similar systems of congruence families exist for more general $c\psi_{k,\beta}(n)$, especially by holding a given $k$ fixed and varying $\beta$?
\end{enumerate}

The main purpose of this paper is to address the problems mentioned above.

\subsection{Vector-valued modular transformation laws of $C\Psi_{k,\beta}$}
In the viewpoint of modular transformations, the functions $c\psi_{2,\beta}(n)$ where $\beta = 0, 1$ are complete since their generating functions constitute a weakly holomorphic vector-valued modular form with respect to $\slZ$. This remains true for general $c\psi_{k,\beta}(n)$. Let $f_{k,\beta}(\tau)=q^{\frac{k}{12}-\frac{\beta^2}{2k}}C\Psi_{k,\beta}(q)$ and define
\begin{equation}
\label{eq:mathfrakBk}
\mathfrak{B}_k=\begin{dcases}
\left\{0,1,2,\dots,\frac{k}{2}\right\} &\text{ if } 2\mid k,\\
\left\{\frac{1}{2},\frac{3}{2},\frac{5}{2}\dots,\frac{k}{2}\right\} &\text{ if } 2\nmid k.\\
\end{dcases}
\end{equation}
Our first result establishes the vector-valued modular transformation laws of $f_{k,\beta}(\tau)$ by giving the explicit coefficients under the actions of $\tbtmat{1}{1}{0}{1}$ and $\tbtmat{0}{-1}{1}{0}$:

\begin{thm}
\label{thm:fkbetaSLZ}
Let $k\in\numgeq{Z}{1}$, $\beta\in\mathfrak{B}_k$. Then we have
\begin{align*}
f_{k,\beta}\vert_{-\frac{1}{2}}\tbtmat{1}{1}{0}{1}&=\etp{\frac{k}{12}-\frac{\beta^2}{2k}}f_{k,\beta},\\
f_{k,\beta}\vert_{-\frac{1}{2}}\tbtmat{0}{-1}{1}{0}&=\sum_{\beta'\in\mathfrak{B}_k}s_{\beta,\beta'}^{(k)}\cdot f_{k,\beta'},
\end{align*}
where
\begin{equation*}
s_{\beta,\beta'}^{(k)}=\mu_{\beta'/k}\cdot\rmi^{2\beta}\etp{\frac{2k+1}{8}}\frac{2}{\sqrt{k}}\cdot\cos\left(\frac{2\uppi\beta'(\beta+\tfrac{k}{2})}{k}\right),
\end{equation*}
$\mu_t=1/2$ if $t=0$ or $1/2$ and $\mu_t=1$ else.
\end{thm}
The notation $f\vert_{-\frac{1}{2}}\tbtmat{a}{b}{c}{d}$ means the weight $-\tfrac{1}{2}$ modular transformation, that is, $f\vert_{-\frac{1}{2}}\tbtmat{a}{b}{c}{d}(\tau):=(c\tau+d)^{1/2}f\left(\frac{a\tau+b}{c\tau+d}\right)$ where $(c\tau+d)^{1/2}$ is the principal branch (the argument belongs to $(-\tfrac{\uppi}{2},\tfrac{\uppi}{2}]$). Since $\slZ$ is generated by $\tbtmat{1}{1}{0}{1}$ and $\tbtmat{0}{-1}{1}{0}$ we can calculate $f_{k,\beta}\vert_{-\frac{1}{2}}\tbtmat{a}{b}{c}{d}$ for any $\tbtmat{a}{b}{c}{d}\in\slZ$ by first expanding $\tbtmat{a}{b}{c}{d}$ as a product or quotient of $\tbtmat{1}{1}{0}{1}$ and $\tbtmat{0}{-1}{1}{0}$ and then applying repeatedly the two formulas in Theorem \ref{thm:fkbetaSLZ}. Therefore, $(f_{k,\beta})_{\beta\in\mathfrak{B}_k}$ is a $\numC^{\lfloor k/2\rfloor+1}$-valued weakly holomorphic modular form with respect to $\slZ$ of weight $-1/2$. For the proof, see Section \ref{r:sec3}.

\begin{examp}
\label{examp:k234}
For $k=2,3,4$, we have
\begin{align*}
\left.\begin{pmatrix}
f_{2,0}\\
f_{2,1}
\end{pmatrix}\right\vert_{-\frac{1}{2}}\tbtMat{1}{1}{0}{1}&=
\begin{pmatrix}
\rme^{\uppi\rmi/3} & 0\\
0 & \rme^{-\uppi\rmi/6}
\end{pmatrix}\cdot\begin{pmatrix}
f_{2,0}\\
f_{2,1}
\end{pmatrix},\\
\left.\begin{pmatrix}
f_{2,0}\\
f_{2,1}
\end{pmatrix}\right\vert_{-\frac{1}{2}}\tbtMat{0}{-1}{1}{0}&=\frac{\rme^{\uppi\rmi/4}}{\sqrt{2}}
\begin{pmatrix}
-1 & 1\\
1 & 1
\end{pmatrix}\cdot\begin{pmatrix}
f_{2,0}\\
f_{2,1}
\end{pmatrix},\\
\left.\begin{pmatrix}
f_{3,1/2}\\
f_{3,3/2}
\end{pmatrix}\right\vert_{-\frac{1}{2}}\tbtMat{1}{1}{0}{1}&=
\begin{pmatrix}
\rme^{5\uppi\rmi/12} & 0\\
0 & \rme^{-\uppi\rmi/4}
\end{pmatrix}\cdot\begin{pmatrix}
f_{3,1/2}\\
f_{3,3/2}
\end{pmatrix},\\
\left.\begin{pmatrix}
f_{3,1/2}\\
f_{3,3/2}
\end{pmatrix}\right\vert_{-\frac{1}{2}}\tbtMat{0}{-1}{1}{0}&=\frac{\rme^{\uppi\rmi/4}}{\sqrt{3}}
\begin{pmatrix}
-1 & 1\\
2 & 1
\end{pmatrix}\cdot\begin{pmatrix}
f_{3,1/2}\\
f_{3,3/2}
\end{pmatrix},\\
\left.\begin{pmatrix}
f_{4,0}\\
f_{4,1}\\
f_{4,2}
\end{pmatrix}\right\vert_{-\frac{1}{2}}\tbtMat{1}{1}{0}{1}&=
\begin{pmatrix}
\rme^{2\uppi\rmi/3} & 0 & 0\\
0 & \rme^{5\uppi\rmi/12} & 0\\
0 & 0 & \rme^{-\uppi\rmi/3}
\end{pmatrix}\cdot\begin{pmatrix}
f_{4,0}\\
f_{4,1}\\
f_{4,2}
\end{pmatrix},\\
\left.\begin{pmatrix}
f_{4,0}\\
f_{4,1}\\
f_{4,2}
\end{pmatrix}\right\vert_{-\frac{1}{2}}\tbtMat{0}{-1}{1}{0}&=\frac{\rme^{\uppi\rmi/4}}{2}
\begin{pmatrix}
1 & -2 & 1\\
-1 & 0 & 1\\
1 & 2 & 1
\end{pmatrix}\cdot\begin{pmatrix}
f_{4,0}\\
f_{4,1}\\
f_{4,2}
\end{pmatrix}.
\end{align*}
\end{examp}
For some considerations on representation theory of the metaplectic group and the modular group, see Section \ref{subsec:representations}.

\subsection{A unified tool to relate congruence families}
The modular transformation $f_{k,\beta}\vert_{-\frac{1}{2}}\gamma$, when $\gamma$ belongs to the subgroup $\Gamma_0(k)$ (see \eqref{eq:Gamma0N}), results in a single term $f_{k,\beta'}$ up to a constant factor, and possibly with a different $\beta'$. This phenomenon is studied in detail in Section \ref{r:sec4} where the main result is Theorem \ref{thm:fkbetaGamma0k}. The following theorem, which is a direct consequence of Theorem \ref{thm:fkbetaGamma0k}, demonstrates how modular transformations map the modular function associated with one generalized Frobenius partition to the function associated with another, and it serves as a unified tool to identify different congruence families with one another:

\begin{thm}
\label{r:thecpsimodfun}
Let $p\geq 5$ be prime and $\beta,\beta'\in \mathfrak{B}_k$ with $(2\beta,k)=(2\beta',k)$. Set
\begin{equation*}
r=\frac{k}{(k,(2\beta)^2(p^2-1)/8)},\quad r_e=\lcm(2,r).
\end{equation*}
Then there exists $\gamma\in \Gamma_0^0(p^2k,r)$ (see \eqref{eq:Gamma00NM}) such that
\begin{equation*}
\frac{f_{k,\beta}(\tau)}{f_{k,\beta}(p^2\tau)}\bigg| \gamma=\frac{f_{k,\beta'}(\tau)}{f_{k,\beta'}(p^2\tau)}.
\end{equation*}
Moreover, if $(r,p)=1$, then we can further require that
\begin{equation*}
U_p'(L\vert \gamma)=U_p'(L)\vert \gamma
\end{equation*}
for all modular functions $L$ with respect to $\Gamma_1^*(2p^2k,r)$ (see \eqref{eq:Gamma1starNM}) where
\begin{equation*}
U_p'(L):=\frac{1}{p}\sum_{x=0}^{p-1}L\left(\frac{\tau+r_ex}{p}\right).
\end{equation*}
\end{thm}
In the above and throughout the paper, $(m,n)$ denotes the greatest common divisor, $\lcm(m,n)$ the least common multiple and $f\vert\tbtmat{a}{b}{c}{d}(\tau):=f\left(\frac{a\tau+b}{c\tau+d}\right)$. (However, the notation $f\vert\gamma$ has a different meaning in Section \ref{subsec:fromf312tof332}.)

For any $k,p,\beta$, and $\beta'$ in Theorem \ref{r:thecpsimodfun} we can calculate $\gamma$ explicitly. For example, for $k=2$, $p=5$, $\beta=1$ and $\beta'=0$ the modular transformation
$$
\gamma=\begin{pmatrix}
27 & 7\\
50 & 13
\end{pmatrix}
$$
maps $c\psi_{2,1}(n)$ to $c\psi_{2,0}(n)$, which has been found in \cite{GSS24} in a different manner.

Theorem \ref{r:thecpsimodfun} also answers that for any fixed $k$, similar congruence families do exist for $c\psi_{k,\beta}(n)$ with $\beta$ varying, at least when one of the congruence families is related to some explicit modular functions. Indeed, the condition $(2\beta,k)=(2\beta',k)$ determines an equivalence relation for a given $k$. Theorem \ref{r:thecpsimodfun} then guarantees that all $c\psi_{k,\beta}(n)$ are related to each other within the same equivalence class. See Theorem \ref{thm:thklkl} for details.

For the proof of Theorem \ref{r:thecpsimodfun} and the construction of $\gamma$, see Section \ref{subsec:proof_of_theorem_thecpsimodfun}.

As $f_{k,k/2}$ is the generating function of Andrews' $k$-colored Frobenius partition, the case $\beta'=k/2$ is important. If in addition $k\equiv2\pmod{4}$ and $\beta=0$, we have a perfect choice of $\gamma$, which is related to certain Atkin-Lehner involution and is a direct generalization of \cite[Eq. (3.37)]{GSS24} where $k=2$ is considered. See Proposition \ref{prop:whenk24gamma}.

We note that there is a precedent in \cite{CCG24} for using vector-valued modular transformations to correspond two families of congruences.

\subsection{A new approach for proving Ramanujan-type congruences}
Even if $c\psi_{k,\beta}(n)$ and $c\psi_{k,\beta'}(n)$ fall into different equivalence classes, they may correspond to each other under some linear combination of modular transformations. This seems to provide a new approach for proving Ramanujan-type congruences. As an example, we consider the case $k=3$.

\begin{thm}
\label{r:themaincong}
For integers $\alpha>0$ and $n\geq 0$, we have
\begin{align}
\label{r:c31cong}
c\psi_{3,1/2}(5^{\alpha}n+\delta_{\alpha})\equiv 0 \pmod{5^{\lfloor\frac{\alpha}{2}\rfloor}},
\end{align}
\begin{align}
\label{r:c33cong}
c\psi_{3,3/2}(5^{\alpha}n+\lambda_{\alpha})\equiv 0 \pmod{5^{\lfloor\frac{\alpha}{2}\rfloor}},
\end{align}
where
$$
\delta_{2\alpha}=\delta_{2\alpha+1}=\frac{5^{2\alpha+1}-5}{24},
$$
\begin{align*}
&\lambda_{2\alpha-1}=\frac{3\cdot 5^{2\alpha-1}+1}{8},\\
&\lambda_{2\alpha}=\frac{7\cdot 5^{2\alpha}+1}{8}.
\end{align*}
\end{thm}

One surprising aspect is that proving the congruences \eqref{r:c31cong} is typically straightforward if one follows the method outlined in \cite{PR12}, whereas there is a greater challenge to prove \eqref{r:c33cong} in the traditional way. We find that the generating function of $c\psi_{3,1/2}(n)$ is
\begin{equation}
\label{eq:cpsi31genfun}
C\Psi_{3,1/2}(q)=3q^{-\frac{5}{24}}\frac{\eta(3\tau)^3}{\eta(\tau)^4},
\end{equation}
which shows that $C\Psi_{3,1/2}(q)/q^5 C\Psi_{3,1/2}(q^{25})$ is a holomorphic modular function on $\Gamma_0(75)$. Then one can prove \eqref{r:c31cong} by working with the modular curve $X_0(15)$ which has genus 1 and cusp count 4. However, we find that the generating function of $c\psi_{3,3/2}(n)$ is
\begin{align}
\label{r:cpsi33genfun}
C\Psi_{3,3/2}(q)=q^{\frac{1}{8}}\left(\frac{1}{\eta(3\tau)}+9\frac{\eta(9\tau)^3}{\eta(3\tau)\eta(\tau)^3}\right).
\end{align}
The issue is that $q^3 C\Psi_{3,3/2}(q)/ C\Psi_{3,3/2}(q^{25})$ has some poles on the upper half plane. Another common approach to deal with \eqref{r:cpsi33genfun} is to shift the focus to $q^{-1/8}\eta(75\tau)C\Psi_{3,3/2}(q)$, which is a holomorphic modular function. However, in this viewpoint we need to work with the modular curve $X_0(45)$ which has genus 3 and cusp count 8. To the best of our knowledge, there appears to be no literature addressing specifically the proof of Ramanujan-type congruences by working with such a complicated modular curve.

\subsection{Organization of the paper and notations}
In Section \ref{r:sec2}, we introduce some tools (Weil representations, Jacobi forms and theta decompositions) for proving Theorems \ref{thm:fkbetaSLZ}, \ref{thm:fkbetaGamma0k} and Proposition \ref{prop:f3betaGamma026}. Theorem \ref{thm:fkbetaSLZ} is then proved in Section \ref{r:sec3}. In Section \ref{r:sec4} we show that modular transformations in $\Gamma_0(k)$ map $f_{k,\beta}$ to a single term $f_{k,\beta'}$ and determine the transformation rule explicitly. Some related results are also presented. In Section \ref{r:sec5}, we prove Theorem \ref{r:thecpsimodfun} and explain how this theorem can be applied to extend the result of Garvan, Sellers and Smoot to more general $c\psi_{k,\beta}$ for a fixed $k$ and for $\beta$ belonging to a fixed equivalence class. In Section \ref{r:sec6}, to show that the correspondence between two congruence families may also occur across different equivalence classes, we prove Theorem \ref{r:themaincong} by establishing a correspondence between \eqref{r:c31cong} and \eqref{r:c33cong}. Finally, we discuss some miscellaneous observations and future works in Section \ref{sec:misc}.

Throughout the paper, we will need some subgroups of $\slZ$. Let $N,M$ be nonnegative integers. Set
\begin{align}
\Gamma_0(N)&:=\{\tbtmat{a}{b}{c}{d}\in\slZ\colon c\equiv0\bmod{N}\},\label{eq:Gamma0N}\\
\Gamma_0^0(N,M)&:=\{\tbtmat{a}{b}{c}{d}\in\slZ\colon c\equiv0\bmod{N},\,b\equiv0\bmod{M}\},\label{eq:Gamma00NM}\\
\Gamma_1(N)&:=\{\tbtmat{a}{b}{c}{d}\in\slZ\colon c\equiv0\bmod{N},\,a\equiv d\equiv1\bmod{N}\},\label{eq:Gamma1N}\\
\Gamma_1^*(N,M)&:=\{\tbtmat{a}{b}{c}{d}\in\slZ\colon c\equiv0\bmod{N},\,a\equiv d\equiv1\bmod{N},\,b\equiv0\bmod{M}\},\label{eq:Gamma1starNM}\\
\Gamma_1^*(N)&:=\Gamma_1^*(N,24),\notag\\
\Gamma(N)&:=\{\tbtmat{a}{b}{c}{d}\in\slZ\colon c\equiv b\equiv0\bmod{N},\,a\equiv d\equiv1\bmod{N}\}.\notag
\end{align}

\section{Weil representations and Jacobi forms}
\label{r:sec2}

The content of this section, except some usual definitions, will be used only in the proofs of Theorems \ref{thm:fkbetaSLZ}, \ref{thm:fkbetaGamma0k} and Proposition \ref{prop:f3betaGamma026}.
\subsection{Lattices and Gauss sums}
Let $V$ be a real vector space of dimension $n\in\numgeq{Z}{1}$ equipped with a nondegenerate symmetric bilinear form $B$. We set
\begin{equation}
\label{eq:QfromB}
Q(v):=\frac{1}{2}B(v,v),\quad v\in V
\end{equation}
and call it the associated quadratic form.
A \emph{lattice} in $\underline{V}=(V, B)$ is a free $\numZ$-module of rank $n$ that spans $V$. The \emph{dual} of a lattice $L$, denoted by $L^\sharp$, is the module of vectors $v$ such that $B(w, v)\in\numZ$ for all $w\in L$. It turns out that $L^\sharp$ is still a lattice. By abuse of language, we call the pair $\underline{L}=(L, B)$ a lattice as well. It is called \emph{even integral} if $Q(v)\in\numZ$ for all $v\in L$, which implies that $L\subseteq L^\sharp$ (in this case, $L$ is called \emph{integral}).

The \emph{discriminant module} of an even integral lattice $L$ is the quotient group $L^\sharp/L$ which is a finite abelian group. There is a well-defined ``quadratic form'' $Q(v+L)=Q(v)+\numZ$ on $L^\sharp/L$ that makes it a finite quadratic module (cf. \cite[\S 2]{Str13}).

\begin{examp}
\label{examo:latticeZand2Z}
Let $k$ be a nonzero integer and $\underline{\numR}_k$ be the one-dimensional space $\numR$ equipped with the bilinear form $(x, y)\mapsto kxy$. Then both $\numZ$ and $2\numZ$ are integral lattices in $\underline{\numR}_k$. To indicate the underlying bilinear forms, we write $\underline{\numZ}_k$ and $\underline{2\numZ}_k$ instead. If $k$ is even then $\underline{\numZ}_k$ is even integral; otherwise if $k$ is odd then only $\underline{2\numZ}_k$ is even integral. It is immediate that $\underline{\numZ}_k^\sharp=\frac{1}{k}{\numZ}$ and $\underline{2\numZ}_k^\sharp=\frac{1}{2k}{\numZ}$; therefore, the discriminant modules $\underline{\numZ}_k^\sharp/\underline{\numZ}_k=\frac{1}{k}{\numZ}/\numZ$ and $\underline{2\numZ}_k^\sharp/\underline{2\numZ}_k=\frac{1}{2k}{\numZ}/2\numZ$ are both cyclic groups.
\end{examp}

\begin{deff}
\label{def:LkLksharp}
Let $k$ be a nonzero integer. We define $L_k:=\numZ$ if $k$ is even, and $L_k:=2\numZ$ if $k$ is odd. Moreover, set $L_k^\sharp:=\frac{1}{k}{\numZ}$ if $k$ is even, and $L_k^\sharp:=\frac{1}{2k}{\numZ}$ if $k$ is odd. Finally, set $D_k:=\underline{\numZ}_k^\sharp/\underline{\numZ}_k=\frac{1}{k}{\numZ}/\numZ$ if $k$ is even, and $D_k=\underline{2\numZ}_k^\sharp/\underline{2\numZ}_k=\frac{1}{2k}{\numZ}/2\numZ$ if $k$ is odd.
\end{deff}

The lattice we need to deal with $C\Psi_{k,\beta}$ ($k\in\numgeq{Z}{1}$) is exactly $L_k$.

Next we recall some facts about Gauss sums. The classical Gauss sum is defined by $G(n, m)=\sum_{x=0}^{\abs{m}-1}\etp{\frac{nx^2}{m}}$ for $n\in\numZ$ and $m\in\numZ\setminus\{0\}$. It is known that in the case $m>0$ and $(n, m)=1$ we have
\begin{equation}
\label{eq:claGaussSum}
G(n,m)=\begin{dcases}
\sqrt{m}\legendre{2n}{m}\etp{\frac{1-m}{8}}, &\text{if } 2\nmid m;\\
\sqrt{2m}\legendre{2m}{n}\etp{\frac{n}{8}}, &\text{if } m\equiv0\bmod{4};\\
0, &\text{if }m\equiv2\bmod{4}.
\end{dcases}
\end{equation}
This is a reformulation of classical results due to Gauss (cf. \cite[\S 1.5]{BEW98}). The notation $\legendre{\cdot}{\cdot}$ refers to the Kronecker-Jacobi symbol. For the reader's convenience we present its definition:
\begin{itemize}
\item $\legendre{m}{p}$ is the usual Legendre symbol if $p$ is an odd prime.
\item $\legendre{m}{2}$ equals $0$ if $2 \mid m$, and equals $(-1)^{(m^2-1)/8}$ if $2 \nmid m$.
\item $\legendre{m}{-1}$ equals $1$ if $m \geq 0$, and equals $-1$ otherwise.
\item $\legendre{m}{1}=1$ by convention.
\item $\legendre{m}{n}$ is defined to make it a complete multiplicative function of $n \in \numZ-\{0\}$.
\item $\legendre{m}{0}=0$ if $m \neq \pm1$, and $\legendre{\pm1}{0}=1$.
\end{itemize}
Below we shall use the known facts that the function $n\mapsto\legendre{m}{n}$ is $\abs{m}$-periodic if $m\equiv0,1\bmod{4}$; it is $\abs{4m}$-periodic if $m\equiv2\bmod{4}$ and that the function $m\mapsto\legendre{m}{n}$ is $n$-periodic if $n$ is odd and positive. For the proofs of these facts along with other basic properties, see \cite[Section 2.2.2]{Coh07}.

We introduce two variants. One is
\begin{equation}
\label{eq:gGauss1}
\mathfrak{g}_{k}(b,d;t)=\abs{d}^{-\frac{1}{2}}\sum_{v\in L_k/dL_k}\etp{\frac{bk(t+v)^2}{2d}}
\end{equation}
where $k\in\numZ\setminus\{0\}$, $b\in\numZ$, $d\in\numZ\setminus\{0\}$, $t\in L_k^\sharp$. (For the definitions of $L_k$ and $L_k^\sharp$, see Definition \ref{def:LkLksharp}.) Note that it is well-defined.

\begin{lemm}
\label{lemm:glbdt}
Suppose $a,b,c,d\in\numZ$ with $ad-bc=1$, $d\neq0$, and $bct\in L_k$. Then
\begin{equation*}
\mathfrak{g}_{k}(b,d;t)=\etp{\frac{a^2bdkt^2}{2}}\mathfrak{g}_{k}(b,d;0).
\end{equation*}
\end{lemm}
\begin{proof}
Since $ad-bc=1$ we have
\begin{equation}
\label{eq:glbdtProof1}
\mathfrak{g}_{k}(b,d;t)=\abs{d}^{-\frac{1}{2}}\sum_{v\in L_k/dL_k}\etp{\frac{bk(adt+(v-bct))^2}{2d}}.
\end{equation}
Since $bct\in L_k$ the map $L_k/dL_k\rightarrow L_k/dL_k$, $v\mapsto v-bct$ is well-defined and bijective. Applying the change of variables $w=v-bct$ in \eqref{eq:glbdtProof1} and then using the facts $(adt+w)^2=(adt)^2+2adtw+w^2$ and $\frac{bk\cdot 2adtw}{2d}\in\numZ$ we obtain the desired identity.
\end{proof}

\begin{lemm}
\label{lemm:gbdt}
If $k$ is an even positive integer, $\beta\in\numZ$ and $\tbtmat{a}{b}{c}{d}\in\slZ$ with $k\mid bc$, then
\begin{equation*}
\mathfrak{g}_{k}\left(b,d;\frac{\beta}{k}\right)=\legendre{\sgn{d}\cdot kb}{\abs{d}}\etp{\frac{1-\abs{d}}{8}}\etp{\frac{bd(a\beta)^2}{2k}}.
\end{equation*}
Moreover, if $k$ is an odd positive integer, $\beta\in\frac{1}{2}+\numZ$ and $\tbtmat{a}{b}{c}{d}\in\slZ$ with $4k\mid bc$, then the above formula holds as well.
\end{lemm}
\begin{proof}
We present the proof only for $k$ being odd since the even case is similar. Setting $t=\beta/k$ in Lemma \ref{lemm:glbdt} we find that, since $4k\mid bc$,
\begin{align}
\mathfrak{g}_{k}\left(b,d;\frac{\beta}{k}\right)&=\etp{\frac{bd(a\beta)^2}{2k}}\cdot\mathfrak{g}_{k}\left(b,d;0\right)\notag\\
&=\etp{\frac{bd(a\beta)^2}{2k}}\cdot\abs{d}^{-\frac{1}{2}} G\left(\sgn{d}\cdot2kb, \abs{d}\right).\label{eq:g2Zk}
\end{align}
Since $4\mid bc$ we have $2\nmid d$, and hence we can insert the first case of \eqref{eq:claGaussSum} into \eqref{eq:g2Zk} from which the desired formula follows.
\end{proof}

The other variant of Gauss sums we need is
\begin{equation}
\label{eq:GDdx}
\mathscr{G}_{k}(d,x)=\frac{1}{\sqrt{\abs{D_k}\cdot\abs{D_k[d]}}}\sum_{y\in D_k}\etp{dky^2/2+kxy},
\end{equation}
where $k\in\numZ\setminus\{0\}$, $d\in\numZ$, $x\in D_k$ (see Definition \ref{def:LkLksharp}), and where $D_k[d]$ is the kernel of the homomorphism $y\mapsto dy$ on $D_k$. See \cite[page 514]{Str13} for Gauss sums associated with arbitrary discriminant modules. Following Str\"omberg, set
\begin{align}
D_k[d]&:=\{y\in D_k\colon dy=0\},\notag\\
D_k[d]^*&:=d\cdot D_k=\{dy\colon y\in D_k\},\notag\\
D_k[d]^\bullet&:=\{y\in D_k\colon dkz^2/2+kyz=0,\,\forall z\in D_k[d]\}.\label{eq:DdBullet}
\end{align}
Scheithauer \cite[Proposition 2.1]{Sch09} shows that $D_k[d]^\bullet=y+D_k[d]^*$ for some $y\in D_k$. (Actually, Scheithauer's lemma is valid for arbitrary discriminant modules besides $D_k$.) As a consequence, $\abs{D_k[d]^\bullet}=\abs{D_k[d]^*}$.
\begin{lemm}[Str\"omberg]
\label{lemm:StrombergscrG}
If $x\not\in D_k[d]^\bullet$, then $\mathscr{G}_{k}(d,x)=0$. Moreover, let $x_0\in D_k[d]^\bullet$. Then for each $x\in D_k[d]^\bullet$ there is a $y\in D_k$ such that $x=x_0+dy$ and for all such $y$ we have
\begin{equation*}
\mathscr{G}_{k}(d,x)=\etp{-dky^2/2-kx_0y}\mathscr{G}_{k}(d,x_0).
\end{equation*}
\end{lemm}
\begin{proof}
The proof of Str\"omberg \cite[Lemma 2.16]{Str13} has a small gap since it is not necessary that $D_k=D_k[d]\oplus D_k[d]^*$, so we provide a proof\footnote{Our proof actually works for arbitrary discriminant modules.} here. Let $\mathscr{R}$ be a complete set of representatives of the quotient group $D_k/D_k[d]$. Applying in \eqref{eq:GDdx} the change of variables corresponding to the bijection $\mathscr{R}\times D_k[d]\rightarrow D_k$, $(y_1,y_2)\mapsto y_1+y_2$ and noting that $d\cdot ky_1y_2=0+L_k$ we obtain
\begin{align}
\mathscr{G}_{k}(d,x)&=\frac{1}{\sqrt{\abs{D_k}\cdot\abs{D_k[d]}}}\sum_{y_1\in \mathscr{R},\,y_2\in D_k[d]}\etp{dk(y_1+y_2)^2/2+kx(y_1+y_2)}\notag\\
&=\frac{1}{\sqrt{\abs{D_k}\cdot\abs{D_k[d]}}}\sum_{y_1\in \mathscr{R}}\etp{dky_1^2/2+kxy_1}\sum_{y_2\in D_k[d]}\etp{dky_2^2/2+kxy_2}.\label{eq:inproofStrombergscrG}
\end{align}
Set $\psi_{d,x}(y_2)=\etp{dky_2^2/2+kxy_2}$, which is a linear character of the finite abelian group $D_k[d]$. It then follows from the orthogonality relations and \eqref{eq:inproofStrombergscrG} that
$$\mathscr{G}_{k}(d,x)=0 \text{ if } x\not\in D_k[d]^\bullet$$
(which is the first desired assertion) and that
\begin{equation}
\label{eq:scrGsimple}
\mathscr{G}_{k}(d,x)=\sqrt{\frac{\abs{D_k[d]}}{\abs{D_k}}}\sum_{y_1\in \mathscr{R}}\etp{dky_1^2/2+kxy_1}
\end{equation}
if $x\in D_k[d]^\bullet$. The above quantity is independent of the choice of the set $\mathscr{R}$ of representatives and hence we can assume $y_1\in D_k/D_k[d]$ in the summation range.

For the second conclusion, let $x_0$ and $x$ be as in the conditions. The existence of $y$ follows from the fact $D_k[d]^\bullet$ is a coset in $D_k/D_k[d]^*$ (\cite[Proposition 2.1]{Sch09}). Finally, for all $y\in D_k$ with $x=x_0+dy$ we have, according to \eqref{eq:scrGsimple},
\begin{align*}
\mathscr{G}_{k}(d,x)&=\sqrt{\frac{\abs{D_k[d]}}{\abs{D_k}}}\sum_{y_1\in D_k/D_k[d]}\etp{dky_1^2/2+k(x_0+dy)y_1)}\\
&=\etp{-dky^2/2-kx_0y}\cdot\sqrt{\frac{\abs{D_k[d]}}{\abs{D_k}}}\sum_{y_1\in D_k/D_k[d]}\etp{dk(y_1+y)^2/2+kx_0(y_1+y)}.
\end{align*}
Applying in the above the change of variables corresponding to the bijection $D_k/D_k[d]\rightarrow D_k/D_k[d]$, $y_1+D_k[d]\mapsto y_1+y+D_k[d]$ and then using \eqref{eq:scrGsimple} with $x$ replaced by $x_0$ we arrive at the final conclusion.
\end{proof}

Str\"omberg indeed gave closed formulas for $\mathscr{G}_{\underline{D}}(d,x)$ where $\underline{D}$ is an arbitrary finite quadratic module; cf. Lemmas 3.6--3.10 and Corollary 3.11 of \cite{Str13}.

\subsection{Weil representations and associated theta series}
\label{subsec:weil_representations_and_associated_theta_series}
For the definition of the Weil representation associated with an arbitrary finite quadratic module, see \cite[\S 5]{Str13}. We only need those associated with $\underline{\numZ}_k$ if $k\in\numZ\setminus\{0\}$ is even and $\underline{2\numZ}_k$ if $k\in\numZ\setminus\{0\}$ is odd. They are even integral lattices as explained in Example \ref{examo:latticeZand2Z}. Let $\sltZ$ be the metaplectic cover of the modular group $\slZ$ (see Appendix \ref{appendix:metaplectic}, in particular Proposition \ref{prop:widetildeG}).
\begin{deff}[Special cases of Weil representations]
\label{def:Weilk1}
If $k\in\numZ\setminus\{0\}$ is even, then we define a representation $\rho_{\underline{\numZ}_k}\colon\sltZ\rightarrow\mathrm{GL}(\numC^{\abs{k}})$ by
\begin{align}
\rho_{\underline{\numZ}_k}\widetilde{\tbtmat{1}{1}{0}{1}}\delta_x &= \etp{kx^2/2}\delta_x, \label{eq:WeilrhoT}\\
\rho_{\underline{\numZ}_k}\widetilde{\tbtmat{0}{-1}{1}{0}}\delta_x &= \sqrt{\frac{-\sgn{k}\rmi}{\abs{D_k}}}\sum_{y \in D_k}\etp{-kxy}\delta_y, \label{eq:WeilrhoS}
\end{align}
where $x\in D_k=\frac{1}{k}\numZ/\numZ$, and where $(\delta_x)_{x\in D_k}$ is the standard basis of $\numC^{\abs{k}}$. On the other hand, if $k\in\numZ\setminus\{0\}$ is odd, then we define a representation $\rho_{\underline{2\numZ}_k}\colon\sltZ\rightarrow\mathrm{GL}(\numC^{\abs{4k}})$ by the same formulas but with $x\in D_k=\frac{1}{2k}\numZ/2\numZ$, and $(\delta_x)_{x\in D_k}$ being the standard basis of $\numC^{\abs{4k}}$.
\end{deff}
The following basic fact makes the above definition reasonable.
\begin{prop}
The assignments \eqref{eq:WeilrhoT} and \eqref{eq:WeilrhoS} extend uniquely to a group homomorphism from $\sltZ$ to $\mathrm{GL}(\numC^{\abs{k}})$ if $k$ even, and to $\mathrm{GL}(\numC^{\abs{4k}})$ if $k$ odd.
\end{prop}
\begin{proof}
We define $t$ ($s$ respectively) by setting $t\delta_x$ ($s\delta_x$ respectively) to be the right-hand side of \eqref{eq:WeilrhoT} (\eqref{eq:WeilrhoS} respectively). They extend to $\numC$-linear isomorphisms. One can check immediately that $s^8=1$, $(st)^3=s^2$, and $s^4t=ts^4$. Note that $\sltZ$ has a presentation with generators $\widetilde{T}=\widetilde{\tbtmat{1}{1}{0}{1}}$, $\widetilde{S}=\widetilde{\tbtmat{0}{-1}{1}{0}}$, and with relations $\widetilde{S}^8=\widetilde{I}$, $(\widetilde{S}\widetilde{T})^3=\widetilde{S}^2$, and $\widetilde{S}^4\widetilde{T}=\widetilde{T}\widetilde{S}^4$ (cf. \cite[Lemma 5.2]{Zhu23_2}). Since the relations of $\widetilde{T}$ and $\widetilde{S}$ are the same as those of $t$ and $s$, $\rho_{\underline{\numZ}_k}\widetilde{T}=t$ and $\rho_{\underline{\numZ}_k}\widetilde{S}=s$ extend uniquely to a homomorphism. The proof for $k$ odd is the same.
\end{proof}

\begin{deff}
\label{def:Weilk2}
Let $\underline{L}_k=\underline{\numZ}_k$ if $k$ is even and $\underline{L}_k=\underline{2\numZ}_k$ if $k$ is odd. For $\gamma\in\sltZ$ and $x,y\in D_k$, we define $\rho_{\underline{L}_k}(\gamma)_{y,x}\in\numC$ to be the $\delta_y$-coefficient of $\rho_{\underline{L}_k}(\gamma)\delta_x$
\end{deff}
Str\"omberg \cite[Remark 6.8]{Str13} gave an explicit formula for $\rho_{\underline{L}}(\gamma)_{y,x}$ (with $\underline{L}$ an arbitrary even integral lattice) which extends Scheithauer's results \cite{Sch09}. Zhu \cite[Theorem 1.4]{Zhu25} provided a more explicit formula.

If we identify the dual space of $\numC^{\abs{k}}$ or $\numC^{\abs{4k}}$ with themselves using the dual basis of $(\delta_x)_{x\in D_k}$, then the \emph{dual representation} of $\rho_{\underline{L}_k}$ is $\rho_{\underline{L}_{-k}}$, namely, 
\begin{equation}
\label{eq:rhoDual}
\rho_{\underline{L}_k}^*(\gamma)_{y,x}:=\overline{\rho_{\underline{L}_k}(\gamma)_{y,x}}=\rho_{\underline{L}_{-k}}(\gamma)_{y,x},\qquad\gamma\in\sltZ,\, y,x\in D_k.
\end{equation}

To each $\underline{L}_k$ with $k>0$, we can associate a family of theta series
\begin{equation}
\label{eq:deffJacobiThetaLatticeIndex}
\vartheta_{\underline{L}_k, t}(\tau,z)=\sum_{v \in t+L_k}\etp{\tau kv^2/2+kvz}=\sum_{w\in L_k}q^{\frac{k(t+w)^2}{2}}\zeta^{k(t+w)},
\end{equation}
where $t\in D_k$, $\tau\in\uhp$, $z\in\numC$. Set $\vartheta_{\underline{L}_k}=\sum_{t \in D_k}\vartheta_{\underline{L}_k, t}\delta_t$ so that it is a vector-valued function from $\uhp\times\numC$ to $\numC^k$ if $k$ is even and to $\numC^{4k}$ if $k$ is odd.

\subsection{Jacobi forms and theta decompositions}
\label{subsec:jacobi_forms_of_lattice_index_and_theta_decompositions}
In this subsection we briefly review the theory of Jacobi forms. The standard reference is the book of Eichler and Zagier \cite{EZ85}. Our definition extends the one in \cite{EZ85} a little in the points that we allow the weight and index to be half-integral, and that the elliptic transformation laws are possibly with respect to $2\numZ$ instead of $\numZ$. An alternative approach, which includes the forms used here as special cases, is to use the theory of Jacobi forms of lattice index (cf. e.g. \cite[\S 2]{Zhu23}).

To deal with half-integral weight Jacobi forms, we need the metaplectic covers of modular groups; see Appendix \ref{appendix:metaplectic}. Let $m\in\frac{1}{2}\numZ$; we also need a variant of the Heisenberg group, denoted by $H(\numR,2m)$, whose elements are $([v,w],\xi)$ with $v,w \in \numR$ and $\xi \in S^1=\{z \in \numC \colon \abs{z}=1\}$. The composition law is given by
\begin{equation*}
([v_1,w_1],\xi_1)\cdot([v_2,w_2],\xi_2)
=([v_1+v_2,w_1+w_2],\xi_1\xi_2 \etp{m(v_1w_2-v_2w_1)}).
\end{equation*}
The semi-direct product $\glptR \ltimes H(\numR,2m)$ is the group of $(\gamma,\varepsilon,[v,w],\xi)$ with $(\gamma,\varepsilon)\in\glptR$ and $([v,w],\xi)\in H(\numR,2m)$. The composition law is thus the following:
\begin{multline*}
(\tbtmat{a_1}{b_1}{c_1}{d_1},\varepsilon_1, [v_1,w_1],\xi_1)\cdot(\tbtmat{a_2}{b_2}{c_2}{d_2},\varepsilon_2, [v_2,w_2],\xi_2)\\
=\left((\tbtmat{a_1}{b_1}{c_1}{d_1},\varepsilon_1)\cdot(\tbtmat{a_2}{b_2}{c_2}{d_2},\varepsilon_2), ([a_2'v_1+c_2'w_1,b_2'v_1+d_2'w_1],\xi_1)\cdot([v_2,w_2],\xi_2)\right),
\end{multline*}
where $\tbtmat{a_2'}{b_2'}{c_2'}{d_2'}=(a_2d_2-b_2c_2)^{-1/2}\tbtmat{a_2}{b_2}{c_2}{d_2}$.
Let $m\in\frac{1}{2}\numZ$, $k\in\frac{1}{2}\numZ$, $n\in\numgeq{Z}{1}$ and let $f$ be a function from $\uhp\times\numC$ to $\numC^n$. Define\footnote{Comparing this and \cite[Theorem 1.4]{EZ85}, one will find a factor $\zeta^m$ appears in Eichler and Zagier's definition, while just $\zeta$ in ours. This difference is caused by the different definitions of Heisenberg groups. Ours is the one valid as well for Jacobi forms of lattice index.}
\begin{multline*}
f\vert_{k,m}(\tbtmat{a}{b}{c}{d},\varepsilon,[v,w],\xi)(\tau,z)=\varepsilon^{-2k}(c'\tau+d')^{-k}\\
\cdot\xi\etp{m\left(-\frac{c}{c\tau+d}(z+\tau v+w)^2+\tau v^2+2vz+vw\right)}\cdot f\left(\frac{a\tau+b}{c\tau+d},\frac{z+\tau v+w}{c^\prime\tau+d^\prime}\right).
\end{multline*}
A direct calculation shows that $f\vert_{k,m}I=f$ and $f\vert_{k,m}(\gamma_1\gamma_2)=(f\vert_{k,m}\gamma_1)\vert_{k,m}\gamma_2$ for $\gamma_1,\gamma_2\in\glptR \ltimes H(\numR,2m)$ and $I$ being the identity element of $\glptR \ltimes H(\numR,2m)$. Set
$$f\vert_{k,m}\tbtmat{a}{b}{c}{d}=f\vert_{k,m}(\tbtmat{a}{b}{c}{d},1,[0,0],1) \text{ and } f\vert_{k,m}([v,w],\xi)=f\vert_{k,m}(\tbtmat{1}{0}{0}{1},1,[v,w],\xi).$$

Let $G$ be a finite index subgroup of $\slZ$ and $L$ be a nontrivial subgroup of $\numZ$. Set $H(L,2m)=\{([v,w],\xi)\colon v,w\in L,\, \xi=\pm1\}$. Then $H(L,2m)$ is a subgroup of $H(\numR,2m)$ and $\widetilde{G}\ltimes H(L,2m)$ is a subgroup of $\glptR \ltimes H(\numR,2m)$.
\begin{deff}[Extension of the definition in page 9 of \cite{EZ85}]
\label{def:JacobiForm}
Let $m\in\frac{1}{2}\numZ$, $k\in\frac{1}{2}\numZ$, $\mathbf{n}\in\numgeq{Z}{1}$, and let $f$ be a holomorphic function from $\uhp\times\numC$ to $\numC^{\mathbf{n}}$ (or any $\numC$-vector space of dimension $\mathbf{n}$). Let $\rho\colon\widetilde{G}\ltimes H(L,2m)\rightarrow\glnC{\mathbf{n}}$ be a group representation with a finite index kernel. We say $f$ is a \emph{weakly holomorphic Jacobi form} on group pair $(G,L)$, of weight $k$, of index $m$ and with multiplier $\rho$ if the following conditions are fulfilled:
\begin{enumerate}
    \item[(a)] $f\vert_{k,m}\gamma=\rho(\gamma)\circ f$ for each $\gamma\in\widetilde{G}\ltimes H(L,2m)$,
    \item[(b)] For each $\gamma\in\widetilde{\slZ}$ there are $n_0\in\numQ$ and $N\in\numgeq{Z}{1}$ such that $f\vert_{k,m}\gamma$ can be expanded as a normally convergent series
    \begin{equation}
    \label{eq:JFFourier}
    f\vert_{k,m}\gamma(\tau,z)= \sum_{n_0\leq n\in\frac{1}{N}\numZ,\, t\in\frac{1}{N}\numZ}c(n,t) q^n\zeta^{2mt}.\qquad (\zeta=\etp{z},\,c(n,t)\in\numC^{\mathbf{n}})
    \end{equation}
\end{enumerate}
Moreover, if $f$ satisfies additionally
\begin{equation}
\label{eq:JFFourier2}
c(n,t)\neq0\implies n\geq mt^2 \text{ in \eqref{eq:JFFourier} for each } \gamma\in\widetilde{\slZ},
\end{equation}
then $f$ is called a \emph{Jacobi form}. The set of all Jacobi forms, which forms a $\numC$-vector space, is denoted by $J_{k,m,L}(G,\rho)$. If $L=\numZ$, then set simply $J_{k,m}(G,\rho):=J_{k,m,\numZ}(G,\rho)$.
\end{deff}

\begin{rema}
(a) The space of all modular forms on group $G$, of weight $k$, and with multiplier $\rho$ is denoted by $M_k(G,\rho)$.

(b) If $\mathbf{n}=1$, we always identify $\glnC{1}$ with $\numC^\times$, the group of nonzero complex numbers, and thus identify $\rho(\gamma)\circ f$ with $\rho(\gamma)\cdot f$, the usual scalar multiplication. Since we have assumed $\ker\rho$ has finite index, $\rho(\gamma)$ must be a root of unity.
\end{rema}

\begin{examp}
\label{examp:thetaL}
Let $k$ be a nonzero integer. The Weil representation $\rho_{\underline{L}_k}$ was introduced in Definitions \ref{def:Weilk1} and \ref{def:Weilk2}. It can be extended to a representation of $\widetilde{\slZ}\ltimes H(L_k,k)$ by setting $\rho_{\underline{L}_k}([v,w],\xi)\delta_x=\xi\delta_x$ for $x\in D_k$. Then if $k>0$ we have $\vartheta_{\underline{L}_k}\in J_{\frac{1}{2},\frac{k}{2},L_k}(\slZ,\rho_{\underline{L}_k})$. (See \eqref{eq:deffJacobiThetaLatticeIndex} and the paragraph following it for the definition of $\vartheta_{\underline{L}_k}$.) A proof can be given by combining \cite[Corollary 3.34]{Boy15}, \cite[Proposition 3.36]{Boy15} and the definition of $\rho_{\underline{L}_k}$. See also \cite[Section 5]{Zhu23}. Below we shall use $\vartheta_{\underline{L}_k}$ as the base of theta decompositions.
\end{examp}
\begin{examp}
\label{examp:classicalTheta}
We need the classical Jacobi theta series
\begin{equation*}
\vartheta(\tau,z):=\rmi\cdot\sum_{r\in\numZ}\legendre{-4}{r}q^{\frac{r^2}{8}}\zeta^{\frac{r}{2}}=\sum_{n\in\frac{1}{2}+\numZ}\rme^{\uppi\rmi n^2\tau+2\uppi\rmi n\left(z+\frac{1}{2}\right)}.\qquad(\zeta=\rme^{2\uppi\rmi z})
\end{equation*}
It is known that $\vartheta\in J_{\frac{1}{2},\frac{1}{2}}(\slZ,\chi_\eta^3\ltimes\chi_H)$ where $\chi_\eta^3\ltimes\chi_H$ is the unitary character of $\widetilde{\slZ}\ltimes H(\numZ,1)$ that sends $\left(\tbtmat{a}{b}{c}{d},\varepsilon\right)$ to $\chi_\eta\left(\tbtmat{a}{b}{c}{d},\varepsilon\right)^3$ where
\begin{equation}
\label{eq:etaChar}
\chi_\eta\left(\tbtmat{a}{b}{c}{d},\varepsilon\right)=\begin{dcases}
\varepsilon\cdot\legendre{d}{\abs{c}}\etp{\frac{1}{24}\left((a+d-3)c-bd(c^2-1)\right)}   & \text{if }2 \nmid c, \\
\varepsilon\cdot\legendre{c}{d}\etp{\frac{1}{24}\left((a-2d)c-bd(c^2-1)+3d-3\right)}   & \text{if }2 \mid c,
\end{dcases}
\end{equation}
and sends $([v,w],\xi)$ to $\chi_H([v,w],\xi)=\xi\cdot(-1)^{vw+v+w}$. Note that $\chi_\eta\colon\sltZ\rightarrow S^1$ is the multiplier system of Dedekind eta function \eqref{eq:eta}. The essential part of the assertion $\vartheta\in J_{\frac{1}{2},\frac{1}{2}}(\slZ,\chi_\eta^3\ltimes\chi_H)$ is the well-known Poisson summation formula for $\vartheta\left(-\frac{1}{\tau},\frac{z}{\tau}\right)$.
\end{examp}
\begin{lemm}
\label{lemm:thetaPowerk}
Let $k$ be a positive integer. Then $\vartheta(\tau,z+\frac{1}{2})^k\in J_{\frac{k}{2},\frac{k}{2}}(\Gamma_0(2),\chi_k)$ where
\begin{align}
\chi_k\colon\widetilde{\Gamma_0(2)}\ltimes H(\numZ,k)&\rightarrow S^1\label{eq:chikFormula}\\
(\tbtmat{a}{b}{c}{d},\varepsilon, [v,w],1) &\mapsto \chi_\eta^{3k}\left(\tbtmat{a}{b}{c}{d},\varepsilon\right)\cdot(-1)^{-kv}\chi_H^k([v,w],1)\notag\\
&\cdot\rmi^{-\frac{kc}{2}}\chi_H^k\left(\left[\tfrac{c}{2},\tfrac{d-1}{2}\right],1\right).\notag
\end{align}
Consequently, $\vartheta(\tau,z+\frac{1}{2})^k\in J_{\frac{k}{2},\frac{k}{2},2\numZ}(\Gamma_0(2),\chi_k)$. (If $2\nmid k$, the lattice $\underline{2\numZ}_{k}$ is even integral but $\underline{\numZ}_{k}$ not.)
\end{lemm}
\begin{proof}
Since $\vartheta(\tau,z)^k\vert_{\frac{k}{2},\frac{k}{2}}([0,\frac{1}{2}],1)=\vartheta(\tau,z+\frac{1}{2})^k$ we consider the left-hand side. The condition (b) of Definition \ref{def:JacobiForm} and \eqref{eq:JFFourier2} can be proved directly by definitions and the fact\footnote{Note that $\chi_H^{k}$ is not the usual power; it maps $([v,w],\xi)$ to $\xi\cdot\chi_H([v,w],1)^k$. See \cite[Proposition 2.10]{Zhu23}.} $\vartheta^k\in J_{\frac{k}{2},\frac{k}{2}}(\slZ,\chi_\eta^{3k}\ltimes\chi_H^{k})$. It remains to verify (a) of Definition \ref{def:JacobiForm} for $f=\vartheta(\tau,z)^k\vert_{\frac{k}{2},\frac{k}{2}}([0,\frac{1}{2}],1)$. Let $(\tbtmat{a}{b}{c}{d},1,[v,w],1)\in\widetilde{\Gamma_0(2)}\ltimes H(\numZ,k)$ be arbitrary. We have, by the composition law of the group $\glptR \ltimes H(\numR,k)$,
\begin{multline}
\label{eq:thetaktransform}
\vartheta(\tau,z)^k\vert_{\frac{k}{2},\frac{k}{2}}([0,\tfrac{1}{2}],1)\vert_{\frac{k}{2},\frac{k}{2}}(\tbtmat{a}{b}{c}{d},1,[v,w],1)=\vartheta(\tau,z)^k\vert_{\frac{k}{2},\frac{k}{2}}(\tbtmat{a}{b}{c}{d},1,[v+\tfrac{c}{2},w+\tfrac{d}{2}],\etp{\tfrac{k(cw-dv)}{4}})\\
=\etp{\tfrac{k(cw-dv)}{2}}\etp{-\tfrac{kc}{8}}\vartheta(\tau,z)^k\vert_{\frac{k}{2},\frac{k}{2}}(\tbtmat{a}{b}{c}{d},1,[v,w],1)\vert_{\frac{k}{2},\frac{k}{2}}([\tfrac{c}{2},\tfrac{d-1}{2}],1)\vert_{\frac{k}{2},\frac{k}{2}}([0,\tfrac{1}{2}],1).
\end{multline}
Since $2\mid c$ we have $\etp{\tfrac{k(cw-dv)}{2}}=(-1)^{-kv}$, $\etp{-\tfrac{kc}{8}}=\rmi^{-\frac{kc}{2}}$ and
$$\vartheta(\tau,z)^k\vert_{\frac{k}{2},\frac{k}{2}}(\tbtmat{a}{b}{c}{d},1,[v,w],1)\vert_{\frac{k}{2},\frac{k}{2}}([\tfrac{c}{2},\tfrac{d-1}{2}],1)=\chi_\eta^{3k}\left(\tbtmat{a}{b}{c}{d},1\right)\cdot\chi_H^k([v,w],1)\cdot\chi_H^k\left(\left[\tfrac{c}{2},\tfrac{d-1}{2}\right],1\right)\vartheta(\tau,z)^k.$$
Inserting these into \eqref{eq:thetaktransform} we find that
\begin{equation*}
\vartheta(\tau,z)^k\vert_{\frac{k}{2},\frac{k}{2}}([0,\tfrac{1}{2}],1)\vert_{\frac{k}{2},\frac{k}{2}}(\tbtmat{a}{b}{c}{d},1,[v,w],1)=\chi_k(\tbtmat{a}{b}{c}{d},1,[v,w],1)\cdot\vartheta(\tau,z)^k\vert_{\frac{k}{2},\frac{k}{2}}([0,\tfrac{1}{2}],1)
\end{equation*}
which concludes our proof.
\end{proof}

\begin{rema}
(a) A similar reasoning shows that $q^{\frac{k}{8}}\zeta^{\frac{k}{2}}\vartheta(\tau,z+\frac{\tau}{2})^k\in J_{\frac{k}{2},\frac{k}{2}}(\Gamma_0^0(1,2),\chi'_k)$ where
\begin{equation*}
\Gamma_0^0(1,2)=\{\tbtmat{a}{b}{c}{d}\in\slZ\colon 2\mid b\},
\end{equation*}
and
\begin{align*}
\chi'_k\colon\widetilde{\Gamma_0^0(1,2)}\ltimes H(\numZ,k)&\rightarrow S^1\\
(\tbtmat{a}{b}{c}{d},\varepsilon, [v,w],1) &\mapsto \chi_\eta^{3k}\left(\tbtmat{a}{b}{c}{d},\varepsilon\right)\cdot(-1)^{kw}\chi_H^k([v,w],1)\\
&\cdot\rmi^{\frac{kb}{2}}\chi_H^k\left(\left[\tfrac{a-1}{2},\tfrac{b}{2}\right],1\right).
\end{align*}
(b) The above two Jacobi forms are related to each other by
\begin{equation}
\label{eq:twothetaS}
\vartheta(\tau,z+\tfrac{1}{2})^k\vert_{\frac{k}{2},\frac{k}{2}}\widetilde{\tbtmat{0}{-1}{1}{0}}=\etp{-\tfrac{3k}{8}}q^{\frac{k}{8}}\zeta^{\frac{k}{2}}\vartheta(\tau,z+\tfrac{\tau}{2})^k.
\end{equation}
This can be proved by expanding the definition of the left-hand side and then using the facts $\vartheta^k\vert_{\frac{k}{2},\frac{k}{2}}\tbtmat{0}{-1}{1}{0}(\tau,z+\tfrac{\tau}{2})=\chi_\eta^{3k}\tbtmat{0}{-1}{1}{0}\vartheta^k(\tau,z+\tfrac{\tau}{2})$ and $\chi_\eta\tbtmat{0}{-1}{1}{0}=\etp{-\tfrac{1}{8}}$.
\end{rema}

Now we describe the \emph{theta decompositions}. The reader may find more details and different versions in \cite[Theorem 5.1]{EZ85}, \cite[Theorem 1]{Kri96}, \cite[Proposition 7]{Wil19}, \cite[Theorem 2.5]{Zem21} or \cite[Theorem 5.2, Remark 5.3]{Zhu23}. The following version is a special case discussed in \cite[Remark 5.3]{Zhu23}.
\begin{thm}
\label{thm:thetaDecomp}
Let $0<m\in\frac{1}{2}\numZ$ and let $L_{2m}$ be the subgroup defined in Definition \ref{def:LkLksharp}. Let $k\in\frac{1}{2}\numZ$ and $G$ be a finite index subgroup of $\slZ$. Let $\chi\colon\widetilde{G}\rightarrow S^1$ be a unitary character whose kernel is of finite index. Let $\chi'\colon\widetilde{G}\ltimes H(L_{2m},2m)\rightarrow S^1$ be the character\footnote{As one can check, $\chi'$ is indeed a group character because the lattice $\underline{L}_{2m}$ is even integral.} given by $$\chi'(\gamma,\varepsilon,[v,w],\xi)=\xi\cdot\chi(\gamma,\varepsilon).$$ Then the map
\begin{align*}
J_{k,m,L_{2m}}(G,\chi')&\rightarrow M_{k-\frac{1}{2}}(G,\chi\cdot\rho_{\underline{L}_{2m}}^*\vert_{\widetilde{G}})\\
\sum_{n,t}c(n,t) q^n\zeta^{2mt} &\mapsto\sum_n\left(\sum_{t\in D_{2m}}c(n+mt^2,t)\delta_t\right)q^n
\end{align*}
is a $\numC$-linear isomorphism. (For the definitions of $\rho_{\underline{L}_{2m}}^*$ and of $\delta_t$, see \eqref{eq:rhoDual} and Definition \ref{def:Weilk1}, respectively.) Moreover, the inverse map is given by
\begin{equation*}
\sum_{t\in D_{2m}}h_t(\tau)\delta_t\mapsto\sum_{t\in D_{2m}}h_t(\tau)\vartheta_{\underline{L}_{2m},t}(\tau,z).
\end{equation*}
\end{thm}
\begin{proof}
Setting in \cite[Theorem 5.2 and Remark 5.3]{Zhu23} $\underline{L}=\underline{L}_{2m}$, $\mathcal{W}=\numC$, $\rho=\chi'$, and notice that $\underline{L}_{2m}$ is even integral.
\end{proof}
\begin{examp}
If $k$ is an even positive integer, then $\vartheta(\tau,z+\frac{1}{2})^k\in J_{\frac{k}{2},\frac{k}{2}}(\Gamma_0(2),\chi_k)$ by Lemma \ref{lemm:thetaPowerk}. Since $D_k=\frac{1}{k}\numZ/\numZ$, we have according to Theorem \ref{thm:thetaDecomp} that
\begin{equation}
\label{eq:thetaExpansionkEven}
\vartheta\left(\tau,z+\frac{1}{2}\right)^k=\sum_{t\in\frac{1}{k}\numZ/\numZ}h_t^{(k)}(\tau)\vartheta_{\underline{\numZ}_k,t}(\tau,z)
\end{equation}
where
\begin{gather}
h_t^{(k)}(\tau):=\etp{kt}\sum_{n\in\numQ}\rho_k\left(2n+k\left(t-\frac{1}{2}\right)^2,k\left(t-\frac{1}{2}\right)\right)q^n,\label{eq:htk}\\
\rho_k(m,r):=\#\{(x_1,\dots,x_k)\in\numZ^k\colon\sum x_j^2=m,\, \sum x_j=r\}.\notag
\end{gather}
The $\numC^k$-valued modular form $\sum_{t\in\frac{1}{k}\numZ/\numZ}h_t^{(k)}(\tau)\delta_t$ thus belongs to $M_{\frac{k-1}{2}}(\Gamma_0(2),\chi_k\cdot\rho_{\underline{\numZ}_k}^*)$. We fix a set of representatives $\{\frac{0}{k},\frac{1}{k},\dots,\frac{k-1}{k}\}$ of $\frac{1}{k}\numZ/\numZ$. Then the modular transformation equations can be expressed as
\begin{equation}
\label{eq:hbetakvertgamma}
h_{\beta/k}^{(k)}\vert_{\frac{k-1}{2}}\gamma=\chi_k(\gamma)\cdot\sum_{0\leq\beta'<k}\rho_{\underline{\numZ}_k}^*(\gamma)_{\frac{\beta}{k},\frac{\beta'}{k}}\cdot h_{\beta'/k}^{(k)},\quad\beta\in\{0,1,\dots,k-1\},\,\gamma\in\widetilde{\Gamma_0(2)}.
\end{equation}
On the other hand, if $k$ is odd positive, then $\vartheta(\tau,z+\frac{1}{2})^k\in J_{\frac{k}{2},\frac{k}{2},2\numZ}(\Gamma_0(2),\chi_k)$ where $\underline{2\numZ}_{k}$ is even, and hence (noting now $D_k=\frac{1}{2k}\numZ/2\numZ$)
\begin{equation}
\label{eq:thetaExpansionkOdd}
\vartheta\left(\tau,z+\frac{1}{2}\right)^k=\sum_{t\in\frac{1}{2k}\numZ/2\numZ}h_t^{(k)}(\tau)\vartheta_{\underline{2\numZ}_k,t}(\tau,z).
\end{equation}
(The coefficients $h_t^{(k)}$ are defined as above.) Theorem \ref{thm:thetaDecomp} now shows that the $\numC^{4k}$-valued modular form $\sum_{t\in\frac{1}{2k}\numZ/2\numZ}h_t^{(k)}(\tau)\delta_t\in M_{\frac{k-1}{2}}(\Gamma_0(2),\chi_k\cdot\rho_{\underline{2\numZ}_k}^*)$. We fix a set of representatives $\{\frac{0}{2k},\frac{1}{2k},\dots,\frac{4k-1}{2k}\}$ of $\frac{1}{2k}\numZ/2\numZ$; the matrix form of the above transformation equation is
\begin{equation}
\label{eq:hbetakvertgammaOdd}
h_{\beta/k}^{(k)}\vert_{\frac{k-1}{2}}\gamma=\chi_k(\gamma)\cdot\sum_{0\leq2\beta'<4k}\rho_{\underline{2\numZ}_k}^*(\gamma)_{\frac{\beta}{k},\frac{\beta'}{k}}\cdot h_{\beta'/k}^{(k)},\quad\beta\in\left\{\frac{0}{2},\frac{1}{2},\dots,\frac{4k-1}{2}\right\},\,\gamma\in\widetilde{\Gamma_0(2)}.
\end{equation}
Note that $h_{\beta/k}^{(k)}=0$ if $2\beta$ is even; hence in the above formula we can assume $\beta,\beta'\in\left\{\frac{1}{2},\frac{3}{2},\dots,\frac{4k-1}{2}\right\}$.
\end{examp}
\begin{rema}
Jiang, Rolen and Woodbury \cite[Theorem 2]{JRW22} gave a recursive formula for $h_{\beta/k}^{(k)}$ which was used by them to obtain expressions of $C\Psi_{k,\beta}$ as combinations of Dedekind eta functions and Klein forms, whereas the purpose of our treatment is the precise transformation equations with respect to a modular group as large as possible and with all multipliers explicitly presented.
\end{rema}

We conclude this section by a formula concerning coefficients of Weil representations. It is relevant to \cite[Theorem 5.4]{Bor00}.
\begin{lemm}
\label{lemm:coeffWeil}
Suppose $k\in\numgeq{Z}{1}$. Let $\gamma=\tbtmat{a}{b}{c}{d}\in\slZ$ satisfying $d\neq0$, and $k\mid bc$ if $2\mid k$; $4k\mid bc$ if $2\nmid k$. Let $x,y\in D_{k}$. If $ay-x\not\in D_{k}[c]^\bullet$, then $\rho_{\underline{L}_k}(\widetilde{\gamma})_{y,x}=0$; otherwise we have
\begin{equation}
\label{eq:rhogammayx}
\rho_{\underline{L}_k}(\widetilde{\gamma})_{y,x}=(-\rmi)^{\frac{1-\sgn d}{2}\cdot\legendre{c}{-1}}\cdot\mathfrak{g}_{k}(b,d;y)\cdot\sqrt{\frac{\abs{D_{k}[c]}}{\abs{D_{k}}}}\cdot\mathscr{G}_{k}(-cd,y-dx).
\end{equation}
\end{lemm}

\begin{proof}
Setting $\underline{L}=\underline{L}_k$ in \cite[Proposition 3.17 and Remark 3.18(c)]{Zhu25}.
\end{proof}

\section{Vector-valued modular transformations of $C\Psi_{k,\beta}$}
\label{r:sec3}

This section mainly revolves around Theorem \ref{thm:fkbetaSLZ}, which we first recall. Note that $f_{k,\beta}(\tau):=q^{\frac{k}{12}-\frac{\beta^2}{2k}}C\Psi_{k,\beta}(q)$.
As shown by Jiang, Rolen and Woodbury \cite[Theorem 1]{JRW22}, for any $k\in\numgeq{Z}{1}$ and $\beta\in\frac{k}{2}+\numZ$, $C\Psi_{k,\beta}$ is the $\zeta^\beta$-coefficient of $\left(\frac{-\vartheta(\tau,z+1/2)}{q^{1/12}\eta(\tau)}\right)^k$. By the definition we find that the $\zeta^\beta$-coefficient of $\vartheta(\tau,z+1/2)^k$ is $q^{\frac{\beta^2}{2k}}\cdot h_{\beta/k}^{(k)}(\tau)$; see \eqref{eq:htk}. It follows that
\begin{equation}
\label{eq:fkbeta}
f_{k,\beta}=q^{\frac{k}{12}-\frac{\beta^2}{2k}}\cdot C\Psi_{k,\beta}=\frac{(-1)^kh_{\beta/k}^{(k)}}{\eta^k}.
\end{equation}
No matter whether $k$ is even or odd, we have
\begin{equation}
\label{eq:htsymmetry}
h_t^{(k)}=h_{t+1}^{(k)}=h_{1-t}^{(k)}
\end{equation}
by \eqref{eq:htk}. Thus without loss of generality, we may assume $\beta\in\mathfrak{B}_k$ (see \eqref{eq:mathfrakBk}) when investigating $f_{k,\beta}$. 

\begin{thm1.1*}
Let $k\in\numgeq{Z}{1}$, $\beta\in\mathfrak{B}_k$. Then we have
\begin{align}
f_{k,\beta}\vert_{-\frac{1}{2}}\tbtmat{1}{1}{0}{1}&=\etp{\frac{k}{12}-\frac{\beta^2}{2k}}f_{k,\beta},\label{eq:fkbetaVertT}\\
f_{k,\beta}\vert_{-\frac{1}{2}}\tbtmat{0}{-1}{1}{0}&=\sum_{\beta'\in\mathfrak{B}_k}s_{\beta,\beta'}^{(k)}\cdot f_{k,\beta'},\label{eq:fkbetaVertS}
\end{align}
where
\begin{equation*}
s_{\beta,\beta'}^{(k)}=\mu_{\beta'/k}\cdot\rmi^{2\beta}\etp{\frac{2k+1}{8}}\frac{2}{\sqrt{k}}\cdot\cos\left(\frac{2\uppi\beta'(\beta+\tfrac{k}{2})}{k}\right),
\end{equation*}
$\mu_t=1/2$ if $t=0$ or $1/2$ and $\mu_t=1$ else.
\end{thm1.1*}

\begin{proof}
The $\tbtmat{1}{1}{0}{1}$ transformation equation follows directly from the definition. The proofs of the $\tbtmat{0}{-1}{1}{0}$ transformation equations for $k$ being even or odd are similar, so here we only present the odd case. We apply the operator $\vert_{\frac{k}{2},\frac{k}{2}}\tbtmat{0}{-1}{1}{0}$ on both sides of \eqref{eq:thetaExpansionkOdd} (for even $k$, on \eqref{eq:thetaExpansionkEven}). Taking into account \eqref{eq:twothetaS}, we find that
\begin{equation}
\label{eq:SactonThetaDecomp}
\etp{-\frac{3k}{8}}q^{\frac{k}{8}}\zeta^{\frac{k}{2}}\vartheta\left(\tau,z+\frac{\tau}{2}\right)^k=\sum_{t\in\frac{1}{2k}\numZ/2\numZ}h_t^{(k)}(\tau)\vert_{\frac{k-1}{2}}\tbtmat{0}{-1}{1}{0}\cdot\vartheta_{\underline{2\numZ}_k,t}(\tau,z)\vert_{\frac{1}{2},\frac{k}{2}}\tbtmat{0}{-1}{1}{0}.
\end{equation}
Since $\underline{2\numZ}_k$ is an even lattice, we have $\vartheta_{\underline{2\numZ}_k}\in J_{\frac{1}{2},\frac{k}{2},2\numZ}(\slZ,\rho_{\underline{2\numZ}_k})$ (see Example \ref{examp:thetaL}). Therefore,
\begin{equation}
\label{eq:theta2ZkunderS}
\vartheta_{\underline{2\numZ}_k,t}\vert_{\frac{1}{2},\frac{k}{2}}\tbtmat{0}{-1}{1}{0}=\sqrt{\frac{-\rmi}{4k}}\sum_{y\in\frac{1}{2k}\numZ/2\numZ}\etp{-kty}\vartheta_{\underline{2\numZ}_k,y}
\end{equation}
by \eqref{eq:WeilrhoS}. On the other hand, applying the theta decomposition (Theorem \ref{thm:thetaDecomp}) to $q^{\frac{k}{8}}\zeta^{\frac{k}{2}}\vartheta\left(\tau,z+\frac{\tau}{2}\right)^k$ and then comparing the obtained coefficients to \eqref{eq:thetaExpansionkOdd} with $z$ replaced by $z+\frac{\tau-1}{2}$ we find that
\begin{equation}
\label{eq:thetaSdecomp}
q^{\frac{k}{8}}\zeta^{\frac{k}{2}}\vartheta\left(\tau,z+\frac{\tau}{2}\right)^k=\sum_{y\in\frac{1}{2k}\numZ/2\numZ}\etp{\frac{k}{4}}\etp{-\frac{ky}{2}}h_{y-1/2}^{(k)}(\tau)\vartheta_{\underline{2\numZ}_k,y}(\tau,z).
\end{equation}
(This can also be proved using the fact $\left(\frac{-\vartheta(\tau,z+1/2)}{q^{1/12}\eta(\tau)}\right)^k=\sum_{\beta\in k/2+\numZ}C\Psi_{k,\beta}(q)\zeta^\beta$ and the change of variables $z\mapsto z+\frac{\tau-1}{2}$.) Inserting \eqref{eq:thetaSdecomp} and \eqref{eq:theta2ZkunderS} into the left-hand side and the right-hand side of \eqref{eq:SactonThetaDecomp}, respectively, we obtain
\begin{multline*}
\etp{-\frac{3k}{8}}\sum_{y\in\frac{1}{2k}\numZ/2\numZ}\etp{\frac{k}{4}}\etp{-\frac{ky}{2}}h_{y-1/2}^{(k)}(\tau)\vartheta_{\underline{2\numZ}_k,y}(\tau,z)\\
=\sqrt{\frac{-\rmi}{4k}}\sum_{y\in\frac{1}{2k}\numZ/2\numZ}\sum_{t\in\frac{1}{2k}\numZ/2\numZ}h_t^{(k)}(\tau)\vert_{\frac{k-1}{2}}\tbtmat{0}{-1}{1}{0}\etp{-kty}\vartheta_{\underline{2\numZ}_k,y}(\tau,z).
\end{multline*}
Since $\vartheta_{\underline{2\numZ}_k,y}$, $y\in\frac{1}{2k}\numZ/2\numZ$ are $\numC$-linearly independent, equating the $\vartheta_{\underline{2\numZ}_k,y}$-coefficients of both sides and replacing $y$ with $y+1/2$ we obtain
\begin{equation*}
h_{y}^{(k)}(\tau)=\frac{1}{\sqrt{4k}}\etp{\frac{3k-1}{8}}\etp{\frac{ky}{2}}\sum_{t\in\frac{1}{2k}\numZ/2\numZ}\etp{-kt\left(y+\frac{1}{2}\right)}h_t^{(k)}(\tau)\vert_{\frac{k-1}{2}}\tbtmat{0}{-1}{1}{0}.
\end{equation*}
On both sides of this identity we apply the operator $\vert_{\frac{k-1}{2}}\tbtmat{0}{-1}{1}{0}$ and note that (being careful that we are dealing with the double cover $\widetilde{\slZ}$ especially when $2\mid k$) $(\tbtmat{0}{-1}{1}{0},1)^2=(\tbtmat{-1}{0}{0}{-1},1)$. This gives
\begin{equation*}
h_{y}^{(k)}(\tau)\vert_{\frac{k-1}{2}}\tbtmat{0}{-1}{1}{0}=\frac{1}{\sqrt{4k}}\etp{\frac{k+1}{8}}\etp{\frac{ky}{2}}\sum_{t\in\frac{1}{2k}\numZ/2\numZ}\etp{-kt\left(y+\frac{1}{2}\right)}h_t^{(k)}(\tau).
\end{equation*}
It follows from this and $h_t^{(k)}=h_{t+1}^{(k)}=h_{1-t}^{(k)}$ that
\begin{multline*}
h_{y}^{(k)}(\tau)\vert_{\frac{k-1}{2}}\tbtmat{0}{-1}{1}{0}=\frac{1}{\sqrt{4k}}\etp{\frac{k+1}{8}}\etp{\frac{ky}{2}}\sum_{t\in k^{-1}\mathfrak{B}_k}h_t^{(k)}(\tau)\mu_t\\
\cdot\left(\etp{-kt\left(y+\frac{1}{2}\right)}+\etp{-k(1-t)\left(y+\frac{1}{2}\right)}+\etp{-k(t+1)\left(y+\frac{1}{2}\right)}+\etp{-k(2-t)\left(y+\frac{1}{2}\right)}\right).
\end{multline*}
Setting in the above $y=\beta/k$, $t=\beta'/k$ and then dividing the obtained identity by
\begin{equation*}
\eta^k\vert_{\frac{k}{2}}\tbtmat{0}{-1}{1}{0}=\etp{-\frac{k}{8}}\eta^k
\end{equation*}
we arrive at \eqref{eq:fkbetaVertS} as desired.
\end{proof}

\begin{coro}
\label{coro:kOddvert10k1}
Let $k$ be an odd positive integer and $\beta\in\mathfrak{B}_k$. Then
\begin{equation}
\label{eq:10k1}
f_{k,\beta}\big\vert_{-\frac{1}{2}}\tbtMat{1}{0}{k}{1}=\etp{\frac{k^2}{24}}f_{k,\beta}.
\end{equation}
\end{coro}
\begin{proof}
Set $\mathbf{f}_k=(f_{k,\beta})_{\beta\in\mathfrak{B}_k}$, viewed as a column vector. Then Theorem \ref{thm:fkbetaSLZ} is equivalent to $(\eta^k\mathbf{f}_k)\vert_{\frac{k-1}{2}}\tbtmat{1}{1}{0}{1}=M_T\cdot\eta^k\mathbf{f}_k$ and $(\eta^k\mathbf{f}_k)\vert_{\frac{k-1}{2}}\tbtmat{0}{-1}{1}{0}=M_S\cdot\eta^k\mathbf{f}_k$ where $M_T$ and $M_S$ are certain square matrices. Moreover, let $M_{-I}$ be the matrix such that $(\eta^k\mathbf{f}_k)\vert_{\frac{k-1}{2}}\tbtmat{-1}{0}{0}{-1}=M_{-I}\cdot\eta^k\mathbf{f}_k$. We have
\begin{equation*}
\tbtMat{1}{0}{k}{1}=\tbtMat{-1}{0}{0}{-1}\tbtMat{0}{-1}{1}{0}\tbtMat{1}{1}{0}{1}^{-k}\tbtMat{0}{-1}{1}{0}.
\end{equation*}
Since the weight $\frac{k-1}{2}$ is an integer, we need not lift the above matrix identity to the double cover $\sltZ$. It follows that
\begin{equation*}
(\eta^k\mathbf{f}_k)\vert_{\frac{k-1}{2}}\tbtmat{1}{0}{k}{1}=M_{-I}M_{S}M_{T}^{-k}M_S\cdot\eta^k\mathbf{f}_k.
\end{equation*}
Since $\eta\vert_{\frac{1}{2}}\tbtmat{1}{1}{0}{1}=\etp{\tfrac{1}{24}}\eta$ and $2\nmid k$ we find that
\begin{align*}
(\eta^k\mathbf{f}_k)\vert_{\frac{k-1}{2}}\tbtmat{1}{-k}{0}{1}&=\left(\eta\vert_{\frac{1}{2}}\tbtmat{1}{-k}{0}{1}\right)^k\cdot \mathbf{f}_k\vert_{-\frac{1}{2}}\tbtmat{1}{-k}{0}{1}\\
&=\etp{-\frac{k^2}{24}}\cdot\mathop{\mathrm{diag}}\left(-k\left(\frac{k}{12}-\frac{\beta^2}{2k}\right)\right)_{\beta\in\mathfrak{B}_k}\cdot\eta^k\mathbf{f}_k\\
&=\eta^k\mathbf{f}_k,
\end{align*}
where in the second equality we have used \eqref{eq:fkbetaVertT} ($\mathop{\mathrm{diag}}$ means diagonal matrix). By this and the definition of $M_T$ we find that $M_T^{-k}$ is the identity matrix, and hence so is $M_{-I}M_{S}M_{T}^{-k}M_S$. In another words, $(\eta^k\mathbf{f}_k)\vert_{\frac{k-1}{2}}\tbtmat{1}{0}{k}{1}=\eta^k\mathbf{f}_k$. This identity, multiplied by $\eta^{-k}\vert_{-\frac{k}{2}}\tbtmat{1}{0}{k}{1}=\etp{\tfrac{k^2}{24}}\eta^{-k}$ (see \eqref{eq:etaChar}), implies the desired identity.
\end{proof}
\begin{rema}
The corresponding result in the case $2\mid k$ is
\begin{equation}
\label{eq:102k1}
f_{k,\beta}\big\vert_{-\frac{1}{2}}\tbtMat{1}{0}{2k}{1}=\etp{\frac{k^2}{12}}f_{k,\beta}.
\end{equation}
The proof is a bit more involved since one has to manipulate elements in the double cover $\sltZ$. We omit it since we will provide a detailed proof of the more general Theorem \ref{thm:fkbetaGamma0k}, which includes \eqref{eq:10k1} and \eqref{eq:102k1} as special cases.
\end{rema}

\section{Modular permutations of $C\Psi_{k,\beta}$ under $\Gamma_0(k)$}
\label{r:sec4}

In this section, we demonstrate that if $\gamma\in\Gamma_0(k)$, then the transformation $f_{k,\beta}\vert_{-1/2}\gamma$ results in a single term $f_{k,\beta'}$ up to a constant factor, and possibly with a different $\beta'$. We also explicitly determine the constant and provide a formula for $\beta'$.

To state the formula for $\beta'$, we need some auxiliary notations. For a fixed $k\in\numgeq{Z}{1}$ and $a\in\numZ$, let $\overline{a}$ denote the least nonnegative remainder of $a\bmod 2k$.
\begin{deff}
We define $\lambda^{(k)}\colon\tfrac{k}{2}+\numZ\rightarrow\mathfrak{B}_k$ by the formula
\begin{equation*}
\lambda^{(k)}(\beta)=\begin{dcases}
\overline{2\beta}/2 & \text{ if }0\leq\overline{2\beta}\leq k,\\
k-\overline{2\beta}/2 & \text{ if } k< \overline{2\beta} < 2k.
\end{dcases}
\end{equation*}
When $k$ is fixed or implicitly understood, the superscript can be omitted: $\lambda(\beta)=\lambda^{(k)}(\beta)$.
\end{deff}

The function $\lambda^{(k)}(\beta)$ is designed such that $f_{k,\beta}=f_{k,\lambda(\beta)}$ for $\beta\in\frac{k}{2}+\numZ$. Moreover, we have $\lambda^{(k)}(k\pm\beta)=\lambda^{(k)}(\beta)$. Note that equivalently $\lambda^{(k)}(\beta)$ can be defined as the unique number in $\mathfrak{B}_k\cap(\pm\beta+k\numZ)$.

Now we state the main theorem of this section.
\begin{thm}
\label{thm:fkbetaGamma0k}
Let $k\in\numgeq{Z}{1}$, $\beta\in\mathfrak{B}_k$ and $\tbtmat{a}{b}{c}{d}\in\Gamma_0(k)$. Then there exist $\beta'\in\mathfrak{B}_k$ and a root of unity $p_{\beta}=p_{\beta}\tbtmat{a}{b}{c}{d}$ such that
\begin{equation*}
f_{k,\beta}\vert_{-\frac{1}{2}}\tbtmat{a}{b}{c}{d}=p_{\beta}\cdot f_{k,\beta'}.
\end{equation*}
In this formula, $\beta'$ is given by
\begin{equation}
\label{eq:betapFormula}
\beta'=T_\beta \tbtmat{a}{b}{c}{d}:=\begin{dcases}
\lambda(a\beta) & \text{ if } 2\mid k,\,2k\mid c,\\
\lambda(a\beta-\tfrac{k}{2}) & \text{ if } 2\mid k,\,2k\nmid c,\\
\lambda(a\beta) & \text{ if } 2\nmid k,\,2\nmid a,\\
\lambda(a\beta-\tfrac{k}{2}) & \text{ if } 2\nmid k,\,2\mid a.
\end{dcases}
\end{equation}
Moreover, when $2\mid c$ then $p_{\beta}$ is defined by
\begin{equation*}
p_{\beta}\tbtmat{a}{b}{c}{d}:=\legendre{\sgn{d}kb}{\abs{d}}\cdot\etp{\frac{\abs{d}-1}{8}}\cdot\rmi^{\frac{1-\sgn{d}}{2}\legendre{c}{-1}}\cdot\etp{\frac{k(2ac-cd+2bd-2bdc^2)}{24}-\frac{bd(a\beta)^2}{2k}};
\end{equation*}
when $2\nmid c$ (in which case we must have $2\nmid k$),
\begin{equation*}
p_{\beta}\tbtmat{a}{b}{c}{d}:=\begin{dcases}
\etp{\frac{\sgn{c}k^2}{24}}\cdot p_{\beta}\tbtMat{a}{b}{c-\sgn{c}ka}{d-\sgn{c}kb} &\text{if } 2\nmid a,\\
\etp{-\frac{k}{12}+\frac{\beta^2}{2k}+\frac{\sgn{c}k^2}{24}}\cdot p_{\beta}\tbtMat{a_1}{b_1}{c_1}{d_1} &\text{if } 2\mid a,
\end{dcases}
\end{equation*}
where
\begin{align}
a_1&=a+c, & b_1&=b+d, \notag\\
c_1&=-\sgn{c}ka+(1-\sgn{c}k)c, & d_1&=-\sgn{c}kb+(1-\sgn{c}k)d.\label{eq:a1b1c1d1}
\end{align}
\end{thm}
\begin{proof}
We divide the proof into five cases:
\begin{gather*}
\text{(i) }2\mid k,\,2k\mid c,\qquad\text{(ii) }2\mid k,\,2k\nmid c,\\
\text{(iii) }2\nmid k,\,2\mid c,\qquad\text{(iv) }2\nmid k,\,2\nmid c,\,2\nmid a,\qquad\text{(v) }2\nmid k,\,2\nmid c,\,2\mid a.
\end{gather*}

\textbf{Case (i).} Since $2\mid k$ and $\tbtmat{a}{b}{c}{d}\in\Gamma_0(k)$, we have $\tbtmat{a}{b}{c}{d}\in\Gamma_0(2)$. Thus, \eqref{eq:hbetakvertgamma} and \eqref{eq:htsymmetry} imply that
\begin{equation}
\label{eq:hbetakvertgamma2}
h_{\beta/k}^{(k)}\vert_{\frac{k-1}{2}}\widetilde{\tbtmat{a}{b}{c}{d}}=\chi_k\widetilde{\tbtmat{a}{b}{c}{d}}\cdot\sum_{\beta'\in\mathscr{B}_k}\mu_{\beta'/k}\cdot\left(\rho_{\underline{\numZ}_k}^*\widetilde{\tbtmat{a}{b}{c}{d}}_{\frac{\beta}{k},\frac{\beta'}{k}}+\rho_{\underline{\numZ}_k}^*\widetilde{\tbtmat{a}{b}{c}{d}}_{\frac{\beta}{k},\frac{k-\beta'}{k}}\right)\cdot h_{\beta'/k}^{(k)}.
\end{equation}
Since $bc\cdot\frac{1}{k}\numZ\subseteq\numZ$ we can apply Lemma \ref{lemm:coeffWeil} to $\underline{\numZ}_k$. Therefore, for $\beta,\beta'\in\numZ$ we have $\rho_{\underline{\numZ}_k}^*\widetilde{\tbtmat{a}{b}{c}{d}}_{\frac{\beta}{k},\frac{\beta'}{k}}=0$ unless $\tfrac{a\beta-\beta'}{k}\in D_{k}[c]^\bullet$, which, by expanding the definition, is equivalent to $\beta'\equiv a\beta\bmod{k}$ since $2k\mid c$. Inserting this into \eqref{eq:hbetakvertgamma2} we obtain
\begin{equation}
\label{eq:hbetakvertgamma3}
h_{\beta/k}^{(k)}\vert_{\frac{k-1}{2}}\widetilde{\tbtmat{a}{b}{c}{d}}=\chi_k\widetilde{\tbtmat{a}{b}{c}{d}}\cdot\rho_{\underline{\numZ}_k}^*\widetilde{\tbtmat{a}{b}{c}{d}}_{\frac{\beta}{k},\frac{a\beta}{k}}\cdot h_{\lambda(a\beta)/k}^{(k)}.
\end{equation}
By the above identity, \eqref{eq:fkbeta} and \eqref{eq:rhoDual} it remains to prove
\begin{multline}
\label{eq:toProveCase1}
\chi_\eta^{-k}\widetilde{\tbtMat{a}{b}{c}{d}}\chi_k\widetilde{\tbtMat{a}{b}{c}{d}}\cdot\overline{\rho_{\underline{\numZ}_k}\widetilde{\tbtMat{a}{b}{c}{d}}_{\frac{\beta}{k},\frac{a\beta}{k}}}\\
=\legendre{\sgn{d}kb}{\abs{d}}\cdot\etp{\frac{\abs{d}-1}{8}}\cdot\rmi^{\frac{1-\sgn{d}}{2}\legendre{c}{-1}}\cdot\etp{\frac{k(2ac-cd+2bd-2bdc^2)}{24}-\frac{bd(a\beta)^2}{2k}}.
\end{multline}
According to \eqref{eq:rhogammayx}, we have
\begin{equation}
\label{eq:rhoZkbab}
\rho_{\underline{\numZ}_k}\widetilde{\tbtMat{a}{b}{c}{d}}_{\frac{\beta}{k},\frac{a\beta}{k}}=(-\rmi)^{\frac{1-\sgn d}{2}\cdot\legendre{c}{-1}}\cdot\mathfrak{g}_{k}\left(b,d;\frac{\beta}{k}\right)\cdot\mathscr{G}_{k}(-cd,0)
\end{equation}
where we have used the fact $\abs{D_{k}[c]}=\abs{D_{k}}$ and $\tfrac{\beta}{k}-\tfrac{ad\beta}{k}+\numZ=0+\numZ$ in $D_{k}$. By the definition and $2k\mid c$ we have $\mathscr{G}_{k}(-cd,0)=1$. Inserting this and Lemma \ref{lemm:gbdt} into \eqref{eq:rhoZkbab} we obtain
\begin{equation}
\label{eq:eq:rhoZkbabSimple}
\rho_{\underline{\numZ}_k}\widetilde{\tbtMat{a}{b}{c}{d}}_{\frac{\beta}{k},\frac{a\beta}{k}}=(-\rmi)^{\frac{1-\sgn d}{2}\cdot\legendre{c}{-1}}\cdot\legendre{\sgn{d}\cdot kb}{\abs{d}}\etp{\frac{1-\abs{d}}{8}}\etp{\frac{bd(a\beta)^2}{2k}}.
\end{equation}
Now \eqref{eq:toProveCase1} follows from the above formula, the formula for $\chi_k$ (see Lemma \ref{lemm:thetaPowerk}) and \eqref{eq:etaChar}, which concludes the proof of case (i).

\textbf{Case (ii).} The proof is similar and will only be indicated briefly. First of all, \eqref{eq:hbetakvertgamma2} still holds. The first difference is that $\tfrac{a\beta-\beta'}{k}\in D_{k}[c]^\bullet$ is equivalent to $\beta'\equiv a\beta-\tfrac{k}{2}\bmod{k}$ since now $2k\nmid c$ but $k\mid c$. Hence, \eqref{eq:hbetakvertgamma3} becomes
\begin{equation*}
h_{\beta/k}^{(k)}\vert_{\frac{k-1}{2}}\widetilde{\tbtmat{a}{b}{c}{d}}=\chi_k\widetilde{\tbtmat{a}{b}{c}{d}}\cdot\rho_{\underline{\numZ}_k}^*\widetilde{\tbtmat{a}{b}{c}{d}}_{\frac{\beta}{k},\frac{a\beta-k/2}{k}}\cdot h_{\lambda(a\beta-k/2)/k}^{(k)}.
\end{equation*}
The rest of the proof runs as case (i) with $\mathscr{G}_{k}(-cd,0)$ replaced by $\mathscr{G}_{k}(-cd,1/2)$, which again equals $1$ for $2k\nmid c$.

\textbf{Case (iii).} This is the most difficult case and will be provided with full details. From \eqref{eq:hbetakvertgammaOdd}, \eqref{eq:htsymmetry} and the fact $h_{\beta/k}^{(k)}=0$ for $2\nmid k$, $\beta\in\numZ$ we deduce that
\begin{multline}
\label{eq:hbetakvertgamma2Odd}
h_{\beta/k}^{(k)}\vert_{\frac{k-1}{2}}\widetilde{\tbtmat{a}{b}{c}{d}}=\chi_k\widetilde{\tbtmat{a}{b}{c}{d}}\cdot\sum_{\beta'\in\mathscr{B}_k}\mu_{\beta'/k}\\
\cdot\left(\rho_{\underline{2\numZ}_k}^*\widetilde{\tbtmat{a}{b}{c}{d}}_{\frac{\beta}{k},\frac{\beta'}{k}}+\rho_{\underline{2\numZ}_k}^*\widetilde{\tbtmat{a}{b}{c}{d}}_{\frac{\beta}{k},\frac{k-\beta'}{k}}+\rho_{\underline{2\numZ}_k}^*\widetilde{\tbtmat{a}{b}{c}{d}}_{\frac{\beta}{k},\frac{k+\beta'}{k}}+\rho_{\underline{2\numZ}_k}^*\widetilde{\tbtmat{a}{b}{c}{d}}_{\frac{\beta}{k},\frac{2k-\beta'}{k}}\right)\cdot h_{\beta'/k}^{(k)}.
\end{multline}
We split the proof into two subcases: $4k\mid bc$ and $4k\nmid bc$.

First assume $4k\mid bc$; then $bc\cdot\frac{1}{2k}\numZ\subseteq2\numZ$. Hence Lemma \ref{lemm:coeffWeil} is applicable to the even lattice $\underline{2\numZ}_k$. Therefore, for $\beta,\beta'\in\tfrac{1}{2}+\numZ$ we have $\rho_{\underline{2\numZ}_k}^*\widetilde{\tbtmat{a}{b}{c}{d}}_{\frac{\beta}{k},\frac{\beta'}{k}}=0$ unless $\tfrac{a\beta-\beta'}{k}\in D_{k}[c]^\bullet$, which, by the definition (see \eqref{eq:DdBullet}), is equivalent to
$$
\etp{c\cdot\frac{k}{2}\cdot\left(\frac{y}{2k}\right)^2+k\cdot\frac{a(2\beta)-(2\beta')}{2k}\cdot\frac{y}{2k}}=1,\quad\forall y\in
\begin{dcases}
\numZ &\text{if } 4k\mid c, \\
2\numZ &\text{if } 4k\nmid c.
\end{dcases}
$$
It is then straightforward that for $\beta,\beta'\in\tfrac{1}{2}+\numZ$,
\begin{equation}
\label{eq:wheninDbullet}
\tfrac{a\beta-\beta'}{k}\in D_{k}[c]^\bullet \Longleftrightarrow
\begin{dcases}
2\beta'\equiv a\cdot2\beta \pmod{4k} &\text{if } 8k\mid c, \\
2\beta'\equiv a\cdot2\beta-2k \pmod{4k} &\text{if } 4k\mid c,\,8k\nmid c,\\
2\beta'\equiv a\cdot2\beta \pmod{2k} &\text{if } 4k\nmid c.
\end{dcases}
\end{equation}
Therefore, for all $\beta\in\mathfrak{B}_k$, there is exactly one (two resp.) $\beta'\in\left\{\frac{1}{2},\frac{3}{2},\dots,\frac{4k-1}{2}\right\}$ such that $\rho_{\underline{2\numZ}_k}^*\widetilde{\tbtmat{a}{b}{c}{d}}_{\frac{\beta}{k},\frac{\beta'}{k}}$ is possibly nonzero if $4k\mid c$ ($4k\nmid c$ resp.) and its value can be calculated via \eqref{eq:rhogammayx} and \eqref{eq:rhoDual} which we now consider.

If $8k\mid c$, then
\begin{equation}
\label{eq:rho2Zkbab}
\rho_{\underline{2\numZ}_k}\widetilde{\tbtMat{a}{b}{c}{d}}_{\frac{\beta}{k},\frac{a\beta}{k}}=(-\rmi)^{\frac{1-\sgn d}{2}\cdot\legendre{c}{-1}}\cdot\mathfrak{g}_{k}\left(b,d;\frac{\beta}{k}\right)\cdot\mathscr{G}_{k}(-cd,0)
\end{equation}
where we have used the fact $\abs{D_{k}[c]}=\abs{D_{k}}$ and $\tfrac{\beta}{k}-\tfrac{ad\beta}{k}+2\numZ=0+2\numZ$ in $D_{k}$ since $4k\mid c$. By the definition we find that $\mathscr{G}_{k}(-cd,0)=1$ for $8k\mid c$. Inserting this and the $2\nmid k$ case of Lemma \ref{lemm:gbdt} into \eqref{eq:rho2Zkbab} we find that \eqref{eq:eq:rhoZkbabSimple} still holds with the left-hand side replaced by $\rho_{\underline{2\numZ}_k}\widetilde{\tbtmat{a}{b}{c}{d}}_{\frac{\beta}{k},\frac{a\beta}{k}}$. Inserting the formula obtained into \eqref{eq:hbetakvertgamma2Odd} we deduce that
\begin{multline*}
h_{\beta/k}^{(k)}\vert_{\frac{k-1}{2}}\widetilde{\tbtmat{a}{b}{c}{d}}=\chi_k\widetilde{\tbtmat{a}{b}{c}{d}}\cdot\rho_{\underline{2\numZ}_k}^*\widetilde{\tbtmat{a}{b}{c}{d}}_{\frac{\beta}{k},\frac{a\beta}{k}}\cdot h_{\lambda(a\beta)/k}^{(k)}\\
=\chi_k\widetilde{\tbtmat{a}{b}{c}{d}}\cdot\rmi^{\frac{1-\sgn d}{2}\cdot\legendre{c}{-1}}\cdot\legendre{\sgn{d}\cdot kb}{\abs{d}}\etp{\frac{\abs{d}-1}{8}}\etp{-\frac{bd(a\beta)^2}{2k}}\cdot h_{\lambda(a\beta)/k}^{(k)}.
\end{multline*}
We multiply this identity by
\begin{equation}
\label{eq:etakTransformation}
\eta^{-k}\vert_{-\frac{k}{2}}\widetilde{\tbtmat{a}{b}{c}{d}}=\chi_\eta^{-k}\widetilde{\tbtmat{a}{b}{c}{d}}\eta^{-k}
\end{equation}
and find that
\begin{multline}
\label{eq:fkbetaVertabcdTemp}
f_{k,\beta}\vert_{-\frac{1}{2}}\widetilde{\tbtmat{a}{b}{c}{d}}=\chi_\eta^{-k}\widetilde{\tbtmat{a}{b}{c}{d}}\chi_k\widetilde{\tbtmat{a}{b}{c}{d}}\\
\cdot\rmi^{\frac{1-\sgn d}{2}\cdot\legendre{c}{-1}}\cdot\legendre{\sgn{d}\cdot kb}{\abs{d}}\etp{\frac{\abs{d}-1}{8}}\etp{-\frac{bd(a\beta)^2}{2k}}\cdot f_{k,\lambda(a\beta)}.
\end{multline}
The proof for $4k\mid c,\,8k\nmid c$ is similar; \eqref{eq:fkbetaVertabcdTemp} still holds. Now we proceed to prove \eqref{eq:fkbetaVertabcdTemp} holds as well if $4k\nmid c$. According to Lemma \ref{lemm:coeffWeil} and \eqref{eq:wheninDbullet}, for all $\beta,\beta'\in\tfrac{1}{2}+\numZ$ we have
\begin{equation}
\label{eq:rho2Zkcases}
\rho_{\underline{2\numZ}_k}\widetilde{\tbtmat{a}{b}{c}{d}}_{\frac{\beta}{k},\frac{\beta'}{k}}=(-\rmi)^{\frac{1-\sgn d}{2}\legendre{c}{-1}}\mathfrak{g}_{k}\left(b,d;\tfrac{\beta}{k}\right)\sqrt{\frac{1}{2}}\cdot
\begin{dcases}
\mathscr{G}_{k}(-cd,0) &\text{if }2\beta'\equiv a\cdot2\beta\bmod{4k},\\
\mathscr{G}_{k}(-cd,1+2\numZ) &\text{if }2\beta'\equiv a\cdot2\beta-2k\bmod{4k},\\
0&\text{otherwise}.
\end{dcases}
\end{equation}
We have
\begin{equation*}
\mathscr{G}_{k}(-cd,0)=\tfrac{1}{\sqrt{2}}(1+(-\rmi)^{\frac{cd}{2k}}),\quad\mathscr{G}_{k}(-cd,1+2\numZ)=\tfrac{1}{\sqrt{2}}(1-(-\rmi)^{\frac{cd}{2k}})
\end{equation*}
for (noting that $4k\mid bc$, $2k\mid c$, $4k\nmid c$ and hence $2\mid b$)
\begin{align*}
\mathscr{G}_{k}(-cd,j+2\numZ)&=\frac{1}{\sqrt{4k\cdot2k}}\sum_{y=0}^{4k-1}\etp{-cd\cdot\frac{k}{2}\left(\frac{y}{2k}\right)^2+k\cdot\frac{jy}{2k}}\\
&=\frac{1}{\sqrt{8}k}\sum_{y_1=0}^{1}\sum_{y_2=0}^{2k-1}\etp{-cd\frac{(2y_2+y_1)^2}{8k}+\frac{j(2y_2+y_1)}{2}}\\
&=\frac{1}{\sqrt{2}}\left(1+\etp{-\frac{cd}{8k}+\frac{j}{2}}\right).\qquad(j=0,1)
\end{align*}
Inserting these two formulas and Lemma \ref{lemm:gbdt} into \eqref{eq:rho2Zkcases}, then inserting the formula obtained into \eqref{eq:hbetakvertgamma2Odd} and multiplying it by \eqref{eq:etakTransformation}, we find that \eqref{eq:fkbetaVertabcdTemp} holds for $4k\nmid c$ as well.

Now we calculate the factor $\chi_\eta^{-k}\widetilde{\tbtmat{a}{b}{c}{d}}\chi_k\widetilde{\tbtmat{a}{b}{c}{d}}$ in \eqref{eq:fkbetaVertabcdTemp}. On account of \eqref{eq:etaChar} and Lemma \ref{lemm:thetaPowerk}, we have
\begin{align*}
\chi_\eta^{-k}\widetilde{\tbtmat{a}{b}{c}{d}}\chi_k\widetilde{\tbtmat{a}{b}{c}{d}}&=\chi_\eta^{2k}\widetilde{\tbtmat{a}{b}{c}{d}}\chi_H^k([\tfrac{c}{2},\tfrac{d-1}{2}],1)\etp{-kc/8}\\
&=\etp{\frac{k(2ac-cd+2bd-2bdc^2)}{24}}.
\end{align*}
Inserting this into \eqref{eq:fkbetaVertabcdTemp} we conclude the proof for the subcase $4k\mid bc$ of case (iii).

Then we consider the subcase $4k\nmid bc$. The key observations are (according to Proposition \ref{prop:cocycleFormula})
\begin{equation*}
\widetilde{\tbtMat{a}{b}{c}{d}}=\widetilde{\tbtMat{1}{-k}{0}{1}}\cdot\widetilde{\tbtMat{a+kc}{b+kd}{c}{d}}
\end{equation*}
in $\sltZ$ and $4k\mid(b+kd)c$. It follows from these facts, the last subcase we just proved and \eqref{eq:fkbetaVertT} that
\begin{align*}
f_{k,\beta}\Big\vert_{-\frac{1}{2}}\widetilde{\tbtMat{a}{b}{c}{d}}&=f_{k,\beta}\Big\vert_{-\frac{1}{2}}\widetilde{\tbtMat{1}{-k}{0}{1}}\Big\vert_{-\frac{1}{2}}\widetilde{\tbtMat{a+kc}{b+kd}{c}{d}}\\
&=\etp{-\frac{k^2}{12}+\frac{\beta^2}{2}}p_\beta\tbtMat{a+kc}{b+kd}{c}{d}\cdot f_{k,\lambda((a+kc)\beta)}.
\end{align*}
We have $\lambda((a+kc)\beta)=\lambda(a\beta)$ since $kc\beta\in k\numZ$. Therefore we only need to show that
\begin{equation}
\label{eq:4knmidbcToshow}
p_\beta\tbtMat{a}{b}{c}{d}=\etp{-\frac{k^2}{12}+\frac{\beta^2}{2}}p_\beta\tbtMat{a+kc}{b+kd}{c}{d}.
\end{equation}
To this end, we evaluate the quotient and find
\begin{multline*}
p_\beta\tbtMat{a}{b}{c}{d}p_\beta\tbtMat{a+kc}{b+kd}{c}{d}^{-1}\\
=\etp{\frac{bd(kc\beta)^2+kd^2(a\beta)^2+kd^2(kc\beta)^2}{2k}}\cdot\etp{-\frac{k^2(c^2+d^2-c^2d^2)}{12}}.
\end{multline*}
Since $2k\mid c$, $4k\nmid bc$ and $ad-bc=1$, we have $a\equiv b\equiv d\equiv1\bmod{2}$. Therefore
\begin{equation*}
\etp{\frac{bd(kc\beta)^2+kd^2(a\beta)^2+kd^2(kc\beta)^2}{2k}}=\etp{\frac{\beta^2}{2}},\quad\etp{-\frac{k^2(c^2+d^2-c^2d^2)}{12}}=\etp{-\frac{k^2}{12}}.
\end{equation*}
Now \eqref{eq:4knmidbcToshow} follows from the above equalities, which concludes the proof of the subcase $4k\nmid bc$, hence of case (iii).

\textbf{Case (iv).} We have (according to Proposition \ref{prop:cocycleFormula})
\begin{equation*}
\widetilde{\tbtMat{a}{b}{c}{d}}=\widetilde{\tbtMat{1}{0}{\sgn{c}k}{1}}\cdot\widetilde{\tbtMat{a}{b}{c-\sgn{c}ka}{d-\sgn{c}kb}}
\end{equation*}
and $2\mid c-\sgn{c}ka$. These two facts, Corollary \ref{coro:kOddvert10k1} and case (iii) imply that
\begin{align*}
f_{k,\beta}\Big\vert_{-\frac{1}{2}}\widetilde{\tbtMat{a}{b}{c}{d}}&=f_{k,\beta}\Big\vert_{-\frac{1}{2}}\widetilde{\tbtMat{1}{0}{\sgn{c}k}{1}}\Big\vert_{-\frac{1}{2}}\widetilde{\tbtMat{a}{b}{c-\sgn{c}ka}{d-\sgn{c}kb}}\\
&=\etp{\frac{\sgn{c}k^2}{24}}p_\beta\tbtMat{a}{b}{c-\sgn{c}ka}{d-\sgn{c}kb}\cdot f_{k,\lambda(a\beta)}.
\end{align*}

\textbf{Case (v).} We start from the identity (according to Proposition \ref{prop:cocycleFormula})
\begin{equation*}
\widetilde{\tbtMat{a}{b}{c}{d}}=\widetilde{\tbtMat{1}{-1}{0}{1}}\cdot\widetilde{\tbtMat{1}{0}{\sgn{c}k}{1}}\cdot\widetilde{\tbtMat{a_1}{b_1}{c_1}{d_1}}
\end{equation*}
where $a_1,b_1,c_1,d_1$ have been given in \eqref{eq:a1b1c1d1}. Note that $2\mid c_1$. Repeating the argument of case (iv) leads to
\begin{equation*}
f_{k,\beta}\Big\vert_{-\frac{1}{2}}\widetilde{\tbtMat{a}{b}{c}{d}}=\etp{-\frac{k}{12}+\frac{\beta^2}{2k}+\frac{\sgn{c}k^2}{24}}p_{\beta}\tbtMat{a_1}{b_1}{c_1}{d_1}\cdot f_{k,\lambda((a+c)\beta)}.
\end{equation*}
Since $c\beta-(-\tfrac{k}{2})\in k\numZ$ we can replace $f_{k,\lambda((a+c)\beta)}$ with $f_{k,\lambda(a\beta-k/2)}$ which gives the desired identity.
\end{proof}

Let $\beta,\beta'\in\mathfrak{B}_k$. We write $\beta\sim\beta'$ if there is a matrix $\tbtmat{a}{b}{c}{d}\in\Gamma_0(k)$ such that $\beta'=T_\beta \tbtmat{a}{b}{c}{d}$ where $T_\beta$ is defined by \eqref{eq:betapFormula}.

\begin{prop}
\label{prop:EquivalenceRelation}
$\beta\sim\beta'$ if and only if $(2\beta,k)=(2\beta',k)$, which indicates that $\sim$ is an equivalence relation.
\end{prop}

\begin{proof}
Suppose that $(2\beta,k)=(2\beta',k)$. By analyzing the 2-adic order, one can prove that either $(2\beta,2k)=(2\beta',2k)$ or $(2\beta,2k)=(2\beta'+k,2k)$ holds. If $(2\beta,2k)=(2\beta',2k)$, then there exists an $a$ coprime to $2k$ such that $2\beta'\equiv a\cdot 2\beta\pmod{2k}$. We choose any $b,d\in\numZ$ such that $ad-2bk=1$. Noting that $a$ is odd, we have $T_\beta \tbtmat{a}{b}{2k}{d}=\lambda(a\beta)=\beta'$. If $(2\beta,2k)=(2\beta'+k,2k)$, since this case only occurs for even $k$, then there exists an $a$ coprime to $2k$ such that $2\beta'+k\equiv a\cdot 2\beta\pmod{2k}$. We choose any $b,d\in\numZ$ such that $ad-bk=1$, then $T_\beta \tbtmat{a}{b}{k}{d}=\lambda(a\beta-k/2)=\beta'$.

Conversely, suppose that $\beta'=T_\beta \tbtmat{a}{b}{c}{d}$. Whatever $\beta'=\lambda(a\beta)$ or $\beta'=\lambda(a\beta-k/2)$, we have $2\beta'\equiv \pm 2a\beta \pmod{k}$ and then $(2\beta,k)=(2\beta',k)$ since $(a,k)=1$.
\end{proof}

The equivalence classes of $\mathfrak{B}_k$ given by $\sim$ are collected in the table below for small $k$.
\begin{longtable}{llll}
\caption{Equivalence classes of $\mathfrak{B}_k$\label{table:BkPartitions}} \\
\toprule
$k$ & classes of $2\cdot\mathfrak{B}_k$ & $k$ & classes of $\mathfrak{B}_k$ \\
\midrule
\endfirsthead
$k$ & classes of $2\cdot\mathfrak{B}_k$ & $k$ & classes of $\mathfrak{B}_k$ \\
\midrule
\endhead
$1$ & $\{1\}$ & $2$ & $\{0,1\}$\\
$3$ & $\{1\}$, $\{3\}$ & $4$ & $\{0,2\}$, $\{1\}$\\
$5$ & $\{1,3\}$, $\{5\}$ & $6$ & $\{0,3\}$, $\{1,2\}$\\
$7$ & $\{1,3,5\}$, $\{7\}$ & $8$ & $\{0,4\}$, $\{1,3\}$, $\{2\}$\\
$9$ & $\{1,5,7\}$, $\{3\}$, $\{9\}$ & $10$ & $\{0,5\}$, $\{1,2,3,4\}$\\
$11$ & $\{1,3,5,7,9\}$, $\{11\}$ & $12$ & $\{0,6\}$, $\{1,5\}$, $\{2,4\}$, $\{3\}$\\
$13$ & $\{1,3,5,7,9,11\}$, $\{13\}$ & $14$ & $\{0,7\}$, $\{1,2,3,4,5,6\}$\\
\bottomrule
\end{longtable}

\begin{coro}
For $k$ even, Andrews' generating function $C\Phi_k$ (i.e. $C\Psi_{k,k/2}$) and $C\Psi_{k,0}$ can be transformed to each other by $\tbtmat{1}{0}{k}{1}$, whereas for $k$ odd, $C\Phi_k$ cannot be transformed from any other $C\Psi_{k,\beta}$ by matrices in $\Gamma_0(k)$.
\end{coro}

Now we consider the quotient
$$
\frac{f_{k,\beta}(q)}{f_{k,\beta}(q^N)}=q^{(1-N)(\frac{k}{12}-\frac{\beta^2}{2k})}\frac{C\Psi_{k,\beta}(q)}{C\Psi_{k,\beta}(q^N)},
$$
which plays an important role in deriving congruences like \eqref{r:cphi2mod5a}. Its transformation equations are direct consequences of Theorem \ref{thm:fkbetaGamma0k}:

\begin{prop}
\label{prop:fqdivfqN}
Let $k,N\in\numgeq{Z}{1}$, $\beta\in\mathfrak{B}_k$. If $2\mid k$ we require that $2\nmid N$. Let $\tbtmat{a}{b}{c}{d}\in\Gamma_0(N\cdot\lcm(2,k))$. Then we have
\begin{equation*}
\left.\frac{f_{k,\beta}(\tau)}{f_{k,\beta}(N\tau)}\right\vert\tbtMat{a}{b}{c}{d}=\legendre{N}{\abs{d}}\cdot\etp{\frac{(N-1)k}{24}\left(\frac{(2a-d)c}{N}-2bd\left(1+\frac{c^2}{N}\right)\right)+\frac{(N-1)bd(a\beta)^2}{2k}}\frac{f_{k,\beta'}(\tau)}{f_{k,\beta'}(N\tau)}
\end{equation*}
where $\beta'=T_\beta \tbtmat{a}{b}{c}{d}$ (see \eqref{eq:betapFormula}).
\end{prop}
\begin{proof}
We have $\left.\frac{f_{k,\beta}(\tau)}{f_{k,\beta}(N\tau)}\right\vert\tbtmat{a}{b}{c}{d}=f_{k,\beta}\vert_{-\frac{1}{2}}\tbtmat{a}{b}{c}{d}(\tau)/f_{k,\beta}\vert_{-\frac{1}{2}}\tbtmat{a}{Nb}{c/N}{d}(N\tau)$, so we apply Theorem \ref{thm:fkbetaGamma0k} twice and obtain the desired formula. The assumption $2\nmid N$ in the case $2\mid k$ ensures the two occurrences of $\beta'$ in the numerator and denominator of the right-hand side coincide.
\end{proof}

\section{Systems of $U_p$-sequences}
\label{r:sec5}

\subsection{Commutativity of $U_m$ operators and modular transformations}
In this subsection we derive some identities of the type $(U_mf)\vert\gamma=U_m(f\vert\gamma')$ in broad generalities.
\begin{deff}
Let $m\in\numgeq{Z}{1}$ and let $R$ be a complete set of representatives of $\numZ/m\numZ$. Set
\begin{equation*}
U_m^R(f):=\frac{1}{m}\sum_{x\in R}f\left(\frac{\tau+x}{m}\right).
\end{equation*}
\end{deff}
For $k\in\frac{1}{2}\numZ$ we have, by the definition, $U_m^R(f)=m^{\frac{k}{2}-1}\sum_{x\in R}f\vert_k\tbtmat{1}{x}{0}{m}$, which gives us some convenience since it translates some transformation problems into $2\times2$ matrix algebra. If $R=\{0,1,\dots,m-1\}$ then $U_m^R$ is the classical Atkin-Lehner $U_m$ operator. Other choice of $R$, say $R=\{24\cdot0,24\cdot1,\dots,24\cdot(m-1)\}$, gives us flexibility to deal with various multiplier systems, e.g., the one in Proposition \ref{prop:fqdivfqN}.

In the rest of this subsection, $m$ and $R$ are fixed as above.
\begin{lemm}
\label{lemm:sigmaBijective}
Let $\gamma=\tbtmat{a}{b}{c}{d}\in\Gamma_0(m)$. Then for any $x\in R$, there exists a unique $y\in R$ such that $(a+cx)y\equiv b+dx\bmod{m}$. Moreover, the assignment $x\mapsto y$ gives a bijection on $R$.
\end{lemm}
\begin{proof}
The existence and uniqueness of $y\in R$ follow from the fact $(a+cx,m)=1$ since $m\mid c$ and $(a,m)=1$. Hence $x\mapsto y$ is a well-defined map from $R$ to itself. To prove the bijectivity, it suffices to prove the injectivity since $R$ is a finite set. To this end, suppose $x_1,x_2\in R$ with the same image $y$. Then $(a+cx_j)y\equiv b+dx_j\bmod{m}$ for $j=1,2$. Subtracting them gives $(cy-d)(x_1-x_2)\equiv0\bmod{m}$. Since $(cy-d,m)=1$ it follows that $x_1\equiv x_2\bmod{m}$ and hence $x_1=x_2$ for $R$ is a complete set of representatives of $\numZ/m\numZ$.
\end{proof}
The solution $y$ is denoted by $\sigma_{\gamma}^R(x)$, $\sigma_{a,b,c,d}^R(x)$ or just $\sigma(x)$ below. Moreover, set $\sigma'(x)=\sigma(x)-\sigma(0)$ provided $0\in R$.

\begin{deff}
Suppose that $0\in R$. For $\tbtmat{a}{b}{c}{d}\in\Gamma_0(m)$, set
\begin{equation*}
C_{m}^R\tbtMat{a}{b}{c}{d}=\tbtMat{a}{m^{-1}(b-a\sigma(0))}{mc}{d-c\sigma(0)}=\tbtMat{1}{0}{0}{m}\tbtMat{a}{b}{c}{d}\tbtMat{1}{\sigma(0)}{0}{m}^{-1}.
\end{equation*}
\end{deff}

The major technical fact about the ``commutativity'' between $U_m^R$ and modular transformations is the following one.
\begin{lemm}
\label{lemm:commUm}
Let $k\in\frac{1}{2}\numZ$, $\gamma=\tbtmat{a}{b}{c}{d}\in\Gamma_0(m)$. Suppose $0\in R$. Set
\begin{align*}
a_1&=a_1(x)=1+c(a+cx)\sigma'(x), &b_1&=b_1(x)=\frac{x-a(a+cx)\sigma'(x)}{m},\\
c_1&=c_1(x)=mc^2\sigma'(x), &d_1&=d_1(x)=1-ac\sigma'(x).
\end{align*}
If a function $f$ satisfies $f\vert_k\tbtmat{a_1}{b_1}{c_1}{d_1}=f$ for all $x\in R$, then
\begin{equation*}
(U_m^Rf)\vert_k\gamma=U_m^R(f\vert_k C_m^R\gamma).
\end{equation*}
\end{lemm}
(To see $b_1\in\numZ$, one expands the definition of $\sigma(x)$ and $\sigma(0)$ and uses the fact $ad-bc=1$.)
\begin{proof}
A straightforward calculation, taking into account Proposition \ref{prop:cocycleFormula}, shows that in $\sltR$
\begin{equation*}
\widetilde{\tbtMat{a_1}{b_1}{c_1}{d_1}}\cdot\widetilde{C_{m}^R\tbtMat{a}{b}{c}{d}}=\widetilde{\tbtMat{a+cx}{r(x)}{mc}{d-c\sigma(x)}}
\end{equation*}
where $x\in R$ and $r(x)=\tfrac{(b+dx)-(a+cx)\sigma(x)}{m}$. Since $f\vert_k\tbtmat{a_1}{b_1}{c_1}{d_1}=f$ for all $x\in R$, it follows that
$$
f\Big\vert_k\tbtMat{a+cx}{r(x)}{mc}{d-c\sigma(x)}=f\Big\vert_k C_m^R\tbtMat{a}{b}{c}{d}
$$
for all $x\in R$. Therefore,
\begin{align*}
(U_m^Rf)\vert_k\gamma&=m^{\frac{k}{2}-1}\sum_{x\in R}f\vert_k\widetilde{\tbtmat{1}{x}{0}{m}}\widetilde{\tbtmat{a}{b}{c}{d}}=m^{\frac{k}{2}-1}\sum_{x\in R}f\vert_k\widetilde{\tbtmat{a+cx}{r(x)}{mc}{d-c\sigma(x)}}\widetilde{\tbtmat{1}{\sigma(x)}{0}{m}}\\
&=m^{\frac{k}{2}-1}\sum_{x\in R}(f\vert_k C_m^R\gamma)\vert_k\tbtmat{1}{\sigma(x)}{0}{m}=U_m^R(f\vert_k C_m^R\gamma).
\end{align*}
In the last step we have used the fact $x\mapsto \sigma(x)$ is a bijection; see Lemma \ref{lemm:sigmaBijective}.
\end{proof}
\begin{examp}
Set
\begin{equation*}
G=\{\tbtmat{A}{B}{C}{D}\in\slZ\colon A\equiv D\equiv1\bmod{24\abs{c}},\,C\equiv0\bmod{24mc^2},B\equiv0\bmod{24}\}.
\end{equation*}
Then $G$ is a subgroup of $\slZ$ and is contained in $\Gamma_0(24mc^2)$. Suppose that $0\in R\subseteq24\numZ$ and $m$ is coprime to $6$. If $f$ is a modular function on $G$, that is, $f\left(\frac{A\tau+B}{C\tau+D}\right)=f(\tau)$ for any $\tbtmat{A}{B}{C}{D}\in G$, then Lemma \ref{lemm:commUm} is applicable to $f$ (with $k=0$) since $\tbtmat{a_1}{b_1}{c_1}{d_1}\in G$. Moreover, repeated applications of the lemma give
\begin{equation*}
(U_m^Rf)\vert(C_m^R)^r\gamma=U_m^R(f\vert(C_m^R)^{r+1}\gamma),\quad r\in\numgeq{Z}{0}.
\end{equation*}
\end{examp}

Besides, we will meet weakly holomorphic modular forms of weight $k=-1/2$ to which Lemma \ref{lemm:commUm} is applicable in Sections \ref{r:sec6}. Thus, Lemma \ref{lemm:commUm} is a very general tool.

It may happen that $f\vert_k(C_m^R)^{r}\gamma=f$ for sufficiently large $r$. The following lemma concerns the case $r=2$.
\begin{lemm}
\label{lemm:CmR2}
Suppose $0\in R$. Let $\tbtmat{a}{b}{c}{d}\in\Gamma_0(m)$. Let $y,z$ be the unique elements in $R$ satisfying
\begin{equation*}
ay\equiv b\bmod{m},\qquad az\equiv\frac{b-ay}{m}\bmod{m}.
\end{equation*}
Then in the metaplectic group $\sltZ$ we have $\widetilde{(C_m^R)^2\tbtmat{a}{b}{c}{d}}=\widetilde{\tbtmat{a'}{b'}{c'}{d'}}\cdot\widetilde{\tbtmat{a}{b}{c}{d}}$ where
\begin{align*}
a'&=1+\frac{(m^2-1)bc+ac(y+mz)}{m^2}, &b'&=-\frac{(m^2-1)ab+a^2(y+mz)}{m^2},\\
c'&=(m^2-1)cd+c^2(y+mz), &d'&=1-(m^2-1)bc-ac(y+mz).
\end{align*}
Moreover, if $m$ is coprime to $6$ and $R\subseteq24\numZ$, then $((C_m^R)^2\tbtmat{a}{b}{c}{d})\cdot\tbtmat{a}{b}{c}{d}^{-1}$ is contained in the group
\begin{equation*}
\Gamma_1^*(24\abs{c}):=\{\tbtmat{A}{B}{C}{D}\in\slZ\colon A\equiv D\equiv1\bmod{24\abs{c}},\,C\equiv0\bmod{24\abs{c}},B\equiv0\bmod{24}\}.
\end{equation*}
\end{lemm}
\begin{proof}
A straightforward verification with the aid of Proposition \ref{prop:cocycleFormula}.
\end{proof}

Combining the general Lemmas \ref{lemm:commUm} and \ref{lemm:CmR2} we obtain an easy-to-use criterion on repeatedly commuting $U_m^R$ and modular transformations. It should be noted that we assume $k\in\numZ$ since for half-integral weights the multiplier system cannot be trivial on the whole metaplectic group.
\begin{prop}
\label{prop:commute2}
Suppose $m$ is coprime to $6$ and $0\in R\subseteq24\numZ$. Let $k\in\numZ$ and $M=\tbtmat{a}{b}{c}{d}\in\Gamma_0(m)$. Let $f\colon\uhp\rightarrow\numC\cup\{\infty\}$ satisfy $f\vert_k\gamma=f$ for all $\gamma\in\Gamma_1^*(24\abs{c})$. Then we have
\begin{equation}
\label{eq:commute2}
(U_m^Rf)\vert_kM=U_m^R(f\vert_k C_m^RM),\qquad(U_m^Rf)\vert_kC_m^RM=U_m^R(f\vert_kM).
\end{equation}
\end{prop}
\begin{proof}
For any fixed $x\in R$ let $a_1,b_1,c_1,d_1$ (depending on $x$) be as in Lemma \ref{lemm:commUm}. Then we have $\tbtmat{a_1}{b_1}{c_1}{d_1}\in\Gamma_1^*(24\abs{c})$ since $R\subseteq24\numZ$ and $(m,6)=1$. Thus, $f\vert_k\tbtmat{a_1}{b_1}{c_1}{d_1}=f$ for any $x\in R$ by the condition we require on $f$. The former formula of \eqref{eq:commute2} is then the conclusion of Lemma \ref{lemm:commUm}. In a similar manner, we find that
\begin{equation}
\label{eq:UmRfTemp}
(U_m^Rf)\vert_kC_m^RM=U_m^R(f\vert_k(C_m^R)^2M).
\end{equation}
By the last conclusion of Lemma \ref{lemm:CmR2} we have $((C_m^R)^2M)\cdot M^{-1}\in\Gamma_1^*(24\abs{c})$, and hence $f\vert_k(C_m^R)^2M=f\vert_kM$. Inserting this into \eqref{eq:UmRfTemp} we obtain the latter formula of \eqref{eq:commute2}.
\end{proof}
\begin{rema}
We call the two formulas in Proposition \ref{prop:commute2} a \emph{system of period $2$}. By similar arguments one can find criteria of systems of any period $r\in\numgeq{Z}{1}$. In principle, the larger $r$ is, the higher the level of $f$ should be. The case $r=1$ is used by Garvan, Sellers and Smoot \cite{GSS24} to deal with $C\Psi_{2,\beta}$. It turns out that a system of period $1$ is not enough for $C\Psi_{3,\beta}$; see the next proposition. Thus (the half-integral variant of) Proposition \ref{prop:commute2} is a necessary tool.
\end{rema}

For the reader's convenience, we provide a period $1$ version of Proposition \ref{prop:commute2} in its most general form.
\begin{prop}
\label{prop:commute1}
Let $m$, $N$ be coprime positive integers and let $0\in R\subseteq N\numZ$. Let $k\in\numZ$ and $M=\tbtmat{a}{b}{c}{d}\in\Gamma_0(m)$. Set
$$N_0=(N,m-1),\quad N_1=(N,(m-1)b).$$
Let $f\colon\uhp\rightarrow\numC\cup\{\infty\}$ satisfy $f\vert_k\gamma=f$ for all $\gamma\in\Gamma_1^*(N_0\abs{c},N_1)$ (see \eqref{eq:Gamma1starNM}).
Then we have
\begin{equation*}
(U_m^Rf)\vert_kM=U_m^R(f\vert_k M).
\end{equation*}
\end{prop}
\begin{proof}
Lemma \ref{lemm:commUm} is still applicable since $\tbtmat{a_1}{b_1}{c_1}{d_1}\in\Gamma_1^*(N_0\abs{c},N_1)$. Hence $(U_m^Rf)\vert_kM=U_m^R(f\vert_k C_m^RM)$. A direct calculation shows that $(C_m^RM)\cdot M^{-1}\in\Gamma_1^*(N_0\abs{c},N_1)$, so $f\vert_k C_m^RM=f\vert_kM$. The conclusion now follows from these two formulas.
\end{proof}

To conclude this subsection, we generalize Proposition \ref{prop:commute2} a bit, allowing $M$ taking values in the group algebra (see Definitions \ref{def:groupAlgebra1} and \ref{def:groupAlgebra2}).

For $0\neq M=\sum_{j}x_j\cdot\left(\tbtmat{a_j}{b_j}{c_j}{d_j},\varepsilon_j\right)\in\mathscr{A}_{\numZ}$ with $\left(\tbtmat{a_j}{b_j}{c_j}{d_j},\varepsilon_j\right)$ different and $x_j\in\numC\setminus\{0\}$, we define the \emph{content} of $M$ to be the (nonnegative) greatest common divisor of all $c_j$. Moreover, set $C_m^RM=\sum_{j}x_j\cdot\left(C_m^R\tbtmat{a_j}{b_j}{c_j}{d_j},\varepsilon_j\right)$.
\begin{prop}
\label{prop:commute2alg}
Suppose $m$ is coprime to $6$ and $0\in R\subseteq24\numZ$. Let $k\in\numZ$ and let $M\in\mathscr{A}_m\setminus\{0\}$ whose content is $l\in\numgeq{Z}{0}$. Let $f\colon\uhp\rightarrow\numC$ satisfy $f\vert_k\gamma=f$ for all $\gamma\in\Gamma_1^*(24l)$. Then \eqref{eq:commute2} remains true.
\end{prop}
\begin{proof}
Set $M=\sum_{j}x_j\cdot\left(\tbtmat{a_j}{b_j}{c_j}{d_j},\varepsilon_j\right)$ with $\left(\tbtmat{a_j}{b_j}{c_j}{d_j},\varepsilon_j\right)\in\widetilde{\Gamma_0(m)}$ different and $x_j\in\numC\setminus\{0\}$. Then $l\mid c_j$ for all $j$. Consequently, $f\vert_k\gamma=f$ for all $\gamma\in\Gamma_1^*(24\abs{c_j})$ and for all $j$. Application of Proposition \ref{prop:commute2} to $\tbtmat{a}{b}{c}{d}=\tbtmat{a_j}{b_j}{c_j}{d_j}$ gives
\begin{equation*}
(U_m^Rf)\Big\vert_k\left(\tbtmat{a_j}{b_j}{c_j}{d_j},\varepsilon_j\right)=U_m^R\left(f\Big\vert_k \left(C_m^R\tbtmat{a_j}{b_j}{c_j}{d_j},\varepsilon_j\right)\right).
\end{equation*}
Multiplying this by $x_j$, then adding them together for all $j$, noting the definition of the action of $\mathscr{A}$ on modular forms and using the linearity of $U_m^R$ we obtain the first formula of \eqref{eq:commute2}. The second formula of \eqref{eq:commute2} is proved similarly.
\end{proof}
There is no elegant criterion like Proposition \ref{prop:commute2alg} for half-integral weights. Nevertheless, the following seemingly artificial criterion turns out successful in dealing with $C\Psi_{3,\beta}$.
\begin{prop}
\label{prop:commute2algHalf}
Suppose $m$ is coprime to $6$ and $0\in R\subseteq24\numZ$. Let $k_1\in\frac{1}{2}\numZ$ and $k_2\in\numZ$. Let $M=\sum_{1\leq j\leq J}x_j\widetilde{\tbtmat{a_j}{b_j}{c_j}{d_j}}\in\mathscr{A}_m\setminus\{0\}$ whose content is $l\in\numgeq{Z}{0}$. Set $\gamma_j=\tbtmat{a_j}{b_j}{c_j}{d_j}$, $y_j=\sigma_{\gamma_j}^R(0)$ and $z_j=\sigma_{C_m^R\gamma_j}^R(0)$. Let $f\colon\uhp\rightarrow\numC$ be a function satisfying
\begin{gather}
f\left\vert_{k_1}\tbtMat{1+c_j(a_j+c_jx)\sigma_{\gamma_j}'(x)}{m^{-1}(x-a_j(a_j+c_jx)\sigma_{\gamma_j}'(x))}{mc_j^2\sigma_{\gamma_j}'(x)}{1-a_jc_j\sigma_{\gamma_j}'(x)}\right.=f,\label{eq:halfCond1}\\
f\left\vert_{k_1}\tbtMat{1+mc_j(a_j+mc_jx)\sigma_{C_m^R\gamma_j}'(x)}{m^{-1}(x-a_j(a_j+mc_jx)\sigma_{C_m^R\gamma_j}'(x))}{m^3c_j^2\sigma_{C_m^R\gamma_j}'(x)}{1-ma_jc_j\sigma_{C_m^R\gamma_j}'(x)}\right.=f,\label{eq:halfCond1C}\\
f\left\vert_{k_1}\tbtMat{1+m^{-2}((m^2-1)b_jc_j+a_jc_j(y_j+mz_j))}{-m^{-2}((m^2-1)a_jb_j+a_j^2(y_j+mz_j))}{(m^2-1)c_jd_j+c_j^2(y_j+mz_j)}{1-(m^2-1)b_jc_j-a_jc_j(y_j+mz_j)}\right.=f\label{eq:halfCond2}
\end{gather}
for all $1\leq j\leq J$ and all $x\in R$. Let $g\colon\uhp\rightarrow\numC$ be a function transforming like a modular form on $\Gamma_1^*(24l)$ of weight $k_2$ with trivial multiplier system. Then
\begin{equation}
\label{eq:halfCommConclusion1}
(U_m^R(fg))\vert_{k_1+k_2}M=U_m^R((fg)\vert_{k_1+k_2} C_m^RM),\qquad(U_m^R(fg))\vert_{k_1+k_2}C_m^RM=U_m^R((fg)\vert_{k_1+k_2}M).
\end{equation}
Moreover, if $f$ is only required to satisfy \eqref{eq:halfCond1}, then the former conclusion of \eqref{eq:halfCommConclusion1} still holds.
\end{prop}
\begin{proof}
Let $A_j$, $B_j$ and $C_j$ denote the matrices in \eqref{eq:halfCond1}, \eqref{eq:halfCond1C} and \eqref{eq:halfCond2} respectively. (Note $A_j$ and $B_j$ depend on $x\in R$.) Since $(m,6)=1$ and $0\in R\subseteq24\numZ$ we have $A_j,B_j,C_j\in\Gamma_1^*(24l)$. Moreover, it follows from Lemma \ref{lemm:CmR2} that $\widetilde{(C_m^R)^2\gamma_j}=\widetilde{C_j}\cdot\widetilde{\gamma_j}$.

By \eqref{eq:halfCond1} and the assumption on $g$ we have $(fg)\vert_{k_1+k_2}\widetilde{A_j}=fg$ for all $j$ and $x$. Then Lemma \ref{lemm:commUm} implies that $(U_m^R(fg))\vert_{k_1+k_2}\widetilde{\gamma_j}=U_m^R((fg)\vert_{k_1+k_2} \widetilde{C_m^R\gamma_j)}$ for all $j$. A summation of this formula over $j$ as in the proof of Proposition \ref{prop:commute2alg} gives the former conclusion of \eqref{eq:halfCommConclusion1}. A similar argument where we use \eqref{eq:halfCond1C} instead of \eqref{eq:halfCond1} gives
\begin{equation}
\label{eq:halfCommConclusion2}
(U_m^R(fg))\vert_{k_1+k_2}C_m^RM=U_m^R((fg)\vert_{k_1+k_2}(C_m^R)^2M).
\end{equation}
Now \eqref{eq:halfCond2} and the assumption on $g$ imply that $(fg)\vert_{k_1+k_2}\widetilde{C_j}=fg$ for all $j$. This, together with the fact $\widetilde{(C_m^R)^2\gamma_j}=\widetilde{C_j}\cdot\widetilde{\gamma_j}$, gives $(fg)\vert_{k_1+k_2}(C_m^R)^2M=(fg)\vert_{k_1+k_2}M$. Inserting it into \eqref{eq:halfCommConclusion2} we obtain the latter conclusion of \eqref{eq:halfCommConclusion1}.
\end{proof}

\begin{rema}
Logically, the assertions in Proposition \ref{prop:commute2algHalf} are equivalent to the seemingly weaker ones with $g$ replaced by $1$, since one can just let $f$ be $fg$. We need the form presented in Proposition \ref{prop:commute2algHalf} because in the applications below (e.g. Proposition \ref{r:prom01}), $U_m^R$ will act on a product of two functions $fg$, with $f$ fixed and $g$ varying. See also Example \ref{examp:commute2algHalf}.
\end{rema}

\begin{rema}
We explain here why the neat Proposition \ref{prop:commute2} is useless for half-integral weights. First note that if we replace the condition $k\in\numZ$ in Proposition \ref{prop:commute2} with $k\in\tfrac{1}{2}\numZ$, then the conclusion still holds as the reader can check through the proof (since Lemma \ref{lemm:commUm} allows $k\in\tfrac{1}{2}\numZ$ and Lemma \ref{lemm:CmR2} is an identity in the metaplectic group). However, if $k\in\tfrac{1}{2}\numZ\setminus\numZ$, then there exists \emph{no} nonzero $f$ such that $f\vert_k\gamma=f$ for all $\gamma\in\Gamma_1^*(24\abs{c})$, which makes Proposition \ref{prop:commute2} be vacuously true. To see why there exists no such $f$, we assume that $f$ satisfies the condition mentioned. We can find three matrices $\tbtmat{a_j}{b_j}{c_j}{d_j}$ in $\Gamma_1^*(24l)$ such that $\tbtmat{a_1}{b_1}{c_1}{d_1}\tbtmat{a_2}{b_2}{c_2}{d_2}=\tbtmat{a_3}{b_3}{c_3}{d_3}$ with $c_1,c_2>0$ but $c_3=c_1a_2+d_1c_2<0$. Lifting the equality to $\widetilde{\slZ}$ we have $\widetilde{\tbtmat{a_1}{b_1}{c_1}{d_1}}\widetilde{\tbtmat{a_2}{b_2}{c_2}{d_2}}=\left(\tbtmat{a_3}{b_3}{c_3}{d_3},-1\right)$; see Proposition \ref{prop:cocycleFormula}. Now $f\vert_k\left(\tbtmat{a_3}{b_3}{c_3}{d_3},-1\right)=f\vert_k\widetilde{\tbtmat{a_1}{b_1}{c_1}{d_1}}\widetilde{\tbtmat{a_2}{b_2}{c_2}{d_2}}=f$. On the other hand, we have $f\vert_k\left(\tbtmat{a_3}{b_3}{c_3}{d_3},-1\right)=-f\vert_k\widetilde{\tbtmat{a_3}{b_3}{c_3}{d_3}}=-f$. Therefore $f=-f$, that is, $f=0$. Thus, to deal with half-integral weights, e.g. $C\Psi_{k,\beta}$, we have to apply Proposition \ref{prop:commute2algHalf}. As the next example shows, there indeed exist nontrivial $f$ and $g$ that it is applicable to.
\end{rema}

\begin{examp}
\label{examp:commute2algHalf}
In Proposition \ref{prop:commute2algHalf} set $m=5$, $R=\{0,24,48,72,96\}$, $J=1$ and $M=\widetilde{\gamma_1}$ with $\gamma_1=\tbtmat{1}{0}{10}{1}$. Then $C_5^RM=\widetilde{\tbtmat{1}{0}{50}{1}}$ and $\sigma_{\gamma_1}(x)=\sigma_{C_5^R\gamma_1}(x)=x$ for all $x\in R$. Thus, the matrices that appear in \eqref{eq:halfCond1} are
\begin{equation*}
\tbtMat{1}{0}{0}{1},\tbtMat{57841}{-1152}{12000}{-239},\tbtMat{230881}{-4608}{24000}{-479},\tbtMat{519121}{-10368}{36000}{-719},\tbtMat{922561}{-18432}{48000}{-959},
\end{equation*}
those in \eqref{eq:halfCond1C} are
\begin{equation*}
\tbtMat{1}{0}{0}{1},\tbtMat{1441201}{-5760}{300000}{-1199},\tbtMat{5762401}{-23040}{600000}{-2399},\tbtMat{12963601}{-51840}{900000}{-3599},\tbtMat{23044801}{-92160}{1200000}{-4799},
\end{equation*}
and the one in \eqref{eq:halfCond2} is
\begin{equation*}
\tbtMat{1}{0}{240}{1}.
\end{equation*}
For $f=f_{k,\beta}=q^{\frac{k}{12}-\frac{\beta^2}{2k}}\cdot C\Psi_{k,\beta}$ where $k=1,2,3,4,6,8,12$ and $\beta\in\mathfrak{B}_k$ or $(k,\beta)=(5,5/2),(10,0),(10,5)$ we have, by Theorem \ref{thm:fkbetaGamma0k}, $f\vert_{-1/2}A=f$ where $A$ is arbitrary one of the above matrices. Therefore, Proposition \ref{prop:commute2algHalf} shows that
\begin{align*}
\left(U_5^R\left(q^{\frac{k}{12}-\frac{\beta^2}{2k}}\cdot C\Psi_{k,\beta}\cdot g\right)\right)\Big\vert_{k_2-1/2}\tbtMat{1}{0}{10}{1}&=U_5^R\left(\left(q^{\frac{k}{12}-\frac{\beta^2}{2k}}\cdot C\Psi_{k,\beta}\cdot g\right)\Big\vert_{k_2-1/2}\tbtMat{1}{0}{50}{1}\right),\\
\left(U_5^R\left(q^{\frac{k}{12}-\frac{\beta^2}{2k}}\cdot C\Psi_{k,\beta}\cdot g\right)\right)\Big\vert_{k_2-1/2}\tbtMat{1}{0}{50}{1}&=U_5^R\left(\left(q^{\frac{k}{12}-\frac{\beta^2}{2k}}\cdot C\Psi_{k,\beta}\cdot g\right)\Big\vert_{k_2-1/2}\tbtMat{1}{0}{10}{1}\right)
\end{align*}
where $k,\beta$ are as above and $g$ is a function satisfying $g\vert_{k_2}\gamma=g$ for all $\gamma\in\Gamma_1^*(240)$. The method demonstrated in this example plays a vital role in transforming \eqref{r:c31cong} to \eqref{r:c33cong}. Finally note that if we want a stronger conclusion like
\begin{equation*}
\left(U_5^R\left(q^{\frac{k}{12}-\frac{\beta^2}{2k}}\cdot C\Psi_{k,\beta}\cdot g\right)\right)\Big\vert_{k_2-1/2}\tbtMat{1}{0}{10}{1}=U_5^R\left(\left(q^{\frac{k}{12}-\frac{\beta^2}{2k}}\cdot C\Psi_{k,\beta}\cdot g\right)\Big\vert_{k_2-1/2}\tbtMat{1}{0}{10}{1}\right),
\end{equation*}
then whatever $R$ we choose, $g$ should satisfy $g\vert_{k_2}\gamma=g$ \emph{at least} (not necessarily sufficient) for all $\gamma\in\Gamma_1(40)$, which is too restrictive for our purpose.

\end{examp}

\subsection{Proof of Theorem \ref{r:thecpsimodfun}}
\label{subsec:proof_of_theorem_thecpsimodfun}
The proof rests on the following fact, which is an immediate consequence of Proposition \ref{prop:fqdivfqN}.
\begin{lemm}
\label{lemm:p2modularFunction}
In Proposition \ref{prop:fqdivfqN}, if $N=p^2$ where $p\geq5$ is a prime, then
\begin{equation*}
\frac{f_{k,\beta}(\tau)}{f_{k,\beta}(p^2\tau)}\Bigg\vert\tbtMat{a}{b}{c}{d}=\etp{\frac{(2\beta)^2(p^2-1)bda^2}{8k}}\frac{f_{k,\beta'}(\tau)}{f_{k,\beta'}(p^2\tau)}.
\end{equation*}
(Recall that $\tbtmat{a}{b}{c}{d}$ is assumed in $\Gamma_0(p^2\lcm(2,k))$.)
\end{lemm}

For the reader's convenience, we recall:
\begin{thm1.3*}
Let $p\geq 5$ be prime and $\beta,\beta'\in \mathfrak{B}_k$ with $(2\beta,k)=(2\beta',k)$. Set
\begin{equation}
\label{eq:rre}
r=\frac{k}{(k,(2\beta)^2(p^2-1)/8)},\quad r_e=\lcm(2,r).
\end{equation}
Then there exists $\gamma\in \Gamma_0^0(p^2k,r)$ (see \eqref{eq:Gamma00NM}) such that
\begin{equation}
\label{eq:thecpsimodfun}
\frac{f_{k,\beta}(\tau)}{f_{k,\beta}(p^2\tau)}\bigg| \gamma=\frac{f_{k,\beta'}(\tau)}{f_{k,\beta'}(p^2\tau)}.
\end{equation}
Moreover, if $(r,p)=1$, then we can further require that
\begin{equation}
\label{eq:thecpsimodfun2}
U_p'(L\vert \gamma)=U_p'(L)\vert \gamma
\end{equation}
for all modular functions $L$ with respect to $\Gamma_1^*(2p^2k,r)$ (see \eqref{eq:Gamma1starNM}) where
\begin{equation}
\label{eq:Upprime}
U_p'(L):=\frac{1}{p}\sum_{x=0}^{p-1}L\left(\frac{\tau+r_ex}{p}\right).
\end{equation}
\end{thm1.3*}

\begin{proof}
Since by assumption $(2\beta,k)=(2\beta',k)$, we have $(2\beta,2k)=(2\beta',2k)$ or $(2\beta,2k)=(2\beta'+k,2k)$ as is explained in the proof of Proposition \ref{prop:EquivalenceRelation}. If $(2\beta,2k)=(2\beta',2k)$, we can find an $a$ coprime to $2k$ such that $2\beta'\equiv a\cdot 2\beta\pmod{2k}$. We can assume that $p\nmid a$ for if $p\mid a$ then $p\nmid2k$ and hence after replacing $a$ with $a+2k$ we have $p\nmid a$. Let $c$ be any integer that is divisible by $2p^2k$ and coprime to $a$, and let $b,d$ be integers with $r\mid b$ and $ad-bc=1$. (They exist since $(r,a)\mid(k,a)=1$.) Then $\beta'=T_\beta\tbtmat{a}{b}{c}{d}$ by \eqref{eq:betapFormula}. The identity \eqref{eq:thecpsimodfun}, where we choose $\gamma=\tbtmat{a}{b}{c}{d}\in\Gamma_0^0(2p^2k,r)\subseteq\Gamma_0^0(p^2k,r)$, now follows from this and Lemma \ref{lemm:p2modularFunction}.

If $(2\beta,2k)=(2\beta'+k,2k)$, then $2\mid k$ and there exists an $a$ coprime to $2k$ such that $2\beta'+k\equiv a\cdot 2\beta\pmod{2k}$. As in the last case, we can assume that $p\nmid a$. Now we choose an integer $c$ coprime to $a$ with $p^2k\mid c$ but $2k\nmid c$ and then $b,d$ with $r\mid b$ and $ad-bc=1$. Then $\beta'=T_\beta\tbtmat{a}{b}{c}{d}$ by the second line of \eqref{eq:betapFormula}. The rest is the same to the last case.

Finally, let $\gamma=\tbtmat{a}{b}{c}{d}$ be fixed as above and assume that $(r,p)=1$; we proceed to prove \eqref{eq:thecpsimodfun2}. Note that $U_p'=U_p^{R_0}$ where $R_0=\{r_e\cdot0,r_e\cdot1,\dots,r_e\cdot(p-1)\}$. The assumption $(r,p)=1$ ensures that $R_0$ is a complete set of representatives of $\numZ/p\numZ$. It follows directly from Proposition \ref{prop:commute1} where we set $m=p$, $N=r_e$ and $R=R_0$ that $U_p^{R_0}(L\vert \gamma)=U_p^{R_0}(L)\vert \gamma$ since $\Gamma_1^*(N_0\abs{c},N_1)\subseteq\Gamma_1^*(2p^2k,r)$ and $L$ is modular on $\Gamma_1^*(2p^2k,r)$. From this $\eqref{eq:thecpsimodfun2}$ follows.
\end{proof}

We summarize the algorithm to find $\gamma$ in Theorem \ref{r:thecpsimodfun} as follows.
\begin{algo}
\label{algo:findgaama}
Given: $p$, $k$, $\beta$, $\beta'$, $r$, $r_e$ satisfying the conditions of Theorem \ref{r:thecpsimodfun}.
\begin{itemize}
	\item The case $(2\beta,2k)=(2\beta',2k)$:
	\begin{description}
	\item[Step 1a] Let $a\in\numZ$ satisfy $2\beta'\equiv a\cdot 2\beta\pmod{2k}$ and $(a,2k)=1$.
	\item[Step 2a] If $p\mid a$, replace $a$ with $a+2k$.
	\item[Step 3a] Let $c\in\numZ$ satisfy $2p^2k\mid c$ and $(a,c)=1$.
	\item[Step 4a] Let $b,d\in\numZ$ satisfy $r\mid b$ and $ad-bc=1$.
	\end{description}
	\item The case $(2\beta,2k)=(2\beta'+k,2k)$ (in which we must have $2\mid k$):
	\begin{description}
	\item[Step 1b] Let $a\in\numZ$ satisfy $2\beta'+k\equiv a\cdot 2\beta\pmod{2k}$ and $(a,2k)=1$.
	\item[Step 2b] If $p\mid a$, replace $a$ with $a+2k$.
	\item[Step 3b] Let $c\in\numZ$ satisfy $p^2k\mid c$, $2k\nmid c$ and $(a,c)=1$.
	\item[Step 4b] Let $b,d\in\numZ$ satisfy $r\mid b$ and $ad-bc=1$.
	\end{description}
\end{itemize}
In both cases, set $\gamma=\tbtmat{a}{b}{c}{d}$. Then the conclusion \eqref{eq:thecpsimodfun} holds, and if $(r,p)=1$, then \eqref{eq:thecpsimodfun2} holds as well.
\end{algo}

\begin{rema}
Let $\gamma$ be a matrix produced by Algorithm \ref{algo:findgaama} and we do \emph{not} assume $(r,p)=1$. Besides \eqref{eq:thecpsimodfun2} (which requires $(r,p)=1$), there is a more general formula if one uses other $U_p^R$ operators instead of $U_p'$. Let $N$ be a positive integer coprime to $p$ and set $N_0=(N,p-1)$, $N_1'=(N,(p-1)r)$. Suppose $R$ is a complete set of representatives of $\numZ/p\numZ$ such that $0\in R\subseteq N\numZ$. Then we have
\begin{equation}
\label{eq:UpRLcomm}
U_p^R(L\vert \gamma)=U_p^R(L)\vert\gamma
\end{equation}
for any modular function $L$ with respect to $\Gamma_1^*(N_0p^2k,N_1')$. This follows directly from Proposition \ref{prop:commute1}. To apply this formula, in the case $p\nmid k$, to as many $L$s as possible, one can set $N=(p-1)k$, so $N_0=p-1$ and $N_1'=(p-1)r$ since $r\mid k$. Then the conclusion \eqref{eq:UpRLcomm} is valid for all modular functions $L$ with respect to $\Gamma_1^*((p-1)p^2k,(p-1)r)$, which is the smallest group that can be attained in the current problem. Nevertheless, the larger group $\Gamma_1^*(2p^2k,r)$ used in Theorem \ref{r:thecpsimodfun} suffices for our purpose.
\end{rema}

The following special case on $\beta'=k/2$ is concerned with Andrews' $C\Phi_k$, in which we have a perfect choice of $\gamma$ if $k\equiv2\pmod{4}$.
\begin{prop}
\label{prop:whenk24gamma}
Let the notations be as in Theorem \ref{r:thecpsimodfun}. If $2\mid k$, then for $\gamma=\tbtmat{1}{0}{p^2k}{1}$ we have
\begin{equation}
\label{eq:0tok2}
\frac{f_{k,0}(\tau)}{f_{k,0}(p^2\tau)}\bigg| \gamma=\frac{f_{k,k/2}(\tau)}{f_{k,k/2}(p^2\tau)}.
\end{equation}
Moreover, \eqref{eq:thecpsimodfun2} holds if $(r,p)=1$. Finally, if $k\equiv2\pmod{4}$, then we have
\begin{equation}
\label{eq:10p2k1Decomp}
2\cdot\tbtMat{1}{0}{p^2k}{1}=\tbtMat{1}{-1/2}{0}{1}\tbtMat{2+p^2k}{2+p^2k/2}{2p^2k}{p^2k+2}\tbtMat{1}{-1/2}{0}{1}
\end{equation}
where $\tbtmat{2+p^2k}{2+p^2k/2}{2p^2k}{p^2k+2}$ is an Atkin-Lehner involution of level $2p^2k$. As a consequence, $\gamma$ maps each eta-quotient with respect to $\Gamma_0(2p^2k)$ to another one with respect to the same group, up to a constant factor.
\end{prop}
\begin{proof}
The $\gamma$ given above obeys the procedure in Algorithm \ref{algo:findgaama} where $\beta=0$, $\beta'=k/2$, so \eqref{eq:0tok2} and \eqref{eq:thecpsimodfun2} hold. The identity \eqref{eq:10p2k1Decomp} obviously holds. Set $Q=4$, $N=2p^2k$. Since $k\equiv2\pmod{4}$, we have $Q\Vert N$ (cf. \cite[Definition 6.6.1]{CS17}). Therefore, $\tbtmat{2+p^2k}{2+p^2k/2}{2p^2k}{p^2k+2}=W_4((2+p^2k)/4,2+p^2k/2,1,(2+p^2k)/4)$ where the right-side hand is defined in \cite[Definition 6.6.2]{CS17}. This is equivalent to saying that $\tbtmat{2+p^2k}{2+p^2k/2}{2p^2k}{p^2k+2}$ is an Atkin-Lehner involution of level $2p^2k$. The final assertion follows from the fact that $\tbtmat{1}{-1/2}{0}{1}$ maps $f(q)$ to $f(-q)$ and standard properties of Atkin-Lehner involutions.
\end{proof}
\begin{rema}
For the reader's convenience, we mention here how Atkin-Lehner involutions map eta-quotients to eta-quotients. We follow the notation of \cite[Section 6.6]{CS17}. Let $Q, N$ be positive integers with $Q\Vert N$. Let $\mathcal{D}_N$ be the set of positive divisors of $N$. We define $\sigma_Q\colon\mathcal{D}_N\to\mathcal{D}_N$ by $\sigma_Q(n)=(Q,N/n)\cdot(n,N/Q)$, which is an involution. Then for integers $j_n$ we have
\begin{equation*}
\prod_{n\mid N}\eta_n^{j_n}\Big\vert_{\sum j_n/2}W_Q=c\cdot\prod_{n\mid N}\eta_n^{j_{\sigma_Q(n)}}
\end{equation*}
where $c$ is a constant. In Proposition \ref{prop:whenk24gamma}, we are using $W_4$ where $N=2p^2k$.
\end{rema}

\subsection{A generalization of Garvan, Sellers and Smoot's result}
\label{subsec:genGSS}

We now explain how Theorem \ref{r:thecpsimodfun}, together with some auxiliary tools, can be applied to extend the result of Garvan, Sellers and Smoot \cite{GSS24} to more general $c\psi_{k,\beta}$. Let $p, k,\beta$ and $\beta'$ satisfy the conditions of Theorem \ref{r:thecpsimodfun} and let $r$, $r_e$ be defined by \eqref{eq:rre}. Define $U_p'$ by \eqref{eq:Upprime}, which depends on $p,k,\beta$ all. Since $(2\beta,k)=(2\beta',k)$ we have
\begin{equation}
\label{eq:rthesame}
r=\frac{k}{(k,(2\beta)^2(p^2-1)/8)}=\frac{k}{(k,(2\beta')^2(p^2-1)/8)},
\end{equation}
so the $U_p'$ given by $(p,k,\beta)$ is the same as the one given by $(p,k,\beta')$.

Define two $U_p'$-sequences by $K_0=L_0=1$, and for $\alpha\geq 0$
\beq
\label{r:s5Lalpha}
L_{2\alpha+1}:=U_p'(A_0L_{2\alpha}), \qquad L_{2\alpha+2}:=U_p'(L_{2\alpha+1}),
\eeq
and
\beq
\label{r:s5Kalpha}
K_{2\alpha+1}:=U_p'(A_1K_{2\alpha}), \qquad K_{2\alpha+2}:=U_p'(K_{2\alpha+1}),
\eeq
where
$$
A_0=q^{(p^2-1)(\beta^2/2k-k/12)}C\Psi_{k,\beta}(q)/C\Psi_{k,\beta}(q^{p^2})
$$
and
$$
A_1=q^{(p^2-1)(\beta'^2/2k-k/12)}C\Psi_{k,\beta'}(q)/C\Psi_{k,\beta'}(q^{p^2}).
$$
\begin{lemm}
\label{lemm:A0A1modular}
$A_0$ and $A_1$ are modular functions with respect to the group $\Gamma_1^*(2p^2k,r)$.
\end{lemm}
\begin{proof}
Since $\Gamma_1^*(2p^2k,r)\subseteq\Gamma_0(2p^2k)$, Lemma \ref{lemm:p2modularFunction} implies that for any $\tbtmat{a}{b}{c}{d}\in\Gamma_1^*(2p^2k,r)$ we have
\begin{equation*}
\frac{f_{k,\beta}(\tau)}{f_{k,\beta}(p^2\tau)}\Bigg\vert\tbtMat{a}{b}{c}{d}=\frac{f_{k,\beta_0}(\tau)}{f_{k,\beta_0}(p^2\tau)}
\end{equation*}
where $\beta_0=T_\beta\tbtmat{a}{b}{c}{d}$. Since $a\equiv1\pmod{2p^2k}$ and $c\equiv0\pmod{2p^2k}$, we have $T_\beta\tbtmat{a}{b}{c}{d}=\beta$ by \eqref{eq:betapFormula}. This proves that $A_0$ is a modular function on $\Gamma_1^*(2p^2k,r)$. To prove the assertion on $A_1$, note that we have \eqref{eq:rthesame}. Hence, we repeat the argument above with $A_0$ replaced by $A_1$ and $\beta$ replaced by $\beta'$ which gives the conclusion on $A_1$.
\end{proof}

\begin{thm}
\label{thm:thklkl}
Let $r$, $L_\alpha$, and $K_\alpha$ be defined as in \eqref{eq:rthesame}, \eqref{r:s5Lalpha} and \eqref{r:s5Kalpha}, respectively. Suppose that $(r,p)=1$. Then we have
\beq
\label{r:thklklStrong}
K_\alpha=L_\alpha \vert \gamma,\text{ and } K_\alpha,L_\alpha\text{ are modular functions with respect to }\Gamma_1^*(2p^2k,r),
\eeq
where $\gamma$ is any matrix chosen according to Algorithm \ref{algo:findgaama}.
\end{thm}

\begin{proof}
We prove \eqref{r:thklklStrong} by induction on $\alpha$. Obviously \eqref{r:thklklStrong} holds for $\alpha=0$. Assume that \eqref{r:thklklStrong} holds for $2\alpha$. Then
\begin{align*}
K_{2\alpha+1}=U_p'(A_1K_{2\alpha})=U_p'(A_0L_{2\alpha}\vert \gamma)=U_p'(A_0L_{2\alpha})\vert \gamma=L_{2\alpha+1} \vert \gamma
\end{align*}
where we have used Theorem \ref{r:thecpsimodfun} in the second and third equalities. Moreover, it is a standard property of $U_p'$ that $K_{2\alpha+1}$ and $L_{2\alpha+1}$ are modular functions on $\Gamma_1^*(2p^2k,r)$ since by the induction hypothesis and Lemma \ref{lemm:A0A1modular} $A_1K_{2\alpha}$ and $A_0L_{2\alpha}$ are. We have shown \eqref{r:thklklStrong} holds for $2\alpha+1$. Similarly, the assumption \eqref{r:thklklStrong} holds for $2\alpha+1$ can imply that it holds for $2\alpha+2$. This concludes the induction.
\end{proof}
\begin{rema}
If $r=1$, that is, $k\mid(2\beta)^2(p^2-1)/8$, then $U_p'(L)=U_p(L)$ for $L$ in Theorem \ref{r:thecpsimodfun} where $U_p$ is the usual Atkin-Lehner operator. In this case, $K_\alpha$ and $L_\alpha$ can be alternatively defined by
$$
L_{2\alpha+1}=U_p(A_0L_{2\alpha}), \qquad L_{2\alpha+2}=U_p(L_{2\alpha+1}),
$$
$$
K_{2\alpha+1}=U_p(A_1K_{2\alpha}), \qquad K_{2\alpha+2}=U_p(K_{2\alpha+1}).
$$
For instance, if $p^2\equiv1\pmod{2k}$, then the above discussion happens.
\end{rema}

The $U_p'$-sequence in Theorem \ref{thm:thklkl} can be employed to prove Ramanujan-type congruences. If there exists a family of congruences
\beq
\label{r:cpcongg}
c\psi_{k,\beta}(p^\alpha n+\delta_\alpha)\equiv 0 \pmod{p^{s(\alpha)}}
\eeq
(where $24k\delta_\alpha \equiv 12\beta^2-2k^2 \pmod{p^\alpha}$) that are guaranteed by the existence of identities
$$
L_\alpha=p^{s(\alpha)}f_\alpha(p_1,p_2,\cdots,p_i)
$$
with a sequence of Laurent polynomial $f_\alpha$ and modular functions $p_1,p_2,\cdots,p_i$ with $p$-adic integral coefficients, then by Theorem \ref{thm:thklkl} we have
\beq
\label{r:s5kp}
K_\alpha=p^{s(\alpha)}f_\alpha(p_1\vert \gamma,p_2\vert \gamma,\cdots,p_i\vert \gamma).
\eeq
We note that \eqref{r:s5kp} are not sufficient to indicate that
\beq
\label{r:s5kp1}
c\psi_{k,\beta'}(p^\alpha n+\delta'_\alpha)\equiv 0 \pmod{p^{s(\alpha)}},
\eeq
where $24k\delta'_\alpha \equiv 12\beta'^2-2k^2 \pmod{p^\alpha}$. However, if one can show that all $p_j\vert \gamma$ have $p$-adic integral coefficients, then \eqref{r:s5kp1} holds. We will show this is the case for $k=3$, $p=5$ in Section \ref{subsec:fromf312tof332}. Thus, Theorem \ref{thm:thklkl} would be a useful tool to find new, prove known, and establish relations between Ramanujan-type congruences.

\section{Congruences of $C\Psi_{3,\beta}$}
\label{r:sec6}

In this section, we provide a sketch of the proof of \eqref{r:c31cong} and then prove \eqref{r:c33cong} by transforming from $C\Psi_{3,1/2}$ to $C\Psi_{3,3/2}$. Here we use the abbreviation $\eta_m$ to denote the function $\eta(m\tau)$.

By calculating the $\zeta_{1/2}$ and $\zeta_{3/2}$ coefficients of \eqref{r:fkdef} directly we have
$$
C\Psi_{3,1/2}(q)=\frac{q^{-5/24}}{\eta_1^3}\sum_{n,m=-\infty}^\infty q^{(n+1/3)^2+(n+1/3)(m+1/3)+(m+1/3)^2},
$$
and
$$
C\Psi_{3,3/2}(q)=\frac{q^{1/8}}{\eta_1^3}\sum_{n,m=-\infty}^\infty q^{n^2+nm+m^2}.
$$
The expansions \eqref{eq:cpsi31genfun} and \eqref{r:cpsi33genfun} then follow immediately from the above formulas, \cite[Eq. (2.1)]{BBG94} and \cite[Prop. 2.2]{BBG94}. Alternatively, one can prove \eqref{eq:cpsi31genfun} and \eqref{r:cpsi33genfun} by the valence formula directly.

\subsection{Sketch of the proof of \eqref{r:c31cong}}
\label{subsec:sketch_r_c31cong}

The proof depends on the relations in Appendix \ref{appendix:modular_relations}. All these relations can be proved algorithmically using the valence formula; cf. \cite[Section 2]{CCG20}. Following Paule and Radu \cite{PR12} we let
\begin{equation}
\label{eq:A}
A:=\frac{C\Psi_{3,1/2}(q)}{q^5 C\Psi_{3,1/2}(q^{25})}=\frac{\eta_{25}^4\eta_{3}^3}{\eta_{75}^3\eta_{1}^4}.
\end{equation}
Define $U^{(0)}(f):=U_5(Af)$, $U^{(1)}(f):=U_5(f)$, and the U-sequence $L_\alpha$ by $L_0=1$ and for $\alpha\geq 0$
$$
L_{2\alpha+1}:=U^{(0)}(L_{2\alpha}), \qquad L_{2\alpha+2}:=U^{(1)}(L_{2\alpha+1}).
$$
Then we have
$$
L_{2\alpha}=\frac{1}{3}\prod_{n=1}^\infty \frac{(1-q^{n})^4}{(1-q^{3n})^3} \sum_{n=0}^\infty c\psi_{3,1/2}(5^{2\alpha}n+\delta_{2\alpha}) q^n,
$$
and
$$
L_{2\alpha+1}=\frac{1}{3q}\prod_{n=1}^\infty \frac{(1-q^{5n})^4}{(1-q^{15n})^3} \sum_{n=0}^\infty c\psi_{3,1/2}(5^{2\alpha+1}n+\delta_{2\alpha+1}) q^n.
$$
We also need
\begin{equation}
\label{eq:tp0p1}
t:=\frac{\eta_5^6}{\eta_1^6}, \quad p_0:=6xy+27x+(y-3)t, \quad \text{and}\qquad p_1:=12xy+81x+y+(12y-9)t,
\end{equation}
with
\beq
\label{r:xxyy}
x:=\frac{\eta_{15}^5 \eta_5}{\eta_3\eta_1^5} \qquad \text{and}  \qquad y:=\frac{\eta_5^2\eta_1^2}{\eta_{15}^2\eta_3^2}.
\eeq

The following lemma is analogous to \cite[Lemma 4.3]{PR12}.
Note that here we use the same $t$ but different $p_0$ and $p_1$. The first part of the proof of this lemma is to check all the relations in Appendix \ref{appendix:modular_relations}, which can be done by any computer algebra system. The rest of the proof follows \cite[Lemma 4.3]{PR12} and is omitted.

\begin{lemma}
\label{r:lm43}
There exist discrete functions (a map from $\mathbb{Z}^m$ to $\mathbb{Z}$ with finite support for some positive integer $m$) $a_i$, $b$, $c$, and $d_i$, $i\in \{0,1\}$,
such that the following relations hold for all $k\in \numgeq{Z}{0}$:
\begin{align*}
&U^{(0)}(t^k)=\sum_{n\geq \lceil (k+3)/5 \rceil} a_0(k,n)5^{\lfloor \frac{5n-k-2}{2} \rfloor} t^n
+p_1\sum_{n\geq \lceil k/5 \rceil} a_1(k,n)5^{\lfloor \frac{5n-k}{2} \rfloor} t^n,\\
&U^{(0)}(p_0t^k)=p_1\sum_{n\geq \lceil k/5 \rceil} b(k,n)5^{\lfloor \frac{5n-k-1}{2} \rfloor} t^n,\\
&U^{(1)}(t^k)=\sum_{n\geq \lceil k/5 \rceil} c(k,n)5^{\lfloor \frac{5n-k-1}{2} \rfloor} t^n,\\
&U^{(1)}(p_1t^k)=\sum_{n\geq \lceil (k+1)/5 \rceil} d_0(k,n)5^{\lfloor \frac{5n-k-1}{2} \rfloor} t^n
+p_0\sum_{n\geq \lceil (k-1)/5 \rceil} d_1(k,n)5^{\lfloor \frac{5n-k+2}{2} \rfloor} t^n.
\end{align*}
\end{lemma}

Let
\begin{align*}
X^{(0)}:=&\left\{\sum_{n=0}^\infty r_0(n)5^{\lfloor \frac{5n}{2} \rfloor}p_0t^n+\sum_{n=1}^\infty r_1(n)5^{\lfloor \frac{5n-3}{2} \rfloor}t^n:\text{$r_i$ are discrete functions} \right\},\\
X^{(1)}:=&\left\{\delta p_1+\sum_{n=1}^\infty r_0(n)5^{\lfloor \frac{5n-1}{2} \rfloor}p_1t^n+\sum_{n=1}^\infty r_1(n)5^{\lfloor \frac{5n-2}{2} \rfloor}t^n:\text{$r_i$ are discrete functions}, \delta \in \mathbb{Z} \right\}.
\end{align*}
By Lemma \ref{r:lm43} and the fundamental identity
$$
U^{(0)}(p_0)=(5^9 t^4+5^8 t^3+8 \cdot 5^5 t^2+4 \cdot 5^3+1)p_1
$$
given in Appendix \ref{appendix:modular_relations} we can prove the following lemma which is analogous to \cite[Lemma 4.4]{PR12}, and the proof is omitted again.

\begin{lemma}
\label{r:lm44}
We have
\begin{align*}
&f\in X^{(0)} \text{ implies } U^{(0)}(f)\in X^{(1)},\\
&f\in X^{(1)} \text{ implies } \frac{1}{5}U^{(1)}(f)\in X^{(0)}.\\
\end{align*}
\end{lemma}

By Lemma \ref{r:lm44} and the fundamental identity
$$
L_1=U^{(0)}(1)=5^7 t^3+9\cdot 5^4 t^2+9 \cdot 5 t+(5^5 t^2+8 \cdot 5^2 t+1)p_1 \in X^{(1)}
$$
given in Appendix \ref{appendix:modular_relations} we have for $\alpha \geq 1$,
\begin{align}
\label{r:Lmod5}
L_{2\alpha-1}\in 5^{\alpha-1} X^{(1)}\qquad \text{and} \qquad L_{2\alpha}\in 5^{\alpha} X^{(0)},
\end{align}
which proves \eqref{r:c31cong}.

\subsection{Finer transformation equations of $C\Psi_{3,1/2}$}
As shown by Table \ref{table:BkPartitions}, there is no modular transformation in $\Gamma_0(3)$ that permutes $C\Psi_{3,1/2}$ and $C\Psi_{3,3/2}$. This remains true if $\Gamma_0(3)$ is replaced by the larger group $\slZ$ according to the middle two formulas in Example \ref{examp:k234}. However, there indeed exists $M$ in the group algebra that sends $C\Psi_{3,1/2}$ to $C\Psi_{3,3/2}$ and has fine interaction with $U_m^R$. This is the theme of this subsection.

Set $T=\tbtmat{1}{1}{0}{1}$ and $\zeta_m=\etp{1/m}$.
\begin{prop}
\label{prop:f3betaGamma026}
Let $\gamma=\tbtmat{a}{b}{c}{d}\in\Gamma_0^0(2,6)$ and $s\in\numZ$. We have
\begin{equation*}
f_{3,1/2}\vert_{-\frac{1}{2}}\widetilde{\gamma T^s}=\alpha_\gamma\cdot\left(\zeta_{24}^{5s}\cdot\frac{2+\zeta_3^{cd/2}}{3}\cdot f_{3,1/2}+\zeta_8^{-s}\cdot\frac{1-\zeta_3^{cd/2}}{3}\cdot f_{3,3/2}\right),
\end{equation*}
where
\begin{equation*}
\alpha_\gamma=\legendre{\sgn{d}3b}{\abs{d}}\cdot\etp{\frac{\abs{d}-1}{8}}\cdot\rmi^{\frac{1-\sgn{d}}{2}\legendre{c}{-1}}\cdot\etp{\frac{(2b+c)d}{8}-\frac{ab}{24}}.
\end{equation*}
\end{prop}
\begin{proof}
If $3\mid c$, then the desired conclusion follows from Theorem \ref{thm:fkbetaGamma0k} and \eqref{eq:fkbetaVertT}. Now we assume that $3\nmid c$.
Let $\underline{2\numZ}$ denote the lattice $2\numZ$ equipped with the bilinear form $(x,y)\mapsto 3xy$. Then the dual lattice $(2\numZ)^\sharp=\tfrac{1}{6}\numZ$. Since $\gamma=\tbtmat{a}{b}{c}{d}\in\Gamma_0^0(2,6)\subseteq\Gamma_0(2)$ we can apply \eqref{eq:hbetakvertgammaOdd} with $k=3$ and obtain
\begin{equation*}
h_{1/6}^{(3)}\vert_{1}\widetilde{\gamma}=\chi_3(\widetilde{\gamma})\cdot\sum_{0\leq2\beta'<12,\,2\beta'\equiv1\bmod2}\rho_{\underline{2\numZ}}^*(\widetilde{\gamma})_{\frac{1}{6},\frac{2\beta'}{6}}\cdot h_{\beta'/3}^{(3)}.
\end{equation*}
Multiplying this formula by $(-\eta)^{-3}\vert_{-3/2}\widetilde{\gamma}=\chi_\eta^{-3}(\widetilde{\gamma})(-\eta)^{-3}$ we obtain
\begin{equation*}
f_{3,1/2}\vert_{-1/2}\widetilde{\gamma}=\chi_3(\widetilde{\gamma})\chi_\eta^{-3}(\widetilde{\gamma})\cdot\sum_{0\leq2\beta'<12,\,2\beta'\equiv1\bmod2}\rho_{\underline{2\numZ}}^*(\widetilde{\gamma})_{\frac{1}{6},\frac{2\beta'}{6}}\cdot f_{3,\beta'}.
\end{equation*}
Since $f_{3,\beta'}=f_{3,\beta'+k}=f_{3,k-\beta'}$ (see \eqref{eq:htsymmetry}) the above formula becomes
\begin{equation}
\label{eq:f312vertTemp}
f_{3,1/2}\vert_{-1/2}\widetilde{\gamma}=\chi_3(\widetilde{\gamma})\chi_\eta^{-3}(\widetilde{\gamma})\cdot\left(c_1(\gamma)\cdot f_{3,1/2}+c_3(\gamma)\cdot f_{3,3/2}\right)
\end{equation}
where
\begin{equation}
\label{eq:f3Gamma26Coeff}
c_1(\gamma)=\rho_{\underline{2\numZ}}^*(\widetilde{\gamma})_{\frac{1}{6},\frac{1}{6}}+\rho_{\underline{2\numZ}}^*(\widetilde{\gamma})_{\frac{1}{6},\frac{5}{6}}+\rho_{\underline{2\numZ}}^*(\widetilde{\gamma})_{\frac{1}{6},\frac{7}{6}}+\rho_{\underline{2\numZ}}^*(\widetilde{\gamma})_{\frac{1}{6},\frac{11}{6}},\quad c_3(\gamma)=\rho_{\underline{2\numZ}}^*(\widetilde{\gamma})_{\frac{1}{6},\frac{3}{6}}+\rho_{\underline{2\numZ}}^*(\widetilde{\gamma})_{\frac{1}{6},\frac{9}{6}}.
\end{equation}
Since $\gamma=\tbtmat{a}{b}{c}{d}\in\Gamma_0^0(2,6)$, we have $bc\cdot (2\numZ)^\sharp\subseteq2\numZ$ and hence Lemma \ref{lemm:coeffWeil} is applicable to $\underline{2\numZ}=\underline{2\numZ}_3$, which, for $l,m\in\{1,3,5,7,9,11\}$, gives
\begin{equation}
\label{eq:rho3coefficients}
\rho_{\underline{2\numZ}}(\widetilde{\gamma})_{\frac{l}{6},\frac{m}{6}}=(-\rmi)^{\frac{(1-\sgn d)}{2}\cdot\legendre{c}{-1}}\cdot\mathfrak{g}_{3}\left(b,d;\frac{l}{6}\right)\cdot\sqrt{\frac{\abs{D_{3}[c]}}{12}}\cdot\mathscr{G}_{3}\left(-cd,\frac{l-dm}{6}\right)
\end{equation}
if $\tfrac{al-m}{6}\in D_{3}[c]^\bullet$, and $\rho_{\underline{2\numZ}}(\widetilde{\gamma})_{\frac{l}{6},\frac{m}{6}}=0$ otherwise. We have (note that we have assumed $3\nmid c$)
\begin{align}
\frac{al-m}{6}\in D_{3}[c]^\bullet &\Longleftrightarrow c\cdot\frac{3}{2}y^2+3\cdot\frac{al-m}{6}\cdot y\in\numZ\quad(\forall y\in D_{3}[c])\notag\\
&\Longleftrightarrow\frac{6c}{(12,c)^2}n^2+\frac{al-m}{(12,c)}n\in\numZ\quad(\forall n\in \numZ)\notag\\
&\Longleftrightarrow\begin{dcases}
al-m\equiv0\pmod{2} & \text{if } c\equiv2\pmod{4},\\
al-m\equiv2\pmod{4} & \text{if } c\equiv4\pmod{8},\\
al-m\equiv0\pmod{4} & \text{if } c\equiv0\pmod{8}.\\
\end{dcases}\label{eq:when2Z3nonzero}
\end{align}
For the Gauss sum $\mathfrak{g}_{3}\left(b,d;\frac{l}{6}\right)$ that appears in \eqref{eq:rho3coefficients} we apply Lemma \ref{lemm:gbdt} and obtain
\begin{equation}
\label{eq:g2Z3}
\mathfrak{g}_{3}\left(b,d;\frac{l}{6}\right)=\legendre{\sgn{d}\cdot 3b}{\abs{d}}\etp{\frac{1-\abs{d}}{8}}\etp{\frac{bda^2l^2}{24}}=\legendre{\sgn{d}\cdot 3b}{\abs{d}}\etp{\frac{1-\abs{d}}{8}}\etp{\frac{abl^2}{24}}
\end{equation}
where we have used the assumption $12\mid bc$.
For the other Gauss sum $\mathscr{G}_{3}$ in \eqref{eq:rho3coefficients} we have, when $\tfrac{al-m}{6}\in D_{3}[c]^\bullet$,
\begin{equation*}
\mathscr{G}_{3}\left(-cd,\frac{l-dm}{6}\right)=\frac{1}{\sqrt{12\cdot\abs{D_{3}[cd]}}}\cdot\sum_{y\in\frac{1}{6}\numZ/2\numZ}\etp{-cd\cdot\frac{3}{2}y^2+3\cdot\frac{l-dm}{6}y}.
\end{equation*}
Inserting this and \eqref{eq:g2Z3} into \eqref{eq:rho3coefficients} we obtain
\begin{multline*}
\rho_{\underline{2\numZ}}(\widetilde{\gamma})_{\frac{l}{6},\frac{m}{6}}=(-\rmi)^{\frac{(1-\sgn d)}{2}\cdot\legendre{c}{-1}}\cdot\legendre{\sgn{d}\cdot 3b}{\abs{d}}\etp{\frac{1-\abs{d}}{8}}\etp{\frac{abl^2}{24}}\\
\cdot\frac{1}{12}\sum_{y\in\frac{1}{6}\numZ/2\numZ}\etp{-cd\cdot\frac{3}{2}y^2+3\cdot\frac{l-dm}{6}y}
\end{multline*}
if $\tfrac{al-m}{6}\in D_{3}[c]^\bullet$. Substituting this and \eqref{eq:when2Z3nonzero} into \eqref{eq:f3Gamma26Coeff} and taking into account \eqref{eq:rhoDual}, we find that
\begin{multline}
\label{eq:c1c3Formula}
c_j(\gamma)=\rmi^{\frac{(1-\sgn d)}{2}\cdot\legendre{c}{-1}}\cdot\legendre{\sgn{d}\cdot 3b}{\abs{d}}\etp{\frac{\abs{d}-1}{8}}\etp{-\frac{ab}{24}}\\
\cdot\frac{1}{12}\sum_{m\in B_j}\sum_{y\in\frac{1}{6}\numZ/2\numZ}\etp{cd\cdot\frac{3}{2}y^2-3\cdot\frac{1-dm}{6}y},
\end{multline}
where $j=1,3$,
$$
B_1=\begin{dcases}
\{1,5,7,11\} & \text{if }c\equiv2\pmod{4}\\
\{7,11\} & \text{if }c\equiv4\pmod{8}\text{ and }a\equiv1\pmod{4}\\
\{1,5\} & \text{if }c\equiv4\pmod{8}\text{ and }a\equiv3\pmod{4}\\
\{1,5\} & \text{if }c\equiv0\pmod{8}\text{ and }a\equiv1\pmod{4}\\
\{7,11\} & \text{if }c\equiv0\pmod{8}\text{ and }a\equiv3\pmod{4}
\end{dcases}
$$
and
$$
B_3=\begin{dcases}
\{3,9\} & \text{if }c\equiv2\pmod{4}\\
\{3\} & \text{if }c\equiv4\pmod{8}\text{ and }a\equiv1\pmod{4}\\
\{9\} & \text{if }c\equiv4\pmod{8}\text{ and }a\equiv3\pmod{4}\\
\{9\} & \text{if }c\equiv0\pmod{8}\text{ and }a\equiv1\pmod{4}\\
\{3\} & \text{if }c\equiv0\pmod{8}\text{ and }a\equiv3\pmod{4}.
\end{dcases}
$$
If follows by a tedious but elementary verification that (recall we have assumed $3\nmid c$)
\begin{equation}
\label{eq:doubleSumB1}
\frac{1}{12}\sum_{m\in B_1}\sum_{y\in\frac{1}{6}\numZ/2\numZ}\etp{cd\cdot\frac{3}{2}y^2-3\cdot\frac{1-dm}{6}y}=\frac{2+\zeta_3^{cd/2}}{3}.
\end{equation}
For instance, if $c\equiv2\pmod{24}$, $d\equiv1\pmod{24}$, then $B_1=\{1,5,7,11\}$ and the two sides both equal $(2+\zeta_3)/3$. Similarly, we have
\begin{equation}
\label{eq:doubleSumB3}
\frac{1}{12}\sum_{m\in B_3}\sum_{y\in\frac{1}{6}\numZ/2\numZ}\etp{cd\cdot\frac{3}{2}y^2-3\cdot\frac{1-dm}{6}y}=\frac{1-\zeta_3^{cd/2}}{3}.
\end{equation}
Inserting \eqref{eq:doubleSumB1} and \eqref{eq:doubleSumB3} into \eqref{eq:c1c3Formula}, then inserting the obtained formula into \eqref{eq:f312vertTemp}, we obtain
\begin{multline}
\label{eq:f312vertTemp2}
f_{3,1/2}\vert_{-1/2}\widetilde{\gamma}=\chi_3(\widetilde{\gamma})\chi_\eta^{-3}(\widetilde{\gamma})\cdot\rmi^{\frac{(1-\sgn d)}{2}\cdot\legendre{c}{-1}}\cdot\legendre{\sgn{d}\cdot 3b}{\abs{d}}\etp{\frac{\abs{d}-1}{8}}\etp{-\frac{ab}{24}}\\
\cdot\left(\frac{2+\zeta_3^{cd/2}}{3}\cdot f_{3,1/2}+\frac{1-\zeta_3^{cd/2}}{3}\cdot f_{3,3/2}\right).
\end{multline}
By \eqref{eq:etaChar} and \eqref{eq:chikFormula}, we have
\begin{align*}
\chi_3(\widetilde{\gamma})\chi_\eta^{-3}(\widetilde{\gamma})&=\chi_\eta^6\widetilde{\tbtmat{a}{b}{c}{d}}\cdot\rmi^{-\frac{3c}{2}}\cdot\chi_H\left([\tfrac{c}{2},\tfrac{d-1}{2}],1\right)\\
&=\etp{\frac{ac+bd+3d-3}{4}}\cdot\etp{-\frac{3c}{8}}\cdot\etp{\frac{1}{2}\left(\frac{c(d-1)}{4}+\frac{c}{2}+\frac{d-1}{2}\right)}\\
&=\etp{\frac{(2b+c)d}{8}}.
\end{align*}
Inserting this into \eqref{eq:f312vertTemp2}, we obtain the desired identity in the case $s=0$. The conclusion for arbitrary $s$ follows from the case $s=0$ and \eqref{eq:fkbetaVertT}. This concludes the proof for $3\nmid c$, hence the whole proof.
\end{proof}
\begin{rema}
Let $M\in\slZ$ be arbitrary. Then there always exists a factorization $$M=\tbtmat{1}{0}{3}{1}^{j_1}\cdot\gamma\cdot \tbtmat{1}{1}{0}{1}^s\cdot\tbtmat{0}{-1}{1}{0}^{j_2}$$ where $\gamma\in\Gamma_0^0(2,6)$, $s\in\{0,1,2,3,4,5\}$ and $j_1,j_2\in\{0,1\}$. Lifting it to $\sltZ$ we obtain $$(M,\varepsilon)=\widetilde{\tbtmat{1}{0}{3}{1}}^{j_1}\cdot\widetilde{\gamma T^s}\cdot\widetilde{\tbtmat{0}{-1}{1}{0}}^{j_2}$$ where $\varepsilon\in\{\pm1\}$. Therefore, the $f_{3,1/2}$ and $f_{3,3/2}$ coefficients of $$f_{3,1/2}\big\vert_{-1/2}\widetilde{M}=\varepsilon\cdot f_{3,1/2}\big\vert_{-1/2}\widetilde{\tbtmat{1}{0}{3}{1}}^{j_1}\big\vert_{-1/2}\widetilde{\gamma T^s}\big\vert_{-1/2}\widetilde{\tbtmat{0}{-1}{1}{0}}^{j_2}$$
can be obtained by combining \eqref{eq:10k1}, Proposition \ref{prop:f3betaGamma026} and \eqref{eq:fkbetaVertS}.
\end{rema}

\begin{examp}
\label{r:exm01}
Among the modular transformations stated in Proposition \ref{prop:f3betaGamma026} we need the following two:
\begin{align*}
f_{3,1/2}\left\vert_{-\frac{1}{2}}\tbtMat{1}{0}{50}{1}\right.=\rmi\cdot\left(\frac{2+\zeta_3}{3}\cdot f_{3,1/2}+\frac{1-\zeta_3}{3}\cdot f_{3,3/2}\right),\\
f_{3,1/2}\left\vert_{-\frac{1}{2}}\tbtMat{1}{0}{100}{1}\right.=-\left(\frac{2+\zeta_3^2}{3}\cdot f_{3,1/2}+\frac{1-\zeta_3^2}{3}\cdot f_{3,3/2}\right).
\end{align*}
Thus, if we set
\begin{equation}
\label{eq:M0}
M_0:=\zeta_{12}^{-1}\cdot\widetilde{\tbtMat{1}{0}{50}{1}}+\zeta_{3}\cdot\widetilde{\tbtMat{1}{0}{100}{1}}\in\mathscr{A}_\numZ,
\end{equation}
then $f_{3,1/2}\vert_{-\frac{1}{2}}M_0=f_{3,3/2}$. As a consequence, we have $f_{3,1/2}(q^5)\vert_{-\frac{1}{2}}M_1=f_{3,3/2}(q^5)$ where
\begin{equation}
\label{eq:M1}
M_1:=\zeta_{12}^{-1}\cdot\widetilde{\tbtMat{1}{0}{10}{1}}+\zeta_{3}\cdot\widetilde{\tbtMat{1}{0}{20}{1}}\in\mathscr{A}_\numZ.
\end{equation}
\end{examp}

Now the general Proposition \ref{prop:commute2algHalf}, applied to $f=f_{3,1/2},\,f_{3,1/2}(q^5)$ and $M=M_1,\,M_0$, gives us:
\begin{prop}
\label{r:prom01}
Let $M_0$, $M_1$ be as in \eqref{eq:M0} and \eqref{eq:M1} respectively. Let $h$ be a modular function with respect to $\Gamma_1^*(240)$. Let $R=\{0,24,48,72,96\}$. Then we have
\begin{align}
\left[U_5^R(h\cdot f_{3,1/2})\right]\vert_{-\frac{1}{2}}M_1 &=U_5^R\left[(h\cdot f_{3,1/2})\vert_{-\frac{1}{2}}M_0\right],\label{eq:U5RM1}\\
\left[U_5^R(h\cdot f_{3,1/2}(q^5))\right]\vert_{-\frac{1}{2}}M_0 &=U_5^R\left[(h\cdot f_{3,1/2}(q^5))\vert_{-\frac{1}{2}}M_1\right].\label{eq:U5RM0}
\end{align}
\end{prop}
\begin{proof}
Set in Proposition \ref{prop:commute2algHalf}
\begin{equation*}
m=5,\quad R=\{0,24,48,72,96\},\quad k_1=-1/2,\quad k_2=0,\quad f=f_{3,1/2},\quad g=h,\quad M=M_1.
\end{equation*}
Then $l=10$ and the matrices in \eqref{eq:halfCond1} are
\begin{gather*}
\tbtMat{1}{0}{0}{1},\tbtMat{57841}{-1152}{12000}{-239},\tbtMat{230881}{-4608}{24000}{-479},\tbtMat{519121}{-10368}{36000}{-719},\tbtMat{922561}{-18432}{48000}{-959},\\
\tbtMat{1}{0}{0}{1},\tbtMat{230881}{-2304}{48000}{-479},\tbtMat{922561}{-9216}{96000}{-959},\tbtMat{2075041}{-20736}{144000}{-1439},\tbtMat{3688321}{-36864}{192000}{-1919}.
\end{gather*}
By Theorem \ref{thm:fkbetaGamma0k} we know that $f_{3,1/2}\vert_{-1/2}A=f_{3,1/2}$ for $A$ being any of the above matrices. The first conclusion \eqref{eq:U5RM1} follows from this, the last conclusion of Proposition \ref{prop:commute2algHalf} and the fact $C_5^RM_1=M_0$.

To prove \eqref{eq:U5RM0}, we set in Proposition \ref{prop:commute2algHalf}
\begin{equation*}
m=5,\quad R=\{0,24,48,72,96\},\quad k_1=-1/2,\quad k_2=0,\quad f=f_{3,1/2}(q^5),\quad g=h,\quad M=M_0.
\end{equation*}
In this setting, $l=50$ and the matrices in \eqref{eq:halfCond1} are
\begin{gather*}
\tbtmat{1}{0}{0}{1},\tbtmat{1441201}{-5760}{300000}{-1199},\tbtmat{5762401}{-23040}{600000}{-2399},\tbtmat{12963601}{-51840}{900000}{-3599},\tbtmat{23044801}{-92160}{1200000}{-4799},\\
\tbtmat{1}{0}{0}{1},\tbtmat{5762401}{-11520}{1200000}{-2399},\tbtmat{23044801}{-46080}{2400000}{-4799},\tbtmat{51847201}{-103680}{3600000}{-7199},\tbtmat{92169601}{-184320}{4800000}{-9599}.
\end{gather*}
Again by Theorem \ref{thm:fkbetaGamma0k} we know that $f_{3,1/2}(q^5)\vert_{-1/2}A=f_{3,1/2}(q^5)$, that is, $f_{3,1/2}\vert_{-1/2}\tbtmat{5}{0}{0}{1}A\tbtmat{1/5}{0}{0}{1}=f_{3,1/2}$, for $A$ being any of the above matrices. It follows that
\begin{equation}
\label{eq:vertM0Temp}
\left[U_5^R(h\cdot f_{3,1/2}(q^5))\right]\vert_{-\frac{1}{2}}M_0=U_5^R\left[(h\cdot f_{3,1/2}(q^5))\vert_{-\frac{1}{2}}C_5^RM_0\right].
\end{equation}
We have
\begin{equation*}
C_5^RM_0=\zeta_{12}^{-1}\cdot\widetilde{\tbtMat{1}{0}{250}{1}}+\zeta_{3}\cdot\widetilde{\tbtMat{1}{0}{500}{1}}.
\end{equation*}
Therefore
\begin{equation}
\label{eq:vertM0Temp2}
(h\cdot f_{3,1/2}(q^5))\vert_{-\frac{1}{2}}C_5^RM_0=(h\cdot f_{3,1/2}(q^5))\vert_{-\frac{1}{2}}M_1
\end{equation}
by the assumption on $h$ and Theorem \ref{thm:fkbetaGamma0k}. Now \eqref{eq:U5RM0} follows from inserting \eqref{eq:vertM0Temp2} into \eqref{eq:vertM0Temp}.
\end{proof}
\begin{rema}
The function $h$ can be of any integral weight $k$ instead of weight $0$. The proof remains the same. Note that in this general case the action $\vert_{-\frac{1}{2}}$ should be replaced by $\vert_{k-\frac{1}{2}}$.
\end{rema}
\begin{rema}
Since $\Gamma_1^*(240)\subseteq\Gamma_0(240)$, if we know that $h$ is a modular function with respect to $\Gamma_0(240)$, then it is as well modular with respect to $\Gamma_1^*(240)$.
\end{rema}

\subsection{From \eqref{r:c31cong} to \eqref{r:c33cong}}
\label{subsec:fromf312tof332}

Unfortunately, it seems hard to follow the method proving \eqref{r:c31cong} to prove \eqref{r:c33cong} since the modular function representations associated with $C\Psi_{3,3/2}$ are more complex than those associated with $C\Psi_{3,1/2}$, which are presented in Section \ref{subsec:sketch_r_c31cong}. As we have mentioned above, there exists no single modular transformation that maps $C\Psi_{3,1/2}(q)$ to $C\Psi_{3,3/2}(q)$. What we have done is to use the combination transformations $M_0$ and $M_1$, which are defined by \eqref{eq:M0} and \eqref{eq:M1} respectively. Then (see Example \ref{r:exm01})
\begin{align}
\label{r:mtran1}
q^{5/24}C\Psi_{3,1/2}(q)\vert M_0&=q^{-1/8}C\Psi_{3,3/2}(q),\\
\label{r:mtran2}
q^{25/24}C\Psi_{3,1/2}(q^5)\vert M_1&=q^{-5/8}C\Psi_{3,3/2}(q^5),
\end{align}
where (and in this subsection) the operator $\vert$ denotes $\vert_{-1/2}$, the weight $-1/2$ slash operator. Let $U_5^*=U_5^R$ with $R=\{0,24,48,72,96\}$. Then for formal Laurent series $f(q)$ and $g(q)$ in $\numC[[q,q^{-1}]]$, we have
$$
U_5^*(f(q^{5/24})g(q^{1/24}))=f(q^{1/24})U_5^*(g(q^{1/24})),\quad U_5^*(f)=U_5(f).
$$
By Proposition \ref{r:prom01} We have
\begin{align}
\label{r:mtran3}
U_5^*(q^{5/24}C\Psi_{3,1/2}(q)h)\vert M_1&=U_5^*(q^{5/24}C\Psi_{3,1/2}(q)h\vert M_0),\\
\label{r:mtran4}
U_5^*(q^{25/24}C\Psi_{3,1/2}(q^5)h)\vert M_0&=U_5^*(q^{25/24}C\Psi_{3,1/2}(q^5)h\vert M_1),
\end{align}
where $h$ is a function modular on $\Gamma_0(240)$ with trivial multiplier and of integral weight. In particular, $h$ can be a modular function on $\Gamma_0(15)$, such as $L_\alpha$ discussed in Section \ref{subsec:sketch_r_c31cong}, or any rational function of $t$, $p_0$ and $p_1$ (see \eqref{eq:tp0p1}) since the multiplier systems of these three functions are trivial on $\Gamma_0(15)$. We let $K_0=1$ and for $\alpha\geq 0$
$$
K_{2\alpha+1}=U_5(A_3K_{2\alpha}),
$$
$$
K_{2\alpha+2}=U_5(K_{2\alpha+1}),
$$
where $A_3:=q^3\frac{C\Psi_{3,3/2}(q)}{C\Psi_{3,3/2}(q^{25})}$. On one hand, we can calculate $K_\alpha$ for $\alpha>0$ directly:
\begin{gather}
K_{2\alpha-1}=\frac{q}{C\Psi_{3,3/2}(q^5)} \sum_{n=0}^\infty c\psi_{3,3/2}(5^{2\alpha-1}n+\lambda_{2\alpha-1}) q^n,\label{eq:Kodddirect}\\
K_{2\alpha}=\frac{q}{C\Psi_{3,3/2}(q)} \sum_{n=0}^\infty c\psi_{3,3/2}(5^{2\alpha}n+\lambda_{2\alpha}) q^n.\label{eq:Kevendirect}
\end{gather}
It follows that $U_5^*(K_\alpha)=U_5(K_\alpha)$. On the other hand, we calculate $K_\alpha$ by the properties \eqref{r:mtran1}--\eqref{r:mtran4}.

\begin{lemma}
\label{lema:LtoK}
For $\alpha\geq 0$ we have
\beq
\label{r:K2a}
K_{2\alpha}=\frac{q^{1/8}}{C\Psi_{3,3/2}(q)}\left(q^{5/24}C\Psi_{3,1/2}(q)L_{2\alpha}\mid M_0\right),
\eeq
and
\beq
\label{r:K2a+1}
K_{2\alpha+1}=\frac{q^{5/8}}{C\Psi_{3,3/2}(q^5)}\left(q^{25/24}C\Psi_{3,1/2}(q^5)L_{2\alpha+1}\mid M_1\right).
\eeq
\end{lemma}

\begin{proof}
By \eqref{r:mtran1} we have
$$
\frac{q^{1/8}}{C\Psi_{3,3/2}(q)}\left(q^{5/24}C\Psi_{3,1/2}(q)L_{0}\mid M_0\right)=1=K_0.
$$
Assume that \eqref{r:K2a} holds for $\alpha$; then
\begin{align*}
K_{2\alpha+1}=&U_5^*(A_3 K_{2\alpha})\\
=&U_5^*\left(\frac{q^{25/8}}{C\Psi_{3,3/2}(q^{25})}(q^{5/24}C\Psi_{3,1/2}(q)L_{2\alpha}\mid M_0)\right)\\
=&\frac{q^{5/8}}{C\Psi_{3,3/2}(q^5)}U_5^*(q^{5/24}C\Psi_{3,1/2}(q)L_{2\alpha}\mid M_0)\\
=&\frac{q^{5/8}}{C\Psi_{3,3/2}(q^5)}\left(U_5^*(q^{5/24}C\Psi_{3,1/2}(q)L_{2\alpha})\mid M_1\right)\\
=&\frac{q^{5/8}}{C\Psi_{3,3/2}(q^5)}\left(q^{25/24}C\Psi_{3,1/2}(q^5)U_5(A L_{2\alpha})\mid M_1\right)\\
=&\frac{q^{5/8}}{C\Psi_{3,3/2}(q^5)}\left(q^{25/24}C\Psi_{3,1/2}(q^5)L_{2\alpha+1}\mid M_1\right).
\end{align*}
Hence \eqref{r:K2a+1} holds for $\alpha$. Similarly \eqref{r:K2a} for $\alpha+1$ follows from \eqref{r:K2a+1} for $\alpha$. Thus, the lemma holds by induction.
\end{proof}

To clarify the right-hand side of \eqref{r:K2a}, we need the following lemma. Recall that $x,y$ are defined by \eqref{r:xxyy}.

\begin{lemma}
\label{r:lmm0}
We have
\begin{align}
q^{5/24}C\Psi_{3,1/2}(q)x\vert M_0=&\frac{1}{9}\eta_1^{-9}\eta_5\cdot q^{3/8}\left(q^{1/3}\sum_{n\geq 0}a_{3n+1}q^n-q^{2/3}\sum_{n\geq0}a_{3n+2}q^n\right),\label{eq:C31xM0}\\
q^{5/24}C\Psi_{3,1/2}(q)y\vert M_0=&9\eta_1^{-2}\eta_5^2\cdot q^{-1/8}\left(-q^{1/3}\sum_{n\geq 0}b_{3n+1}q^n+q^{2/3}\sum_{n\geq0}b_{3n+2}q^n\right),\label{eq:C31yM0}\\
q^{5/24}C\Psi_{3,1/2}(q)xy\vert M_0=&\eta_1^{-7}\eta_5^3\cdot q^{5/24}\left(\sum_{n\geq0}c_{3n}q^n-q^{1/3}\sum_{n\geq 0}c_{3n+1}q^n\right),\label{eq:C31xyM0}
\end{align}
where the sequences $(a_n)_{n\geq0}$, $(b_n)_{n\geq0}$ and $(c_n)_{n\geq0}$ are defined by
\begin{align*}
\eta_1^2\eta_5^5=q^{9/8}\sum_{n\geq0}a_nq^n,\quad \eta_1\eta_5^{-2}=q^{-3/8}\sum_{n\geq0}b_nq^n,\quad \eta_5^3=q^{5/8}\sum_{n\geq0}c_nq^n.
\end{align*}
\end{lemma}

\begin{proof}
By \eqref{r:xxyy} and \eqref{eq:cpsi31genfun}, the left-hand side of \eqref{eq:C31xM0} equals $3\eta_1^{-9}\eta_3^2\eta_5\eta_{15}^5\vert M_0$. The key observation is (according to Proposition \ref{prop:cocycleFormula})
\begin{equation*}
\widetilde{\tbtMat{1}{0}{2}{1}}=\widetilde{\tbtMat{-1}{1}{-3}{2}}\widetilde{\tbtMat{0}{-1}{1}{0}}\widetilde{\tbtMat{1}{-1}{0}{1}}.
\end{equation*}
It follows that
\begin{align}
\eta_3\Big\vert_{\frac{1}{2}}\widetilde{\tbtMat{1}{0}{2}{1}}&=\eta_3\Big\vert_{\frac{1}{2}}\widetilde{\tbtMat{-1}{1}{-3}{2}}\Big\vert_{\frac{1}{2}}\widetilde{\tbtMat{0}{-1}{1}{0}}\Big\vert_{\frac{1}{2}}\widetilde{\tbtMat{1}{-1}{0}{1}}\notag\\
&=\chi_\eta\widetilde{\tbtMat{-1}{3}{-1}{2}}\cdot\left(\tau^{-\frac{1}{2}}\eta\left(-\frac{3}{\tau}\right)\right)\Big\vert_{\frac{1}{2}}\widetilde{\tbtMat{1}{-1}{0}{1}}\notag\\
&=\frac{1}{\sqrt{3}}\etp{-\frac{1}{24}}\eta\left(\frac{\tau-1}{3}\right)\label{eq:eta3vert1011}
\end{align}
where we have used \eqref{eq:etaChar} twice for $\tbtmat{-1}{3}{-1}{2}$ and $\tbtmat{0}{-1}{1}{0}$.
Since
\begin{equation*}
\widetilde{\tbtMat{1}{0}{1}{1}}=\widetilde{\tbtMat{-1}{1}{-3}{2}}\widetilde{\tbtMat{1}{0}{2}{1}}\widetilde{\tbtMat{1}{-1}{0}{1}},
\end{equation*}
we have similarly
\begin{equation}
\label{eq:eta3vert1021}
\eta_3\Big\vert_{\frac{1}{2}}\widetilde{\tbtMat{1}{0}{1}{1}}=\frac{1}{\sqrt{3}}\etp{\frac{1}{24}}\eta\left(\frac{\tau-2}{3}\right).
\end{equation}
From \eqref{eq:eta3vert1011} and \eqref{eq:eta3vert1021} we obtain
\begin{align*}
\eta_3\Big\vert_{\frac{1}{2}}\widetilde{\tbtMat{1}{0}{50}{1}}&=\eta_3\Big\vert_{\frac{1}{2}}\widetilde{\tbtMat{1}{0}{48}{1}}\Big\vert_{\frac{1}{2}}\widetilde{\tbtMat{1}{0}{2}{1}}=\frac{1}{\sqrt{3}}\etp{\frac{7}{24}}\eta\left(\frac{\tau-1}{3}\right),\\
\eta_3\Big\vert_{\frac{1}{2}}\widetilde{\tbtMat{1}{0}{100}{1}}&=\eta_3\Big\vert_{\frac{1}{2}}\widetilde{\tbtMat{1}{0}{99}{1}}\Big\vert_{\frac{1}{2}}\widetilde{\tbtMat{1}{0}{1}{1}}=\frac{1}{\sqrt{3}}\etp{\frac{2}{3}}\eta\left(\frac{\tau-2}{3}\right)
\end{align*}
respectively. A similar argument shows that
\begin{align*}
\eta_{15}\Big\vert_{\frac{1}{2}}\widetilde{\tbtMat{1}{0}{50}{1}}&=\frac{1}{\sqrt{3}}\etp{-\frac{1}{24}}\eta\left(\frac{5(\tau-1)}{3}\right),\\
\eta_{15}\Big\vert_{\frac{1}{2}}\widetilde{\tbtMat{1}{0}{100}{1}}&=\frac{1}{\sqrt{3}}\etp{-\frac{1}{6}}\eta\left(\frac{5(\tau-2)}{3}\right).
\end{align*}
Taking into account the last four formulas, we have
\begin{align*}
&\mathrel{\phantom{=}}3\eta_1^{-9}\eta_3^2\eta_5\eta_{15}^5\Big\vert\widetilde{\tbtMat{1}{0}{50}{1}}\\
&=3\chi_\eta\widetilde{\tbtMat{1}{0}{50}{1}}^{-9}\eta_1^{-9}\cdot\left(\frac{1}{\sqrt{3}}\etp{\frac{7}{24}}\eta\left(\frac{\tau-1}{3}\right)\right)^2\cdot\chi_\eta\widetilde{\tbtMat{1}{0}{10}{1}}\eta_5\cdot\left(\frac{1}{\sqrt{3}}\etp{-\frac{1}{24}}\eta\left(\frac{5(\tau-1)}{3}\right)\right)^{5}\\
&=3^{-\frac{5}{2}}\etp{\frac{17}{24}}\cdot\eta_1^{-9}\eta_5\cdot\eta\left(\frac{\tau-1}{3}\right)^2\eta\left(\frac{5(\tau-1)}{3}\right)^5
\end{align*}
and
\begin{equation*}
3\eta_1^{-9}\eta_3^2\eta_5\eta_{15}^5\Big\vert\widetilde{\tbtMat{1}{0}{100}{1}}=-3^{-\frac{5}{2}}\etp{\frac{2}{3}}\cdot\eta_1^{-9}\eta_5\cdot\eta\left(\frac{\tau-2}{3}\right)^2\eta\left(\frac{5(\tau-2)}{3}\right)^5.
\end{equation*}
Therefore,
\begin{equation*}
3\eta_1^{-9}\eta_3^2\eta_5\eta_{15}^5\vert M_0=3^{-\frac{5}{2}}\eta_1^{-9}\eta_5\cdot\left(\etp{\frac{5}{8}}\eta\left(\frac{\tau-1}{3}\right)^2\eta\left(\frac{5(\tau-1)}{3}\right)^5-\eta\left(\frac{\tau-2}{3}\right)^2\eta\left(\frac{5(\tau-2)}{3}\right)^5\right).
\end{equation*}
By expanding the $q$-series of the right-hand {}side we obtain \eqref{eq:C31xM0}. Similar arguments give
\begin{gather*}
3\eta_1^{-2}\eta_3\eta_5^2\eta_{15}^{-2}\vert M_0=3^{\frac{3}{2}}\eta_1^{-2}\eta_5^2\cdot\left(\etp{\frac{5}{8}}\eta\left(\frac{\tau-1}{3}\right)\eta\left(\frac{5(\tau-1)}{3}\right)^{-2}+\eta\left(\frac{\tau-2}{3}\right)\eta\left(\frac{5(\tau-2)}{3}\right)^{-2}\right),\\
3\eta_1^{-7}\eta_5^3\eta_{15}^{3}\vert M_0=3^{-\frac{1}{2}}\eta_1^{-7}\eta_5^3\cdot\left(\etp{\frac{1}{8}}\eta\left(\frac{5(\tau-1)}{3}\right)^{3}-\eta\left(\frac{5(\tau-2)}{3}\right)^{3}\right),
\end{gather*}
which lead to \eqref{eq:C31yM0} and \eqref{eq:C31xyM0} respectively.
\end{proof}

The parallel lemma for \eqref{r:K2a+1} is the following one.
\begin{lemma}
\label{r:lmm1}
We have
\begin{align*}
q^{25/24}C\Psi_{3,1/2}(q^5)x\vert M_1=&\frac{1}{9}\eta_1^{-5}\eta_5^{-3}\cdot q^{13/24}\left(\sum_{n\geq0}a'_{3n}q^n-q^{2/3}\sum_{n\geq 0}a'_{3n+2}q^n\right),\\
q^{25/24}C\Psi_{3,1/2}(q^5)y\vert M_1=&9\eta_1^{2}\eta_5^{-2}\cdot q^{1/24}\left(-\sum_{n\geq 0}b'_{3n}q^n+q^{2/3}\sum_{n\geq0}b'_{3n+2}q^n\right),\\
q^{25/24}C\Psi_{3,1/2}(q^5)xy\vert M_1=&\eta_1^{-3}\eta_5^{-1}\cdot q^{3/8}\left(-q^{1/3}\sum_{n\geq 0}c'_{3n+1}q^n+q^{2/3}\sum_{n\geq0}c'_{3n+2}q^n\right),
\end{align*}
where the sequences $(a_n')_{n\geq0}$, $(b_n')_{n\geq0}$ and $(c_n')_{n\geq0}$ are defined by

\begin{align*}
\eta_1^{-1}\eta_5^8=q^{13/8}\sum_{n\geq0}a_n'q^n,\quad \eta_1^{-2}\eta_5^{1}=q^{1/8}\sum_{n\geq0}b_n'q^n,\quad \eta_1^{-3}\eta_5^6=q^{9/8}\sum_{n\geq0}c_n'q^n.
\end{align*}
\end{lemma}

The proof is similar to Lemma \ref{r:lmm0} and is omitted.

Recall that $t$ is defined in \eqref{eq:tp0p1}. It is a modular function on $\Gamma_0(5)$. A straightforward verification shows the following:
\begin{lemma}
\label{lemm:hftvertMi}
Let $f(t)$ be a Laurent polynomial of the modular function $t$, $i=0,1$ and $h$ be arbitrary function defined on $\uhp$. Then we have
$$
hf(t)\vert M_i=h\vert M_i\cdot f(t).
$$
\end{lemma}

Now we define the $C\Psi_{3,3/2}$-analog of \eqref{eq:tp0p1}. The modular function $t$ remains unchanged and let
$$
\overline{p_0}:=\frac{q^{1/8}}{C\Psi_{3,3/2}(q)}\left(q^{5/24}C\Psi_{3,1/2}(q)p_0\vert M_0\right),
$$
$$
\overline{p_1}:=\frac{q^{5/8}}{C\Psi_{3,3/2}(q^5)}\left(q^{25/24}C\Psi_{3,1/2}(q^5)p_1\vert M_1\right).
$$

\begin{lemm}
\label{lemm:overlinep0p1U}
For each of the identities stated in Appendix \ref{appendix:modular_relations}, if we replace $A$, $p_0$ and $p_1$ with $A_3$, $\overline{p_0}$ and $\overline{p_1}$, respectively, then the obtained formula remains true. In particular, we have
\begin{align}
\overline{p_1}=-\frac{25t}{2}-\frac{1}{2}U_5(A_3t^{-1}),\label{eq:overlinep1U}\\
\overline{p_0}=\frac{18}{5}+\frac{1}{5}U_5(\overline{p_1}t^{-1}).\label{eq:overlinep0U}
\end{align}
\end{lemm}
\begin{proof}
First we show $U_5(At^{-1})=-5^2 t-2p_1$ implies \eqref{eq:overlinep1U}. Multiplying both sides of $U_5(At^{-1})=-5^2 t-2p_1$ by $q^{25/24}C\Psi_{3,1/2}(q^5)$ we obtain
\begin{equation*}
U_5^*(q^{5/24}C\Psi_{3,1/2}(q)t^{-1})=-25tq^{25/24}C\Psi_{3,1/2}(q^5)-2q^{25/24}C\Psi_{3,1/2}(q^5)p_1.
\end{equation*}
Applying to both sides the operator $\vert M_1$ and using \eqref{r:mtran1}-\eqref{r:mtran3} and Lemma \ref{lemm:hftvertMi} we find that
\begin{equation*}
U_5^*(q^{-1/8}C\Psi_{3,3/2}(q)t^{-1})=-25tq^{-5/8}C\Psi_{3,3/2}(q^5)-2(q^{25/24}C\Psi_{3,1/2}(q^5)p_1)\vert M_1.
\end{equation*}
Therefore,
\begin{equation*}
\overline{p_1}=-\frac{25t}{2}-\frac{1}{2}\frac{U_5^*(q^{-1/8}C\Psi_{3,3/2}(q)t^{-1})}{q^{-5/8}C\Psi_{3,3/2}(q^5)}=-\frac{25t}{2}-\frac{1}{2}U_5(A_3t^{-1}).
\end{equation*}
To see \eqref{eq:overlinep0U}, we multiply both sides of $U_5(p_1t^{-1})=-18+5p_0$ by $q^{5/24}C\Psi_{3,1/2}(q)$ and obtain
\begin{equation*}
U_5^*(q^{25/24}C\Psi_{3,1/2}(q^5)p_1t^{-1})=-18q^{5/24}C\Psi_{3,1/2}(q)+5q^{5/24}C\Psi_{3,1/2}(q)p_0.
\end{equation*}
Applying to both sides the operator $\vert M_0$ and using \eqref{r:mtran1}, \eqref{r:mtran2}, \eqref{r:mtran4} and Lemma \ref{lemm:hftvertMi} we deduce that
\begin{equation*}
q^{-1/8}C\Psi_{3,3/2}(q)U_5(\overline{p_1}t^{-1})=-18q^{-1/8}C\Psi_{3,3/2}(q)+5q^{-1/8}C\Psi_{3,3/2}(q)\overline{p_0},
\end{equation*}
which reduces to \eqref{eq:overlinep0U}. Other implications can be proved in the same strategy.
\end{proof}

The explicit expressions of the coefficients of $\overline{p_0}$ and $\overline{p_1}$ are embodied in the following formulas.
\begin{lemma}
\label{eq:overlinep0p1coeff}
We have
$$
\overline{p_0}=\frac{1}{C\Psi_{3,3/2}(q)}
\left(9q^{2/3}\frac{\eta_5^8}{\eta_1^8}\sum_{n=0}^\infty b_{3n+2}q^n-3q^{7/6}\frac{\eta_5}{\eta_1^9} \sum_{n=0}^\infty a_{3n+2}q^n \right)-3t,
$$
\begin{align*}
\overline{p_1}=&\frac{1}{C\Psi_{3,3/2}(q^5)}
\bigg(9q^{4/3}\frac{\eta_1^2}{\eta_5^2}(1+12t)\sum_{n=0}^\infty b'_{3n+2}q^n
-12q^{4/3}\frac{1}{\eta_1^3 \eta_5} \sum_{n=0}^\infty c'_{3n+1}q^n\\
&-9q^{11/6}\frac{1}{\eta_1^5 \eta_5^3} \sum_{n=0}^\infty a'_{3n+2}q^n \bigg)-9t,
\end{align*}
\end{lemma}

\begin{proof}
It follows from \eqref{eq:tp0p1}, \eqref{r:mtran1}, Lemmas \ref{r:lmm0} and \ref{lemm:hftvertMi} that
\begin{align}
&\mathrel{\phantom{=}}C\Psi_{3,3/2}(q)\overline{p_0}\label{eq:C3p0decomp}\\
&=q^{1/8}\left(6q^{5/24}C\Psi_{3,1/2}xy\vert M_0 + 27q^{5/24}C\Psi_{3,1/2}x\vert M_0+tq^{5/24}C\Psi_{3,1/2}y\vert M_0-3tq^{5/24}C\Psi_{3,1/2}\vert M_0\right)\notag\\
&=S_1+S_2-3tC\Psi_{3,3/2}(q)\notag
\end{align}
where
\begin{equation*}
S_1=9q^{2/3}\frac{\eta_5^8}{\eta_1^8}\sum_{n=0}^\infty b_{3n+2}q^n-3q^{7/6}\frac{\eta_5}{\eta_1^9} \sum_{n=0}^\infty a_{3n+2}q^n-6q^{2/3}\frac{\eta_5^3}{\eta_1^7}\sum_{n=0}^\infty c_{3n+1}q^n
\end{equation*}
and
\begin{equation*}
S_2=-9q^{1/3}\frac{\eta_5^8}{\eta_1^8}\sum_{n=0}^\infty b_{3n+1}q^n+3q^{5/6}\frac{\eta_5}{\eta_1^9} \sum_{n=0}^\infty a_{3n+1}q^n+6q^{1/3}\frac{\eta_5^3}{\eta_1^7}\sum_{n=0}^\infty c_{3n}q^n.
\end{equation*}
Likewise, by \eqref{eq:tp0p1}, \eqref{r:mtran2}, Lemmas \ref{lemm:hftvertMi} and \ref{r:lmm1} we have
\begin{equation}
C\Psi_{3,3/2}(q^5)\overline{p_1}=S_3+S_4-9tC\Psi_{3,3/2}(q^5)\label{eq:C3p1decomp}
\end{equation}
where
\begin{equation*}
S_3=9q^{4/3}\frac{\eta_1^2}{\eta_5^2}(1+12t)\sum_{n=0}^\infty b'_{3n+2}q^n-12q^{4/3}\frac{1}{\eta_1^3 \eta_5} \sum_{n=0}^\infty c'_{3n+1}q^n-9q^{11/6}\frac{1}{\eta_1^5 \eta_5^3} \sum_{n=0}^\infty a'_{3n+2}q^n
\end{equation*}
and
\begin{equation*}
S_4=-9q^{2/3}\frac{\eta_1^2}{\eta_5^2}(1+12t)\sum_{n=0}^\infty b'_{3n}q^n + 12q^{5/3}\frac{1}{\eta_1^3 \eta_5} \sum_{n=0}^\infty c'_{3n+2}q^n+9q^{7/6}\frac{1}{\eta_1^5 \eta_5^3} \sum_{n=0}^\infty a'_{3n}q^n.
\end{equation*}
Expanding the Fourier series we find that
\begin{gather*}
S_1=\sum_{n\in\numZ}s_1(n)q^n,\quad S_2=q^{2/3}\sum_{n\in\numZ}s_2(n)q^n,\\
S_3=\sum_{n\in\numZ}s_3(n)q^n,\quad S_4=q^{1/3}\sum_{n\in\numZ}s_4(n)q^n.\\
\end{gather*}
We have $S_2=S_4=0$ for each term of the $q$-series expansions of $\overline{p_0}$ and $\overline{p_1}$ has integral exponent by \eqref{eq:overlinep1U} and \eqref{eq:overlinep0U}. Thus, the assertion on $\overline{p_1}$ follows from this and \eqref{eq:C3p1decomp}. To conclude the proof for $\overline{p_0}$ from \eqref{eq:C3p0decomp}, it remains to show $c_{3n+1}=0$, which can be seen from the well-known identity
\begin{equation*}
\eta(\tau)^3=\frac{1}{2}\sum_{n\in\numZ}\legendre{-4}{n}nq^{n^2/8}.
\end{equation*}
\end{proof}

We are now ready to prove \eqref{r:c33cong}, that is,
\begin{align*}
c\psi_{3,3/2}(5^{\alpha}n+\lambda_{\alpha})\equiv 0 \pmod{5^{\lfloor\frac{\alpha}{2}\rfloor}}.
\end{align*}
\begin{proof}
Obviously, $\overline{p_0}$ and $\overline{p_1}$ are formal $q$-series with integral coefficients by Lemma \ref{eq:overlinep0p1coeff}. We have shown in Section \ref{subsec:sketch_r_c31cong} that $L_{\alpha}$ can be decomposed as
$$
L_{\alpha}=f_{\alpha}(t)+p_i g_{\alpha}(t),
$$
where $i=0,1$ is determined by $\alpha\equiv i \pmod{2}$ and $f_\alpha$, $g_\alpha$ are Laurent polynomials. Applying Lemma \ref{lema:LtoK} to this identity and using \eqref{r:mtran1}, \eqref{r:mtran2} and Lemma \ref{lemm:hftvertMi}, we obtain
$$
K_{\alpha}=f_{\alpha}(t)+\overline{p_i} g_{\alpha}(t).
$$
It then follows from \eqref{r:Lmod5} that
\begin{equation}
\label{eq:KinX}
K_{2\alpha-1}\in 5^{\alpha-1} \overline{X}^{(1)}\qquad \text{and} \qquad K_{2\alpha}\in 5^{\alpha} \overline{X}^{(0)},
\end{equation}
where
\begin{align*}
\overline{X}^{(0)}:=&\left\{\sum_{n=0}^\infty r_0(n)5^{\lfloor \frac{5n}{2} \rfloor}\overline{p_0}t^n+\sum_{n=1}^\infty r_1(n)5^{\lfloor \frac{5n-3}{2} \rfloor}t^n:\text{$r_i$ are discrete functions} \right\},\\
\overline{X}^{(1)}:=&\left\{\delta \overline{p_1}+\sum_{n=1}^\infty r_0(n)5^{\lfloor \frac{5n-1}{2} \rfloor}\overline{p_1}t^n+\sum_{n=1}^\infty r_1(n)5^{\lfloor \frac{5n-2}{2} \rfloor}t^n:\text{$r_i$ are discrete functions}, \delta \in \mathbb{Z} \right\}.
\end{align*}
Combining \eqref{eq:Kodddirect}, \eqref{eq:Kevendirect} and \eqref{eq:KinX}, we conclude the proof of \eqref{r:c33cong}.
\end{proof}
\begin{rema}
An alternative way to shift from \eqref{r:c31cong} to \eqref{r:c33cong} is firstly using a single Atkin-Lehner involution to shift from $C\Psi_{3,1/2}$ to a linear combination of $C\Psi_{3,1/2}(q^3)$ and $C\Psi_{3,3/2}(q^3)$, up to different factors of powers of $q$, then finding out congruences of this combination from \eqref{r:c31cong}, and finally disregarding the $C\Psi_{3,1/2}(q^3)$-part. The action of the Atkin-Lehner involution mentioned above follows from Theorem \ref{thm:fkbetaSLZ} and is given by
\begin{equation*}
q^{5/24}C\Psi_{3,1/2}\Big\vert_{-1/2}\tbtMat{0}{-1}{3}{0}=\left(\frac{\sqrt{6}}{6}+\frac{\sqrt{6}}{6}\rmi\right)\cdot\left(-q^{5/24}C\Psi_{3,1/2}(q^3)+q^{-1/8}C\Psi_{3,1/2}(q^3)\right).
\end{equation*}
This proof may be tidier, but the proof that we have presented seems to apply to more general cases due to Theorem \ref{thm:smallk}.
\end{rema}

\section{Miscellaneous observations and future works}
\label{sec:misc}

\subsection{On representations of $\sltZ$}
\label{subsec:representations}

Let $k\in\numgeq{Z}{1}$ and $\beta\in\mathfrak{B}_k$. The initial term of $f_{k,\beta}=q^{\frac{k}{12}-\frac{\beta^2}{2k}}\cdot C\Psi_{k,\beta}$ is $c\psi_{k,\beta}(0)q^{\frac{k}{12}-\frac{\beta^2}{2k}}$ where $c\psi_{k,\beta}(0)$ equals the $\zeta^{\beta-k/2}$-coefficient of $(1+\zeta^{-1})^k$. It follows that $\{f_{k,\beta}\colon\beta\in\mathfrak{B}_k\}$ is $\numC$-linearly independent. Therefore, if we define
$$
\rho_k\colon\sltZ\to\glnC{\lfloor k/2\rfloor+1}
$$
partially by (noting that $\delta_{\beta,\beta'}$ is the Kronecker's delta)
\begin{align*}
\rho_k\widetilde{\tbtMat{1}{1}{0}{1}}&=\left(\delta_{\beta,\beta'}\cdot\etp{\frac{k}{12}-\frac{\beta^2}{2k}}\right)_{\beta,\beta'\in\mathfrak{B}_k},\\
\rho_k\widetilde{\tbtMat{0}{-1}{1}{0}}&=\left(s_{\beta,\beta'}^{(k)}\right)_{\beta,\beta'\in\mathfrak{B}_k},
\end{align*}
then it actually extends to a representation of the metaplectic group $\sltZ$ in the space $\numC^{\lfloor k/2\rfloor+1}$ by Theorem \ref{thm:fkbetaSLZ}. Set $\mathbf{f}_k=(f_{k,\beta})_{\beta\in\mathfrak{B}_k}$, viewed as a column vector. Then Theorem \ref{thm:fkbetaSLZ} amounts to saying that
\begin{equation*}
\mathbf{f}_k\vert_{-1/2}\gamma=\rho_k(\gamma)\mathbf{f}_k,\quad\gamma\in\sltZ.
\end{equation*}
\begin{prop}
\label{prop:kerrhok}
The kernel of $\rho_k$ has finite index in $\sltZ$. Moreover, each matrix $\rho_k(\gamma)$ has finite order and so are its eigenvalues.
\end{prop}
\begin{proof}
For each matrix $\gamma\in\Gamma(24k)$, we have $f_{k,\beta}\vert_{-1/2}\widetilde{\gamma}=\varepsilon_\gamma\cdot f_{k,\beta}$ with $\varepsilon_\gamma\in\{\pm1\}$ independent of $\beta$ by Theorem \ref{thm:fkbetaGamma0k}. Therefore, $\rho_k(\gamma,\varepsilon_\gamma)$ is the identity matrix in $\glnC{\lfloor k/2\rfloor+1}$. Since the group generated by all $(\gamma,\varepsilon_\gamma)$ has index $2$ in $\widetilde{\Gamma(24k)}$, we have $[\sltZ\colon \ker \rho_k]\leq2[\sltZ\colon\widetilde{\Gamma(24k)}]<+\infty$. For the second conclusion, let $\gamma\in\sltZ$ be arbitrary. By what we have just proved, $\gamma^n\in\ker \rho_k$ for some $n\in\numgeq{Z}{1}$, so $\rho_k(\gamma)^n$ equals the identity matrix in $\glnC{\lfloor k/2\rfloor+1}$. The final assertion then follows from this.
\end{proof}

The following proposition\footnote{This is suggested by one of the anonymous referees. We are grateful for his or her insightful comments.} links $\rho_k$ to the Weil representation $\rho_{\underline{L}_k}$.
\begin{prop}
\label{prop:rhokrhoLk}
Let $k$ be a positive integer. Let $\gamma=\tbtmat{a}{b}{c}{d}\in\Gamma_0(2)$. Let $\beta,\beta'\in\mathfrak{B}_k$. If $k$ is even, then
\begin{equation}
\label{eq:rhokrhoLkeven}
\rho_k(\widetilde{\gamma})_{\beta,\beta'}=\rmi^{-\frac{kc}{2}}\chi_\eta^{2k}(\widetilde{\gamma})\mu_{\beta'/k}\cdot\left(\rho_{\underline{\numZ}_k}^*(\widetilde{\gamma})_{\beta/k,\beta'/k}+\rho_{\underline{\numZ}_k}^*(\widetilde{\gamma})_{\beta/k,-\beta'/k}\right),
\end{equation}
where $\mu_{\beta'/k}$ is defined in Theorem \ref{thm:fkbetaSLZ}, and $\chi_\eta$ is the character of the Dedekind eta function (see \eqref{eq:etaChar}). On the other hand, if $k$ is odd, then
\begin{multline}
\label{eq:rhokrhoLkodd}
\rho_k(\widetilde{\gamma})_{\beta,\beta'}=(-1)^{\frac{c(d-1)}{4}+\frac{c+d-1}{2}}\rmi^{-\frac{kc}{2}}\chi_\eta^{2k}(\widetilde{\gamma})\mu_{\beta'/k}\\
\cdot\left(\rho_{\underline{2\numZ}_k}^*(\widetilde{\gamma})_{\beta/k,\beta'/k}+\rho_{\underline{2\numZ}_k}^*(\widetilde{\gamma})_{\beta/k,-\beta'/k}+\rho_{\underline{2\numZ}_k}^*(\widetilde{\gamma})_{\beta/k,(k+\beta')/k}+\rho_{\underline{2\numZ}_k}^*(\widetilde{\gamma})_{\beta/k,-(k+\beta')/k}\right).
\end{multline}
\end{prop}
\begin{proof}
For $k$ even, the assertion follows by combining \eqref{eq:hbetakvertgamma2}, \eqref{eq:fkbeta}, and \eqref{eq:chikFormula}. For $k$ odd, we use \eqref{eq:hbetakvertgamma2Odd} instead of \eqref{eq:hbetakvertgamma2}.
\end{proof}
\begin{rema}
Numerical experiments show that \eqref{eq:rhokrhoLkeven} and \eqref{eq:rhokrhoLkodd} do not necessarily hold for $\gamma\in\slZ$. We do not know explicit formulas that link $\rho_k$ with $\rho_{\underline{L}_k}$ on the whole $\slZ$. Maybe one could acquire such formulas by modifying the coefficients a bit in \eqref{eq:rhokrhoLkeven} and \eqref{eq:rhokrhoLkodd}.
\end{rema}

We have defined an equivalence relation on $\mathfrak{B}_k$ according to whether there exists $\gamma\in\Gamma_0(k)$ such that we can transform one $f_{k,\beta}$ to another by $\gamma$; see Proposition \ref{prop:EquivalenceRelation}. It already leads to a new approach for establishing correspondence among congruence families; see Theorem \ref{thm:thklkl}. However, our correspondence between $f_{3,1/2}$ and $f_{3,3/2}$ does not rely on $\gamma\in\Gamma_0(3)$ since $1/2$ and $3/2$ are inequivalent, but relies on a linear combination of elements in $\sltZ$; see Example \ref{r:exm01}. Furthermore, this correspondence is powerful: it indeed let us find out a new congruence family on Andrews' $c\phi_{3}(n)$, see \eqref{r:c33cong}. Therefore, the following definition seems to be important as well:
\begin{deff}
Let $\beta,\beta'\in\mathfrak{B}_k$. We write $\beta\succeq\beta'$ (or more rigorous $\beta\succeq_k\beta'$) if there exists an $M\in\mathscr{A}_\numZ$ such that $f_{k,\beta}\vert_{-1/2}M=f_{k,\beta'}$. (For $\mathscr{A}_\numZ$ and its action, see Definitions \ref{def:groupAlgebra2} and \ref{def:groupAlgebra3}.)
\end{deff}
The proposition below follows directly from the expressions of $\rho_k\widetilde{\tbtmat{1}{1}{0}{1}}$, $\rho_k\widetilde{\tbtmat{0}{-1}{1}{0}}$ and basic linear algebra.
\begin{prop}
\label{prop:MinK}
If $f_{k,\beta}\vert_{-1/2}M=f_{k,\beta'}$, then we can require that $M\in K[\sltZ]$ where $K$, depending only on $k$, is an algebraic number field.
\end{prop}
Indeed, $K$ can always be chosen as the number field generated by $\etp{\frac{k}{12}-\frac{\beta^2}{2k}}$ and $s_{\beta,\beta'}^{(k)}$, $\beta,\beta'\in\mathfrak{B}_k$. Furthermore, $\rho_k$ can be seen as a representation over the base field $K$ just mentioned.

From the last proposition, we see that if $\beta\succeq\beta'$, then it is possible to transform a number-theoretical or combinatorial property, such as a congruence, from $C\Psi_{k,\beta}$ to $C\Psi_{k,\beta'}$ via $M\in\mathscr{A}_\numZ$ as we have done for $k=3$ in Section \ref{subsec:fromf312tof332}. The following problem is then fundamental:
\begin{oq}
\label{op:btobp}
Find out an explicit criterion on when $\beta\succeq\beta'$.
\end{oq}
The answer in the following case is simple.
\begin{prop}
\label{prop:rhokirredu}
If $\rho_k$ is irreducible, then $\beta\succeq\beta'$ for all $\beta,\beta'\in\mathfrak{B}_k$.
\end{prop}
\begin{proof}
Set $\mathfrak{F}_k=\bigoplus_{\beta\in\mathfrak{B}_k}\numC f_{k,\beta}$. The (left) group action $(\gamma,f)\mapsto f\vert_{-1/2}\gamma^{-1}$ makes $\mathfrak{F}_k$ to be a left $\sltZ$-module of $\numC$-dimension $\lfloor k/2\rfloor+1$. This $\sltZ$-module structure gives rise to a homomorphism $\rho_k'\colon\sltZ\to\mathrm{GL}(\mathfrak{F}_k)$, $\gamma\mapsto(f\mapsto f\vert_{-1/2}\gamma^{-1})$. Let $\mathcal{T}\colon\mathrm{GL}(\mathfrak{F}_k)\to\mathrm{GL}(\numC^{\lfloor k/2\rfloor+1})$ be the isomorphism that maps $g$ to its matrix representation under the basis $\{f_{k,\beta}\colon\beta\in\mathfrak{B}_k\}$. Then $\mathcal{T}\circ\rho_k'$ is a group representation and one can verify immediately that $\mathcal{T}\circ\rho_k'(\gamma)=\rho_k(\gamma)^{-T}$ where ``T'' means matrix transpose. As the two representations both factor through representations of the finite group $\sltZ/\ker \rho_k$ and the traces of $\rho_k(\gamma)^{-T}$ and $\rho_k(\gamma)$ are conjugate to each other, we deduce from the assumption $\rho_k$ is irreducible that $\rho_k'$ is irreducible as well (cf. \cite[p.17]{Ser77}), which is equivalent to saying $\mathfrak{F}_k$ is a simple $\sltZ$-module. Therefore, the submodule generated by $f_{k,\beta}$ is the whole module so $\beta\succeq\beta'$.
\end{proof}

To determine whether $\rho_k$ is irreducible (for any fixed $k$), we may use character theory of finite groups. Let $G$ be arbitrary finite-index normal subgroup of $\sltZ$ contained in $\ker\rho_k$, and let $\widetilde{\rho}_k\colon\sltZ/G\to\glnC{\lfloor k/2\rfloor+1}$ be the representation given by $gG\mapsto\rho_k(g)$. We define
\begin{equation}
\label{eq:defmk}
\mathfrak{m}_k:=\frac{1}{\abs{\sltZ/G}}\cdot\sum_{g\cdot G\in\sltZ/G}\abs{\mathop{\mathrm{Tr}}\rho_k(g)}^2.
\end{equation}
Alternatively, $\mathfrak{m}_k$ is the norm of the character of $\widetilde{\rho}_k$. Note that $\mathfrak{m}_k$ is independent of the choice of $G$. By the character theory, $\mathfrak{m}_k$ is a positive integer; moreover,
\begin{equation}
\label{eq:mk1}
\rho_k \text{ is irreducible if and only if } \mathfrak{m}_k=1.
\end{equation}
See \cite[p. 17]{Ser77}.

We now give a concrete formula to calculate $\mathfrak{m}_k$, which is not an efficient one\footnote{A more efficient method should involve Str\"{o}mberg's formula \cite[Remark 6.8]{Str13}, which we have not considered.} but is enough for small $k$ with the aid of any computer algebra system. First some notations. Set $T=\tbtmat{1}{1}{0}{1}$, $S=\tbtmat{0}{-1}{1}{0}$, $N=24k$. Let $A_d$ be a complete set of representatives of $\left(\numZ/(d,N/d)\numZ\right)^\times$ with the additional conditions $0\leq x<N/d$ and $(x,N/d)=1$ for all $x\in A_d$. Set $\mathscr{N}_d=\{dx\colon x\in A_d\}$ and $\mathscr{N}=\bigcup_{d\mid N,\,d\neq1,N}\mathscr{N}_d$, which is a disjoint union.
\begin{prop}
\label{prop:formulamk}
We have
\begin{multline*}
\mathfrak{m}_k\cdot N^3\prod_{p\mid N}\left(1-\frac{1}{p^2}\right)\\
=\sum_{j=0}^{N-1}\sum_{\twoscript{a=0}{(a,N)=1}}^{N-1}\left\lvert{\mathop{\mathrm{Tr}}\rho_k(\widetilde{T})^j\rho_k\widetilde{\tbtmat{a}{b_a}{N}{d_a}}}\right\rvert^2
+\sum_{j=0}^{N-1}\sum_{\twoscript{a=0}{(a,N)=1}}^{N-1}\sum_{i=0}^{N-1}\left\lvert{\mathop{\mathrm{Tr}}\rho_k(\widetilde{T})^j\rho_k\widetilde{\tbtmat{a}{b_a}{N}{d_a}}}\rho_k(\widetilde{S})\rho_k(\widetilde{T})^i\right\rvert^2\\
+\sum_{j=0}^{N-1}\sum_{\twoscript{a=0}{(a,N)=1}}^{N-1}\sum_{n\in\mathscr{N}}\sum_{i=0}^{N/(N,n^2)-1}\left\lvert{\mathop{\mathrm{Tr}}\rho_k(\widetilde{T})^j\rho_k\widetilde{\tbtmat{a}{b_a}{N}{d_a}}}\rho_k(\widetilde{S})\rho_k(\widetilde{T})^n\rho_k(\widetilde{S})\rho_k(\widetilde{T})^i\right\rvert^2
\end{multline*}
where $p$ denotes primes and $b_a$, $d_a$ are arbitrary integers such that $ad_a-Nb_a=1$.
\end{prop}
\begin{proof}[Sketch of proof]
Let $G_k$ be the index $2$ subgroup of $\widetilde{\Gamma(24k)}$ introduced in the proof of Proposition \ref{prop:kerrhok}. One can prove that $G_k$ is a normal subgroup of $\sltZ$. Setting $G=G_k$ in \eqref{eq:defmk} we deduce that
\begin{equation}
\label{eq:mkTemp}
\mathfrak{m}_k=\frac{1}{\abs{\slZ/\Gamma(24k)}}\cdot\sum_{g\in R}\abs{\mathop{\mathrm{Tr}}\rho_k(\widetilde{g})}^2
\end{equation}
where $R$ is a complete set of representatives of the orbit space $\Gamma(24k)\backslash\slZ$. Since $[\slZ\colon\Gamma(24k)]=N^3\prod_{p\mid N}\left(1-\frac{1}{p^2}\right)$ it remains to find a suitable $R$. We have $\Gamma_1(N)=\bigcup_{0\leq j<N}\Gamma(N)\cdot T^j$, $\Gamma_0(N)=\bigcup_{0\leq a<N,\,(a,N)=1}\Gamma_1(N)\cdot\tbtmat{a}{b_a}{N}{d_a}$, both of which are disjoint unions. Moreover, a slight modification of \cite[Proposition 2]{Las02} shows that $\slZ=\bigcup_{\gamma\in R_0}\Gamma_0(N)\cdot\gamma$, a disjoint union, where $R_0$ consists of the following matrices:
\begin{equation*}
I,\quad ST^i\,(0\leq i<N),\quad ST^nST^i\,(n\in\mathscr{N},\,0\leq i<N/(N,n^2)).
\end{equation*}
Combining these three decompositions we obtain a choice of $R$. Inserting it into \eqref{eq:mkTemp} and noting that $\rho_k(\gamma,1)=-\rho_k(\gamma,-1)$ we obtain the desired formula.
\end{proof}

The values of $\rho_k(\widetilde{T})$ and $\rho_k(\widetilde{S})$ in Proposition \ref{prop:formulamk} are known by the definition, whereas the matrices $\rho_k\widetilde{\tbtmat{a}{b_a}{N}{d_a}}$ can be obtained via Theorem \ref{thm:fkbetaGamma0k} since $\tbtmat{a}{b_a}{N}{d_a}\in\Gamma_0(k)$. We have worked out $\mathfrak{m}_k$ for small $k$ using Proposition \ref{prop:formulamk}; see Table \ref{table:mk}.
\begin{longtable}{llllllllll}
\caption{$\mathfrak{m}_k$ for small $k$\label{table:mk}} \\
\toprule
$k$ & $1$ & $2$ & $3$ & $4$ & $5$ & $6$ & $7$ & $8$ & $9$\\
\endfirsthead
$\mathfrak{m}_k$ & $1$ & $1$ & $1$ & $1$ & $1$ & $1$ & $1$ & $2$ & $2$\\
\midrule
$k$ & $10$ & $11$ & $12$ & $13$ & $14$ & $15$ & $16$ & $17$ & $18$\\
$\mathfrak{m}_k$ & $1$ & $1$ & $2$ & $1$ & $1$ & $2$ & $2$ & $1$ & $2$\\
\bottomrule
\end{longtable}

Skoruppa \cite[\S 6]{Sko08} provided the decomposition of each $\rho_{\underline{L}_k}$ into irreducible parts. It is possible to apply this theory to $\rho_k$ and $\mathfrak{m}_k$, provided that one has rigorously established the relationship between $\rho_k$ with $\rho_{\underline{L}_k}$. Up to now, we only know Proposition \ref{prop:rhokrhoLk}, which concerns only $\Gamma_0(2)$.

Combining Propositions \ref{prop:MinK}, \ref{prop:rhokirredu}, Table \ref{table:mk} and \eqref{eq:mk1} we obtain a partial answer to Question \ref{op:btobp}:
\begin{thm}
\label{thm:smallk}
Let $k\in\{1,2,3,4,5,6,7,10,11,13,14,17\}$ and $\beta,\beta'\in\mathfrak{B}_k$. Then there exists $M\in K[\sltZ]$ such that $f_{k,\beta}\vert_{-1/2}M=f_{k,\beta'}$, where $K$ is the number field generated by $\etp{\frac{k}{12}-\frac{b^2}{2k}}$ and $s_{b,b'}^{(k)}$, $b,b'\in\mathfrak{B}_k$.
\end{thm}

We state another two questions related to this theme.
\begin{oq}
In the case $\beta\succeq_k\beta'$, can we always find $M$ in the ring of integers of $K$ such that $f_{k,\beta}\vert_{-1/2}M=f_{k,\beta'}$? The $M_0$ in Example \ref{r:exm01} is of this kind.
\end{oq}
\begin{oq}
In the case $\beta\succeq_k\beta'$, find an explicit expression of or an algorithm (like Algorithm \ref{algo:findgaama}) to produce $M$.
\end{oq}

\subsection{More congruences for $c\psi_{k,\beta}(n)$}

There may be more examples analogous to Theorem \ref{r:themaincong}.

\begin{conj}
For integers $\alpha>0$, $\beta=0,1,2$, and $n$ such that $3\beta^2-8\equiv 24 n\pmod{7^{2\alpha-1}}$, we have
\beq
\label{r:ex4}
c\psi_{4,\beta}(n) \equiv 0 \pmod{7^\alpha}.
\eeq
\end{conj}

\begin{conj}
For integers $\alpha>0$, $\beta=0,1,2,3$, and $n$ such that $\beta^2-6 \equiv 12 n\pmod{7^{2\alpha-1}}$, we have
\beq
\label{r:ex6}
c\psi_{6,\beta}(49n+24-4\beta^2) \equiv x_\alpha \cdot c\psi_{6,\beta}(n) \pmod{7^{\alpha}},
\eeq
where $x_\alpha$ is a constant depending on $\alpha$.
\end{conj}

The main theme of this paper, on which all our results are centered, is to connect congruences between $c\psi_{k,\beta}(n)$ and $c\psi_{k,\beta'}(n)$. Nonetheless, our results do not directly prove any congruences, but provide machinery to convert from congruences for one $\beta$, which should be the simplest one to prove, to those for another $\beta'$, which may be difficult to prove by standard methods. The congruences \eqref{r:ex4} may be easier than \eqref{r:ex6} since the generating function of $c\psi_{4,\beta}$ appears simple for $\beta=1$. Thus, one could prove \eqref{r:ex4} for $\beta=1$ by standard procedures, for example using the method in Section \ref{subsec:sketch_r_c31cong}, and then prove \eqref{r:ex4} for $\beta=0,2$ using the method we present in Section \ref{subsec:fromf312tof332}. On the other hand, the congruences \eqref{r:ex6} appear more challenging since we do not have ideas for proving any of them. However, if one can prove \eqref{r:ex6} for one value of $\beta$ via modular techniques, it may be possible to transform this proof to other values of $\beta$ as in Section \ref{subsec:fromf312tof332}.

\begin{appendix}
\section{Modular relations associated with $C\Psi_{3,1/2}(q)$}
\label{appendix:modular_relations}

In this appendix, we list all modular relations that are used in the proof of \eqref{r:c31cong}. See Section \ref{subsec:sketch_r_c31cong}.

\begin{align*}
&\text{Group I:}\\
&\qquad U^{(0)}(1)=5^7 t^3+9\cdot 5^4 t^2+9 \cdot 5 t+(5^5 t^2+8 \cdot 5^2 t+1)p_1,\\
&\qquad U^{(0)}(t^{-1})=-5^2 t-2p_1,\\
&\qquad U^{(0)}(t^{-2})=5^5 t^2+6\cdot 5^2 t+(2\cdot 5^3 t+4 \cdot 5)p_1,\\
&\qquad U^{(0)}(t^{-3})=-9\cdot 5^5 t^2-9\cdot 5^3 t+1+(5^6 t^2-2\cdot 5^4 t-37 \cdot 5)p_1,\\
&\qquad U^{(0)}(t^{-4})=5^{11} t^4+3\cdot 5^9 t^3+9\cdot 5^6 t^2+79 \cdot 5^3 t-3\cdot 5-(5^8 t^2+12\cdot 5^4 t-67 \cdot 5^2)p_1.\\
&\text{Group II:}\\
&\qquad U^{(0)}(p_0)=(5^9 t^4+5^8 t^3+8 \cdot 5^5 t^2+4 \cdot 5^3+1)p_1,\\
&\qquad U^{(0)}(p_0t^{-1})=-p_1,\\
&\qquad U^{(0)}(p_0t^{-2})=(5^3 t+14)p_1,\\
&\qquad U^{(0)}(p_0t^{-3})=(5^6 t^2-26 \cdot 5)p_1,\\
&\qquad U^{(0)}(p_0t^{-4})=-(5^9 t^3+38\cdot 5^6 t^2+38 \cdot 5^4 t-228 \cdot 5)p_1.\\
&\text{Group III:}\\
&\qquad U^{(1)}(1)=1,\\
&\qquad U^{(1)}(t^{-1})=-5^2t-6,\\
&\qquad U^{(1)}(t^{-2})=-5^5t^2+54,\\
&\qquad U^{(1)}(t^{-3})=-5^8t^3-102\cdot 5,\\
&\qquad U^{(1)}(t^{-4})=-5^{11}t^4+966\cdot 5.\\
&\text{Group IV:}\\
&\qquad U^{(1)}(p_1)=-18\cdot 5^7 t^3-234\cdot 5^4 t^2-126 \cdot 5^2 t+(5^9t^3+14 \cdot5^6 t^2+44\cdot 5^3 t+2\cdot 5)p_0,\\
&\qquad U^{(1)}(p_1t^{-1})=-18+5p_0,\\
&\qquad U^{(1)}(p_1t^{-2})=126+5^4 t p_0,\\
&\qquad U^{(1)}(p_1t^{-3})=-234\cdot 5+5^7 t^2p_0,\\
&\qquad U^{(1)}(p_1t^{-4})=18\cdot 5^9 t^3+234\cdot 5^6 t^2+126\cdot 5^4 t+2268\cdot 5\\
&\qquad\qquad\qquad\qquad-(4\cdot 5^{10}t^3+14\cdot 5^8 t^2+44\cdot 5^5 t+2\cdot 5^3-t^{-1})p_0.
\end{align*}

\section{Metaplectic covers}
\label{appendix:metaplectic}
Since the metaplectic covers (i.e., double covers) and elements like $\widetilde{\tbtmat{a}{b}{c}{d}}$ are used throughout the paper, it is worthwhile to recall the definitions and basic properties. To our best knowledge, the idea of using metaplectic covers to study modular forms of half-integral weights was first proposed by Shimura \cite[p. 443]{Shi73}. For more details on this topic, see e.g. \cite[p. 15]{Bru02}, \cite[\S 4]{Str13}, or \cite[\S 1.1.6]{CS17}.

Let $\glpR$ be the group of $2\times2$ real matrices of positive determinant. For a complex number $z\neq0$, we define $\sqrt{z}:=\exp(\frac{1}{2}\log z)$ with $-\uppi\rmi<\Im\log z\leq\uppi\rmi$. As a consequence, we have $-\frac{\uppi}{2}<\arg\sqrt{c\tau+d}\leq\frac{\uppi}{2}$, provided that $c\tau+d\neq0$.

\begin{deff}
The \emph{metaplectic cover} of $\glpR$, denoted by $\glptR$, is the set of elements $\left(\tbtmat{a}{b}{c}{d},\varepsilon\right)$, where $\tbtmat{a}{b}{c}{d}\in\glpR$ and $\varepsilon\in\{\pm1\}$. Moreover, for $\tbtmat{a}{b}{c}{d}\in\glpR$, we set $\widetilde{\tbtmat{a}{b}{c}{d}}=\left(\tbtmat{a}{b}{c}{d},1\right)$.
\end{deff}

There is a trivial composition on the set $\glptR$, which makes it into a group isomorphic to the direct product $\glpR\times\{\pm1\}$. However, the useful composition is the following one, which makes $\glptR$ into a nontrivial central extension of $\glpR$.
\begin{prop}
The composition
\begin{align*}
\glptR\times\glptR&\rightarrow\glptR\\
\left(\tbtmat{a_1}{b_1}{c_1}{d_1},\varepsilon_1\right),\left(\tbtmat{a_2}{b_2}{c_2}{d_2},\varepsilon_2\right)&\mapsto\left(\tbtmat{a_3}{b_3}{c_3}{d_3},\varepsilon_3\right),
\end{align*}
where $\tbtmat{a_3}{b_3}{c_3}{d_3}=\tbtmat{a_1}{b_1}{c_1}{d_1}\tbtmat{a_2}{b_2}{c_2}{d_2}$, and where
\begin{equation}
\label{eq:epsilon3formula}
\varepsilon_3=\varepsilon_1\varepsilon_2\cdot\frac{\sqrt{c_1(a_2\tau+b_2)/(c_2\tau+d_2)+d_1}\sqrt{c_2\tau+d_2}}{\sqrt{c_1(a_2\tau+b_2)+d_1(c_2\tau+d_2)}},
\end{equation}
makes $\glptR$ into a group. Note that in the right-hand side of \eqref{eq:epsilon3formula}, $\tau\in\uhp$ is chosen arbitrarily, and the value of $\varepsilon_3\in\{\pm1\}$ is independent of the choice of $\tau$.
\end{prop}
\begin{proof}
A simplification shows that the fraction in \eqref{eq:epsilon3formula} takes values in the discrete space $\{\pm1\}$. Since $\uhp$ is connected, and since this fraction is a holomorphic function of $\tau\in\uhp$, it is a constant function, namely, it is independent of the choice of $\tau$. The fact the composition given above makes $\glptR$ into a group follows from a straightforward but tedious verification of the three group axioms.
\end{proof}

It is inconvenient to calculate $\varepsilon_3$ using the right-hand side of \eqref{eq:epsilon3formula}. Instead, we use the following formula.
\begin{prop}
\label{prop:cocycleFormula}
Let $\sigma\left(\tbtmat{a_1}{b_1}{c_1}{d_1},\tbtmat{a_2}{b_2}{c_2}{d_2}\right)$ be the fraction in \eqref{eq:epsilon3formula} and let $\tbtmat{a_3}{b_3}{c_3}{d_3}=\tbtmat{a_1}{b_1}{c_1}{d_1}\tbtmat{a_2}{b_2}{c_2}{d_2}$. Then we have
\begin{enumerate}
  \item $\sigma\left(\tbtmat{a_1}{b_1}{c_1}{d_1},\tbtmat{a_2}{b_2}{c_2}{d_2}\right)=-1$ if one of the following three conditions holds:
  \begin{itemize}
    \item $c_1=c_2=0$ and $d_1<0,\,d_2<0$,
    \item $c_1 \geq 0,\, c_2 \geq 0$ but $c_3 < 0$,
    \item $c_1 < 0$, $c_2 < 0$ but $c_3 \geq 0$.
   \end{itemize}
  \item $\sigma\left(\tbtmat{a_1}{b_1}{c_1}{d_1},\tbtmat{a_2}{b_2}{c_2}{d_2}\right)=1$ otherwise.
\end{enumerate}
\end{prop}
\begin{proof}
Set $z_1=c_1(a_2\tau+b_2)/(c_2\tau+d_2)+d_1$ and $z_2=c_2\tau+d_2$, then $z_1z_2=c_3\tau+d_3$, and
$$\sigma\left(\tbtmat{a_1}{b_1}{c_1}{d_1},\tbtmat{a_2}{b_2}{c_2}{d_2}\right)=\frac{\sqrt{z_1}\sqrt{z_2}}{\sqrt{z_1z_2}}=\exp\left(\frac{\log z_1+\log z_2 - \log(z_1z_2)}{2}\right).$$
If $c_1 < 0$, $c_2 < 0$ but $c_3 \geq 0$, then $\Im z_1<0$, $\Im z_2<0$, and $\Im(z_1z_2)\geq0$. It follows that $\log z_1+\log z_2 - \log(z_1z_2)=\rmi(\arg(z_1)+\arg(z_2)-\arg(z_1z_2))=-2\uppi\rmi$. Hence $\sigma\left(\tbtmat{a_1}{b_1}{c_1}{d_1},\tbtmat{a_2}{b_2}{c_2}{d_2}\right)=-1$ as desired. The proofs for other cases are similar.
\end{proof}
\begin{rema}
There is a formula expressing $\sigma\left(\tbtmat{a_1}{b_1}{c_1}{d_1},\tbtmat{a_2}{b_2}{c_2}{d_2}\right)$ using the Hilbert symbol. See e.g. \cite[Theorem 4.1]{Str13}.
\end{rema}

\begin{examp}
We have $\left(\tbtmat{-1}{0}{1}{-1},1\right)\cdot\left(\tbtmat{1}{0}{2}{1},1\right)=\left(\tbtmat{-1}{0}{-1}{-1},-1\right)$.
\end{examp}

\begin{deff}
\label{def:slashaction}
Let $f\colon\uhp\rightarrow\numC$ be a function defined on $\uhp$. Let $\left(\tbtmat{a}{b}{c}{d},\varepsilon\right)\in\glptR$. Let $k\in\frac{1}{2}\numZ$. We set
\begin{equation*}
f\vert_k\left(\tbtmat{a}{b}{c}{d},\varepsilon\right)(\tau):=\left(\varepsilon\sqrt{c'\tau+d'}\right)^{-2k}f\left(\frac{a\tau+b}{c\tau+d}\right),
\end{equation*}
where $\tbtmat{a'}{b'}{c'}{d'}\in\slR$ is the unique modular matrix proportional to $\tbtmat{a}{b}{c}{d}$, that is, $\tbtmat{a'}{b'}{c'}{d'}=(ad-bc)^{-1/2}\tbtmat{a}{b}{c}{d}$. 
\end{deff}

The following facts illustrate the motivation behind \eqref{eq:epsilon3formula}, and the main advantage of using the metaplectic cover.
\begin{prop}
Let $k\in\frac{1}{2}\numZ$. We have $f(\vert_k\gamma_1)\vert_k\gamma_2=f\vert_k(\gamma_1\gamma_2)$ for $\gamma_1,\gamma_2\in\glptR$. Moreover, we have $f\vert_k(I,1)=f$ where $I$ is the identity matrix. ($(I,1)$ is the group unit of $\glptR$).
\end{prop}
\begin{proof}
A straightforward verification by expanding the definitions.
\end{proof}

\begin{rema}
We can also define, for $\tbtmat{a}{b}{c}{d}\in\glpR$, $f\vert_k\tbtmat{a}{b}{c}{d}(\tau):=(c'\tau+d')^{-k}f\left(\frac{a\tau+b}{c\tau+d}\right)$. Then for $k$ not an integer, we only have $(f\vert_{k}\tbtmat{a_1}{b_1}{c_1}{d_1})\vert_{k}\tbtmat{a_2}{b_2}{c_2}{d_2}=\pm f\vert_{k}\left(\tbtmat{a_1}{b_1}{c_1}{d_1}\tbtmat{a_2}{b_2}{c_2}{d_2}\right)$, with an annoying factor $\pm1$. For example,
\begin{equation*}
\left(f\vert_{1/2}\tbtmat{-1}{0}{1}{-1}\right)\vert_{1/2}\tbtmat{1}{0}{2}{1}=-f\vert_{1/2}\tbtmat{-1}{0}{-1}{-1}.
\end{equation*}
Of course, one can work with $\glpR$ instead of $\glptR$, always paying attention to the factors $\pm1$ which occur. In this way, if we have a modular form $f$ with $f\vert_{1/2}\gamma=\chi(\gamma)f$ for a group of $\gamma$s, then $\chi$ is a so-called projective character. If we work with $\glptR$, then $\chi$ will become a true group character on the double cover, which is one of the reasons why we use metaplectic covers. However, the main reason, which will be stated at the end of this appendix, is that we need an action of linear combinations of modular transformations on modular forms of half-integral weights, in which situation only considering $\glpR$ is problematic.
\end{rema}

\begin{prop}
\label{prop:widetildeG}
Let $G$ be a subgroup of $\glpR$, and set $\widetilde{G}=\left\{\left(\gamma,\varepsilon\right)\colon\gamma\in G,\varepsilon\in\{\pm1\}\right\}$. Then $\widetilde{G}$ is a subgroup of $\glptR$.
\end{prop}
\begin{proof}
This is immediate.
\end{proof}

\begin{deff}
\label{def:groupAlgebra1}
The group algebra $\mathscr{A}:=\numC[\glptR]$ is the set of finite formal sums
$$\sum_{\gamma\in\glptR}c_\gamma\cdot\gamma$$
with $c_\gamma\in\numC$. Equivalently, the above element can be understood as a function from $\glptR$ to $\numC$ with a finite support. Moreover, the element $1\cdot\gamma\in\mathscr{A}$ is identified with $\gamma\in\glptR$.
\end{deff}
\begin{examp}
We have $3\cdot\left(\tbtmat{-1}{0}{1}{-1},1\right)-2\cdot\left(\tbtmat{1}{0}{2}{1},1\right)\in\mathscr{A}$. It is important to note that this is just a \emph{formal} sum.
\end{examp}

\begin{prop}
Under the usual addition, and under the multiplication define by $\left(\sum c_\gamma\cdot\gamma\right)\cdot\left(\sum d_\gamma\cdot\gamma\right)=\sum_{\gamma}\sum_{\gamma_1\gamma_2=\gamma}(c_{\gamma_1}d_{\gamma_2})\cdot\gamma$, $\mathscr{A}$ is a unital associative $\numC$-algebra.
\end{prop}
\begin{proof}
See \cite[p. 127, Ex. 8]{Jac85}.
\end{proof}

\begin{deff}
\label{def:groupAlgebra2}
We define $\mathscr{A}_\numZ:=\numC[\sltZ]$ and $\mathscr{A}_m:=\numC[\widetilde{\Gamma_0(m)}]$, where $m\in\numgeq{Z}{1}$. They are subalgebras of $\mathscr{A}$.
\end{deff}

\begin{deff}
\label{def:groupAlgebra3}
Let $f\colon\uhp\rightarrow\numC$ be a function defined on $\uhp$. Let $M=\sum_{j=1}^{r}c_j\cdot\left(\tbtmat{a_j}{b_j}{c_j}{d_j},\varepsilon_j\right)\in\mathscr{A}$. Let $k\in\frac{1}{2}\numZ$. We define
\begin{equation*}
f\vert_kM(\tau):=\sum_{j=1}^rc_j\cdot\left(\varepsilon_j\sqrt{c_j'\tau+d_j'}\right)^{-2k}\cdot f\left(\frac{a_j\tau+b_j}{c_j\tau+d_j}\right),
\end{equation*}
where $\tbtmat{a_j'}{b_j'}{c_j'}{d_j'}=(a_jd_j-b_jc_j)^{-1/2}\tbtmat{a_j}{b_j}{c_j}{d_j}$.
\end{deff}

This is compatible with Definition \ref{def:slashaction}, and makes any particular class of functions, say the space of holomorphic functions, a right $\mathscr{A}$-module:
\begin{prop}
\label{prop:mathscrARightAction}
Let $M_1,M_2\in\mathscr{A}$ and let $k\in\frac{1}{2}\numZ$. Then $(f\vert_kM_1)\vert_kM_2=f\vert_k(M_1M_2)$ and $f\vert_k\widetilde{I}=f$.
\end{prop}
\begin{proof}
A straightforward verification by expanding the definitions.
\end{proof}

\begin{rema}
The formula $(f_1\vert_{k_1}M)\cdot(f_2\vert_{k_2}M)=(f_1f_2)\vert_{k_1+k_2}M$ holds for $M\in\glptR$. However, it does \emph{not} necessarily hold for $M\in\mathscr{A}$.
\end{rema}

To conclude this appendix, we explain the main reason to use metaplectic covers. If one uses the matrix group algebra $\numC[\slZ]$ instead, then at each step when two elements act on a half-integral weight modular form successively and one wants to first multiply the two elements, many cocycle factors $\pm1$ must be tracked, which seems to be a drawback. See the next example. The use of $\numC[\sltZ]$ absorbs these cocycle factors into the algebra itself.

\begin{examp}
We consider the formal sum $\tbtmat{-1}{0}{1}{-1}\boxplus\tbtmat{1}{0}{1}{1}\in\numC[\slZ]$ (i.e., the function that sends $\tbtmat{-1}{0}{1}{-1}$ and $\tbtmat{1}{0}{1}{1}$ to $1$ and takes the value $0$ at other matrices). We use $\boxplus$ to denote the addition on $\numC[\slZ]$ to distinguish it from the usual matrix addition. In the group algebra $\numC[\slZ]$ we have
$$\left(\tbtmat{-1}{0}{1}{-1}\boxplus\tbtmat{1}{0}{1}{1}\right)\cdot\tbtmat{1}{0}{2}{1}=\tbtmat{-1}{0}{-1}{-1}\boxplus\tbtmat{1}{0}{3}{1}.$$
Now let $f$ be a function define on $\uhp$, and set $M_1=\left(\tbtmat{-1}{0}{1}{-1}\boxplus\tbtmat{1}{0}{1}{1}\right)$ and $M_2=\tbtmat{1}{0}{2}{1}$. We have
\begin{align*}
(f\vert_{1/2}M_1)\vert_{1/2}M_2&=-f\vert_{1/2}\tbtmat{-1}{0}{-1}{-1}+f\vert_{1/2}\tbtmat{1}{0}{3}{1},\\
f\vert_{1/2}(M_1M_2)&=f\vert_{1/2}\tbtmat{-1}{0}{-1}{-1}+f\vert_{1/2}\tbtmat{1}{0}{3}{1}.
\end{align*}
They are not proportional. This means the multiplication on $\numC[\slZ]$ is not compatible with the action of $\numC[\slZ]$ on modular forms of half-integral weights. However, there is no problem with $\numC[\sltZ]$ as Proposition \ref{prop:mathscrARightAction} shows.
\end{examp}

\end{appendix}

\section*{Acknowledgement}
We thank the anonymous referees for their constructive comments, which greatly improve this manuscript.

\bibliographystyle{plain}
\bibliography{ref}

\end{document}